\documentclass[12pt]{amsart}
\usepackage{amsmath}
\usepackage{geometry,amsthm,graphics,tabularx,amssymb,shapepar,eucal}
\usepackage{latexsym,amsfonts,amssymb,amsmath,amsxtra,bbm,color,mathrsfs}
\usepackage{amscd}
\usepackage{hyperref}
\usepackage{pdfsync}
\usepackage[all]{xy}
\usepackage{verbatim}
\usepackage{tikz-cd}

\usepackage{braket}

\newcommand{\bsl}{\backslash}

\newcommand{\BA}{{\mathbb {A}}}

\newcommand{\BC}{{\mathbb {C}}}

\newcommand{\BQ}{{\mathbb {Q}}}
\newcommand{\BR}{{\mathbb {R}}}

\newcommand{\BZ}{{\mathbb {Z}}}

\newcommand{\CB}{{\mathcal {B}}}

\newcommand{\CE}{{\mathcal {E}}}
\newcommand{\CF}{{\mathcal {F}}}
\newcommand{\CG}{{\mathcal {G}}}
\newcommand{\CH}{{\mathcal {H}}}

\newcommand{\CL}{{\mathcal {L}}}
\newcommand{\CM}{{\mathcal {M}}}

\newcommand{\CO}{{\mathcal {O}}}

\newcommand{\CS}{{\mathcal {S}}}

\newcommand{\CW}{{\mathcal {W}}}
\newcommand{\CX}{{\mathcal {X}}}

\newcommand{\RG}{{\mathrm {G}}}
\newcommand{\RH}{{\mathrm {H}}}

\newcommand{\RO}{{\mathrm {O}}}

\newcommand{\RS}{{\mathrm {S}}}

\newcommand{\RU}{{\mathrm {U}}}

\newcommand{\Ad}{{\mathrm{Ad}}}
\newcommand{\Aut}{{\mathrm{Aut}}}

\newcommand{\GL}{{\mathrm{GL}}}

\newcommand{\Hom}{{\mathrm{Hom}}}

\newcommand{\Ind}{{\mathrm{Ind}}}

\newcommand{\Ker}{{\mathrm{Ker}}}

\newcommand{\ord}{{\mathrm{ord}}}

\newcommand{\rk}{{\mathrm{k}}}

\newcommand{\bs}{\backslash}

\newcommand{\temp}{\operatorname{temp}}

\newcommand{\sgn}{\operatorname{sgn}}

\newcommand{\od}{\operatorname{d}}

\newcommand{\oL}{\operatorname{L}}

\newcommand{\oH}{\operatorname{H}}

\newcommand{\oZ}{\operatorname{Z}}

\newcommand{\g}{\mathfrak g}

\renewcommand{\k}{\mathfrak k}
\newcommand{\h}{\mathfrak h}

\newcommand{\q}{\mathfrak q}
\renewcommand{\a}{\mathfrak a}
\renewcommand{\b}{\mathfrak b}

\renewcommand{\u}{\mathfrak u}

\renewcommand{\l}{\mathfrak l}
\renewcommand{\t}{\mathfrak t}
\newcommand{\s}{\mathfrak s}

\newcommand{\R}{\mathbb R}

\newcommand{\A}{\mathbb{A}}

\newcommand{\abs}[1]{\lvert#1\rvert}

\newcommand{\la}{\langle}
\newcommand{\ra}{\rangle}

\newcommand{\be}{\begin {equation}}
\newcommand{\ee}{\end {equation}}
\newcommand{\bee}{\begin {equation*}}
\newcommand{\eee}{\end {equation*}}

\newcommand{\cf}{\emph{cf.}~}

\theoremstyle{Theorem}

\theoremstyle{Theorem}

\theoremstyle{Theorem}
\newtheorem{lem}{Lemma}[section]
\newtheorem{corl}[lem]{Corollary}
\newtheorem{thml}[lem]{Theorem}
\newtheorem{leml}[lem]{Lemma}
\newtheorem{prpl}[lem]{Proposition}

\theoremstyle{Theorem}

\theoremstyle{Plain}

\newtheorem{remarkl}[lem]{Remark}

\theoremstyle{remark}

\theoremstyle{remark}

\theoremstyle{Definition}

\newtheorem{dfnl}[lem]{Definition}

\newtheorem{cond}[lem]{Condition}

\newcommand{\Irr}{\mathrm{Irr}}
\newcommand{\bdl}{\boldsymbol}

\newcommand{\bk}{\mathbbm{k}}

\numberwithin{equation}{section}

\newcommand{\eps}{\epsilon}

\begin{document}

\title[Blasius-Deligne conjecture]
{On the Blasius-Deligne conjecture for the standard $L$-functions of symplectic type} 

\author[D. Jiang]{Dihua Jiang}
\address{School of Mathematics, University of Minnesota, Minneapolis, MN 55455, USA}
\email{dhjiang@math.umn.edu}

\author[D. Liu]{Dongwen Liu}
\address{School of Mathematical Sciences, Zhejiang University, Hangzhou, 310058, P. R. China}
\email{maliu@zju.edu.cn}

\author[B. Sun]{Binyong Sun}
\address{Institute for Advanced Study in Mathematics and New Cornerstone Science Laboratory, Zhejiang University, Hangzhou, 310058, P. R. China}
\email{sunbinyong@zju.edu.cn}

\author[F. Tian]{Fangyang Tian}
\address{School of Mathematical Sciences, Zhejiang University, Hangzhou, 310058, P. R. China}
\email{tianfangyangmath@zju.edu.cn}

\date{\today}


\begin{abstract}
  The Blasius-Deligne conjecture for automorphic $L$-functions is about the period relations and the algebraicity of critical $L$-values.  In this paper, we prove the Blasius-Deligne conjecture for  standard $L$-functions attached to regular algebraic cuspidal automorphic representations of $\GL_{2n}$ of symplectic type, 
    under the expected Galois invariance condition for these representations. 
    The main results extend the classical work of G. Shimura (\cite{Shi76}) and
Y. Manin (\cite{Mn76}) from $\GL_2$ to $\GL_{2n}$, and complete the approach that was initially taken by H. Grobner and A. Raghuram in \cite{GR14}  and by F. Januszewski in \cite{Jan16}, where a few technical assumptions were taken. 

\end{abstract}

\maketitle

\tableofcontents

\section{Introduction}

The Blasius-Deligne conjecture (\cite{D79,B97}) for automorphic $L$-functions is about the period relations and the algebraicity of critical $L$-values. In this paper, we give an unconditional proof of the Blasius-Deligne conjecture for the $\GL_{2n}$-standard $L$-functions of symplectic type with $n\geq 1$, at least when the base field is not totally complex. 

Let  $\rk$ be a number field with its adelic ring $\BA$. Let $\rk_v$ be the local field at a local place $v$ of $\rk$. Write $\BA=\BA_f\times \rk_\infty$,  with $\BA_f = \prod'_{v\nmid \infty}\rk_v$ being the finite part,  and 
$\rk_\infty$ being the 
infinity part which has the following realization:
\[
   \rk_\infty:=\rk\otimes_\BQ \BR 
  = \prod_{v\mid \infty} \rk_v \hookrightarrow \rk\otimes_\BQ \BC 
  = \prod_{\iota\in \CE_\rk} \BC,   
\]
where $\CE_\rk$ is the set of field embeddings $\iota: \rk\hookrightarrow\BC$. 

Let $\Pi$ be an   irreducible cuspidal automorphic representation of $\GL_{2n}(\BA)$ ($n\geq 1$) that is regular algebraic in the sense of \cite{Cl90}, realized as a space of smooth automorphic forms on $\GL_{2n}(\rk)\backslash \GL_{2n}(\BA)$ (see \cite{Gro23}). Write 
\[
\Pi =\widehat \otimes_v'\Pi_v:=\Pi_f\otimes \Pi_\infty,\quad \Pi_f := \otimes'_{v\nmid\infty}\Pi_v, 
\quad \Pi_\infty:=\widehat \otimes_{v\mid\infty}\Pi_v,  
\]
where ``$\widehat\otimes$" denotes the completed inductive tensor product, ``$\otimes'$" denotes the restricted tensor product, $\Pi_v$ denotes the $v$-component of $\Pi$. Similar notation with subscript ``$v$", ``$f$" or ``$\infty$" will be used without further explanation.  Up to isomorphism there is a unique irreducible algebraic representation 
$F_\mu$ of $\GL_{2n}(\rk\otimes_\BQ\BC)$ with highest weight $\mu = \{\mu^\iota\}_{\iota\in \CE_\rk} \in (\BZ^{2n})^{\CE_\rk}$ such that the total continuous cohomology 
\be \label{coh}
\oH^*_{\rm ct}(\BR^\times_+\bsl \GL_{2n}(\rk_\infty)^0; \Pi_\infty \otimes F_\mu^\vee)\neq \{0\},
\ee
where  the group  $\BR^\times_+$ of positive real numbers is viewed as a subgroup of $\GL_{2n}(\rk_\infty)^0$ via the diagonal embedding. Here and henceforth, a superscript ${}^\vee$ indicates the contragredient representation (when such a notion is defined) and $X^0$ denotes the identity connected component of a topological group $X$. 
The representation $F_\mu$ is called the coefficient system of $\Pi$. For  every $\sigma\in \Aut(\BC)$, denote by ${^\sigma}\Pi$ the $\sigma$-twist of $\Pi$ in the sense of \cite{Cl90}, which is also  a  regular algebraic irreducible cuspidal automorphic representation of $\GL_{2n}(\BA)$. Write 
\be\label{simgamu0}{}^\sigma\!\mu:=\{\mu^{\sigma^{-1}\circ\iota}\}_{\iota\in \CE_\rk}\in (\BZ^{2n})^{\CE_\rk}.
\ee
Then $ F_{^\sigma\!\mu}$ (the irreducible algebraic representation with highest weight $^\sigma\!\mu$) is   the coefficient system of ${}^\sigma\Pi$.   

We identify  $\BZ^{\CE_\rk}$  with the group of algebraic characters of $(\rk\otimes_\BQ\BC)^\times$ such that 
$\nu=\{\nu^\iota\}_{\iota\in \CE_\rk}\in \BZ^{\CE_\rk}$ is identified with the character $\otimes_{\iota\in \CE_\rk} \iota^{\nu^{\iota}}$. Here and henceforth by abuse of notation we write $\iota:(\rk\otimes_\BQ\BC)^\times\rightarrow \BC^\times$ for the projection to the $\iota$-component. 

Assume that $\Pi$ is of symplectic type (Definition \ref{def: global repn of symplectic type}). By Proposition \ref{prop: strong transfer}, this is equivalent to that
there is a Hecke character 
$\bdl{\eta}: \rk^\times \bsl \BA^\times \to\BC^\times$ such that $\Pi$ has central character $\bdl{\eta}^n$ and  
\be\label{pole0}
\textrm{the exterior square $L$-function $\oL(s, \Pi, \wedge^2\otimes\bdl{\eta}^{-1})$ has a pole at $s=1$.}
\ee
In this case, the highest weight $\mu = \{\mu^\iota\}_{\iota\in \CE_\rk}$  
 has the following property: for each $\iota\in \CE_\rk$, there exists $w_\iota\in\BZ$ such that 
\be\label{wiota}
\mu^\iota_1+\mu^\iota_{2n} = \mu^\iota_2+\mu^\iota_{2n-1} = \cdots = \mu^\iota_n +\mu^\iota_{n+1} =w_\iota,
\ee 
where we write  $\mu^\iota = (\mu^\iota_1, \mu^\iota_2, \dots, \mu^\iota_{2n})\in \BZ^{2n}$.  By Lemma \ref{etamu}, the infinite part of $\bdl \eta$ is 
\[
   \bdl \eta_\infty=\left(\otimes_{\iota\in \CE_\rk} \iota^{w_\iota}\right)|_{\rk_\infty^\times}.
\]
 In particular, the Hecke character $\bdl \eta$ is algebraic. 

 View $\RH:=\GL_n\times \GL_n$ as a standard Levi subgroup of $\GL_{2n}$, and for each $\nu=\{\nu^\iota\}_{\iota\in \CE_\rk}\in \BZ^{\CE_\rk}$ define an algebraic character 
\be \label{char:xi}
\xi_{\mu, \nu}:= \otimes_{\iota\in\CE_\rk} ( {\det}^{\nu^{\iota}} \boxtimes  {\det}^{-\nu^{\iota}-w_\iota}):\RH(\rk\otimes_\BQ\BC)\rightarrow \BC^\times. 
\ee

\begin{dfnl} \label{def:bal}
With the above notation, we say that $\nu$ is $F_\mu$-balanced if 
\[
\Hom_{\RH(\rk\otimes_\BQ\BC)}( F_\mu^\vee\otimes \xi_{\mu, \nu}^\vee, \BC)\neq \{0\},
\]
which is equivalent to that
\be \label{bal}
-\mu^\iota_n \leq \nu^\iota \leq -\mu^\iota_{n+1}\quad \text{for all $\iota\in \CE_\rk$.}
\ee
\end{dfnl}


Let  $\chi=\chi_f\otimes \chi_\infty:\rk^\times \bsl \BA^\times \to \BC^\times$ be an algebraic Hecke character, where $\chi_f$ is the finite part and $\chi_\infty$ is the infinite part. Write $\chi_\natural$ for the coefficient system of $\chi$, namely the algebraic character of  $(\rk\otimes_\BQ\BC)^\times$ whose differential equals the complexified differential of $\chi_\infty$. Then 
\be \label{inftype}
\chi_\infty =  \chi_\natural \vert_{\rk_\infty^\times} \cdot \chi^\natural \quad\textrm{for a  unique quadratic character $\chi^\natural$ of $\rk_\infty^\times$}.
\ee
For each $\sigma\in \Aut(\BC)$, denote by $^{\sigma\!} \chi$ the $\sigma$-twist of $\chi$, which is an algebraic Hecke character of $\rk^\times \bsl \BA^\times $ with coefficient system $^{\sigma\!} (\chi_\natural)$ (defined as in \eqref{simgamu0}). By \cite[Remark 6.3]{LLS24}, $(^{\sigma\!} \chi)^\natural=\chi^\natural$. 



Some remarks are in order. 
\begin{remarkl} If  $\chi_\natural$ is $F_\mu$-balanced, then the integers $j$ such that $\chi_\natural\cdot\otimes_{\iota\in\CE_\rk}\iota^j$ is $F_\mu$-balanced are in bijection with the critical places $\frac{1}{2}+j$ of $\oL(s, \Pi\otimes\chi)$. 
\end{remarkl}
  
    \begin{remarkl} Let  $\sigma\in \Aut(\BC)$. Then 
     \[
  \textrm{$\frac{1}{2}$ is a critical place of $\oL(s, \Pi\otimes\chi)$}\ \Leftrightarrow  \textrm{ $\frac{1}{2}$ is a critical place of $\oL(s, ^{\sigma\!}\Pi\otimes ^{\sigma\!\!\!}\chi)$};
    \]
    and
     \[
  \textrm{ $\chi_\natural$ is $F_\mu$-balanced}\ \Leftrightarrow  \textrm{ ${}^{\sigma\!}(\chi_\natural)$ is $F_{{}^{\sigma\!}\mu}$-balanced}.
    \]
\end{remarkl}
    
\begin{remarkl} For every $\nu=\{\nu^\iota\}_{\iota\in \CE_\rk}\in \BZ^{\CE_\rk}$, set 
    \be\label{omegamunu}
    \Omega_{\mu,\nu}: = {\rm i}^{ \sum_{\iota\in\CE_\rk}\sum^n_{i=1}(\mu^\iota_i + \nu^{\iota})} \qquad ({\rm i}:=\sqrt{-1}\in \BC).
    \ee
    Then we  have that 
\[
\Omega_{\mu,\nu}=\Omega_{^\sigma\! \mu,^\sigma \!\nu}\quad \textrm{and}\quad \Omega_{\mu, \chi_\natural\cdot\otimes_{\iota\in\CE_\rk}\iota^j} = {\rm i}^{jn[\rk\,:\,\BQ]}\cdot \Omega_{\mu,\chi_\natural}\qquad (\sigma\in \Aut(\BC),\, j\in \BZ).
\]

\end{remarkl}

\begin{remarkl} If $\rk$ contains no CM field, then 
\begin{itemize}
    \item $\chi_\natural=\otimes_{\iota\in \CE_\rk} \iota^a$ for some $a\in \BZ$; 
    \item $\chi_\natural$ is $F_\mu$-balanced if and only if $\frac{1}{2}$ is a critical place of $\oL(s, \Pi\otimes \chi)$; and 
    \item $\frac{1}{2}$ is a critical place of $\oL(s, \Pi\otimes \chi)$ for some algebraic Hecke characters $\chi: \rk^\times\backslash \A^\times\rightarrow \mathbb C^\times$, which follows easily from 
    \eqref{wiota} and \eqref{bal}.
\end{itemize}
\end{remarkl}

Let $\Upsilon$ be a subgroup of $\Aut(\BC)$ such that
\be\label{upsilon}
\textrm{$\oL(s, {}^\sigma\Pi, \wedge^2\otimes {}^\sigma\bdl{\eta}^{-1})$ has a pole at $s=1$ for all $\sigma\in \Upsilon$},
\ee
where $\BQ(\Pi,\bdl{\eta})$ denotes the compositum of the rationality fields of $\Pi$ and $\bdl{\eta}$ (see \eqref{rf}). The group $\Aut(\BC/\BQ(\Pi,\bdl{\eta}))$ clearly satisfies condition \eqref{upsilon}. It is expected that the group $\Aut(\BC)$ itself satisfies \eqref{upsilon}. In view of the proof of \cite[Proposition 8.2]{CKT26}, this would follow from an extension of Arthur's endoscopic classification (\cite{A13}) to ${\rm GSpin}_{2n+1}$; however, such an extension is currently unavailable. At present, it is known that $\Aut(\BC)$ satisfies \eqref{upsilon} when $\rk$ is not totally imaginary (\cite[Theorem 5.3]{GR13}, \cite[Corollary 7.2]{CCDR26}), or when $\bdl{\eta}$ is trivial (\cite[Proposition 8.2]{CKT26}).


We identify the set of quadratic characters of  $\rk_\infty^\times$ with the set of characters $\widehat{\pi_0(\rk_\infty^\times)}$ of the component group $\pi_0(\rk_\infty^\times)$ of $\rk_\infty^\times$, 
so that $\chi^\natural \in \widehat{\pi_0(\rk_\infty^\times)}$. 

The main result of this paper is the following theorem.

\begin{thml}[Blasius-Deligne conjecture] \label{BDconj}
Let the notation be as above. 
Then the following reciprocity identity 
\be \label{reci}
\sigma\left( \frac{\oL(\frac{1}{2}, \Pi\otimes\chi)}{\Omega_{\mu, \chi_\natural} \cdot \CG(\chi)^n  \cdot \Omega_{\chi^\natural}(\Pi,\bdl{\eta})}\right) = \frac{\oL(\frac{1}{2}, {}^\sigma \Pi\otimes{}^\sigma\!\chi)}{\Omega_{\mu, \chi_\natural}  \cdot \CG({}^\sigma\!\chi)^n \cdot \Omega_{\chi^\natural}(^\sigma \Pi, {}^\sigma\! \bdl{\eta})}
\ee
holds for every $\sigma\in\Upsilon$ and  every  algebraic Hecke character  $\chi$ of $\rk^\times\bsl \BA^\times$ 
such that $\chi_\natural$ is $F_\mu$-balanced. Here 
\begin{itemize} 
\item $\Omega_{\mu,\chi_\natural}$ is defined as in \eqref{omegamunu}; 
\item $\CG(\chi) = \CG(\chi_f)$ is the Gauss sum of $\chi$ in \eqref{globalgauss};
\item $\{\Omega_{\chi^\natural}({}^\sigma \Pi, {}^\sigma\!\bdl{\eta} )\}_{\sigma\in \Upsilon/\Upsilon_{\Pi,\bdl{\eta}}}$ is the family of Shalika periods in Definition \ref{def:period}, and $\Upsilon_{\Pi,\bdl{\eta}}:=\{\sigma\in \Upsilon \,\mid\, ({}^\sigma \Pi, {}^\sigma\!\bdl{\eta} )=(\Pi, \bdl{\eta} )\}$. 
\end{itemize}
In particular, 
\be \label{alg}
\frac{\oL(\frac{1}{2}, \Pi\otimes\chi)}{\Omega_{\mu, \chi_\natural} \cdot \CG(\chi)^n  \cdot \Omega_{\chi^\natural}(\Pi, \bdl{\eta})} \in \BQ(\Pi, \bdl{\eta}, \chi),
\ee
where $\BQ(\Pi, \bdl{\eta},\chi)$ is the compositum of the rationality fields  of $\Pi, \bdl{\eta}$ and $\chi$. 
\end{thml}

Here the rationality fields are similarly defined as in \eqref{rf}.

The reciprocity identity \eqref{reci} has the following important consequence, the general conjecture of which is attributed to P. Deligne. We refer to \cite{CK25, CKT26} for some relevant progress. 
In particular, it is proved  in \cite{CKT26} by a different method under the assumptions that $\Pi$ is self-dual of symplectic type, $\Pi_v$ is superregular for all $v\mid\infty$ and $\chi$ is trivial, while the $F_\mu$-balanced condition is not assumed (See \cite[Theorem A]{CKT26}
for the precise statement).  
 
\begin{corl}
  Let $\chi: \rk^\times \bsl \BA^\times \to\BC^\times$ be an algebraic Hecke character such that $\chi_\natural$ is $F_\mu$-balanced. Then   $\oL(\frac{1}{2}, \Pi\otimes\chi)\neq 0$ implies that $\oL(\frac{1}{2}, {}^\sigma\Pi\otimes{}^\sigma\!\chi)\neq 0$ 
    for all  $\sigma\in\Upsilon$.
\end{corl}

We make some historical remarks regarding Theorem \ref{BDconj}.

\begin{remarkl}

In \cite{GR14}, H. Grobner and A. Raghuram study the critical $L$-values 
$\oL(\frac{1}{2}+j,\Pi_f\otimes\chi_f)$ under the extra assumption that the 
ground field $\rk$ is totally real. One of the main results, 
\cite[Theorem~7.1.2]{GR14}, asserts that the critical $L$-value 
$\oL(\frac{1}{2}+j,\Pi_f\otimes\chi_f)$ is of the form
\[
\omega^{(-1)^{j}\epsilon_{\chi}}(\Pi_f)\cdot\CG(\chi)^n\cdot\omega(\Pi_\infty,j),
\]
up to a multiple of a number belonging to the rationality field 
$\BQ(\Pi,\bdl{\eta},\chi)$, where $\omega(\Pi_\infty,j)\in \mathbb C^\times$ 
is a scalar represented by a certain integral which only depends on $j$ and the archimedean component $\Pi_\infty$. 
This case was also considered earlier in \cite{AshGinzburgPadicLfunGL(2n)}.


\end{remarkl}

\begin{remarkl}
    In \cite{Jan16}, F. Januszewski studies  the critical $L$-values
$\oL(\frac{1}{2}+j,\Pi_f\otimes\chi_f)$ (also with the extra assumption that the ground field $\rk$ is totally real)  and asserts that
$\oL(\frac{1}{2}+j,\Pi_f\otimes\chi_f)$ is of the form
$$
(2\pi\mathrm{i})^{jn[\rk:\BQ]}\cdot\CG(\chi)^n\cdot\Omega_{(-1)^{j}\sgn\chi},
$$
up to a multiple of a number belonging to the rationality field $\BQ(\Pi,\chi)$. 
His argument towards this global period relation
assumes without proof in \cite[Section 5.3]{Jan16} the existence of the cohomological test vector for $\Pi_\infty$ that is independent of the critical places. As the most technical ingredient of this paper,
our Theorem \ref{APR} establishes the archimedean period relations without using cohomological test vectors. 
Meanwhile, our proof involves no rational structures of Harish-Chandra modules, which is however a major ingredient in the argument developed by Januszewski.
\end{remarkl}

\begin{remarkl}
In \cite{AshGinzburgPadicLfunGL(2n)}, \cite{GR14} and \cite{Jan16}, the base field is assumed to be totally real so that a certain numerical coincidence called ``degree matching" is satisfied. As far as we know, the proof of Theorem \ref{BDconj} is the first case in the study of special values of $L$-functions via modular symbols without the ``degree matching" condition. 
\end{remarkl}

\begin{remarkl}
 For $n=1$, all irreducible cuspidal automorphic representations of $\GL_2(\BA)$ are 
  of symplectic type. In this case, the theorem was considered by Y. Manin in \cite{Mn76}, by G. Shimura in \cite[Theorem 1]{Shi76} and \cite[Theorem 4.3]{Shi78}, by G. Harder in \cite[Page 84]{Hd83}, and by H. Hida in \cite[Theorem 8.1]{Hi94}.
In \cite[Section L.9.1]{B97}, D. Blasius gives an explicit discussion on the conjecture of period relations and critical values of $L$-functions attached to cuspidal automorphic forms in general.
We also refer to the survey paper by M. Harris and J. Lin in \cite{HL17} and recent Bourbaki lecture of M. Harris in \cite{H25} for further explanation of the recent progress in the theory.
\end{remarkl}

\begin{remarkl}
In the case of Rankin–Selberg convolutions for $\GL_n \times \GL_{n-1}$ for general 
$n$, much work has been done on the rationality of the critical values and period relations since the landmark work of Kazhdan, Mazur, and Schmidt in \cite{KMS00}, see  \cite{M,KS13,Rag1,Rag2,GH,HL17,GS20,GL21,GHL,LLS24, HMN25}. 
\end{remarkl}

Theorem \ref{BDconj} is the first to consider the Blasius-Deligne conjecture with general algebraic Hecke characters, which are not necessarily unramified twists of finite order Hecke characters.  Here are some more detailed remarks regarding Theorem \ref{BDconj}, which give an outline of the strategy and byproducts of its proof.

Among others, there are two technical key results needed for the formulation and the proof of Theorem \ref{BDconj}: the {\it non-vanishing hypothesis at infinity} 
 and the {\it archimedean period relations}.

The non-vanishing hypothesis at infinity (Theorem \ref{thm:NV}) is proved by using the methods in \cite{S17, S19} and the non-vanishing result in \cite{CJLT20, LT20}. However, 
the results on the uniform cohomological test vectors in \cite{CJLT20, LT20} are not enough to establish 
a general version of the archimedean period relations  (Theorem \ref{APR}).   
In the works \cite{GR14, GM24}, an even weaker version of test vectors was used, which is also  unable to reach the complete result as in Theorem \ref{BDconj}.  

The archimedean period relations (Theorem \ref{APR}) are proved via a robust application of Jantzen--Zuckerman translation functors and the method of modifying factors (\cite{JLST}). This approach has been used in \cite{LLS24}  for the Rankin-Selberg case. The arguments in this paper (based on \cite{JLST}) combined with those in \cite{LLS24} represent a novel, more effective approach to establish the reciprocity law in the Blasius-Deligne conjecture for automorphic $L$-functions. We explain some more details as follows.

The Shalika period  $\Omega_\varepsilon(\Pi, \bdl{\eta})$  is  defined in terms of the Friedberg-Jacquet local zeta integrals (\cite{FJ93}).  The definition of such integrals needs a local Shalika functional. In order to establish archimedean period relations (Theorem \ref{APR}), we need an explicitly normalized local Shalika functional 
(Proposition \ref{prop:Sha}) to define explicit Friedberg-Jacquet local zeta integrals as in \eqref{FJ-I}. 
We construct such explicitly normalized local Shalika functionals as in Proposition \ref{prop:Sha} by exploring refined properties of the Jacquet-Shalika local zeta integrals (\cite{JS90}) for principal series representations, which are established in \cite{JLST}.
 In particular, the local theory of Jacquet-Shalika integrals in the even case gives an explicit realization of Shalika functionals (Theorem \ref{cor:sha}). 
  With this realization, an open-orbit comparison for the Friedberg-Jacquet local zeta integrals leads to the modifying factors for standard $L$-functions of symplectic type via Friedberg-Jacquet integrals (Theorem \ref{MFSha}, which is the key local result in the analysis aspect in this paper). As an application of modifying factors, we prove the archimedean period relations for Friedberg-Jacquet integrals (Theorem \ref{APR}) by using translation functors between cohomological representations.

This general open-orbit comparison method yields substantial arithmetic applications, and can be extended to treat the arithmetic problems for more general automorphic $L$-functions for higher rank groups. In the Jacquet-Shalika case, it leads to the modifying factors for exterior square $L$-functions (Theorem \ref{thm:MF_m}), which are  compatible with the prediction for $p$-adic $L$-functions in \cite{CPR89, C89}. The main result of this paper and the open-orbit method have been used in \cite{LS25} to construct nearly ordinary standard $p$-adic $L$-functions of symplectic type.



This paper is organized as follows. In Section \ref{sec:MLR} we give a summary of the local results in \cite{JLST}. In Section \ref{sec:FJMF} we establish the modifying factors for Friedberg-Jacquet integrals. We prove the archimedean period relations in Section \ref{sec:PAPR}, and prove the non-vanishing hypothesis at infinity in 
Section \ref{sec:NV}.  The non-archimedean period relations are proved in Section \ref{sec:NAPR}. We turn to the global setting in Section \ref{sec:CGMS}, where we introduce certain cohomology groups and the global modular symbols for Friedberg-Jacquet integrals. Finally in Section \ref{sec:SPBD} we define the family of Shalika periods and prove the Blasius-Deligne conjecture (Theorem \ref{BDconj}).

\section{Jacquet-Shalika integrals and modifying factors}  \label{sec:MLR}

In this section, we recall the results on local Jacquet-Shalika integrals and the corresponding open orbit integrals for $\GL_{2n}$ proved in \cite{JLST}.

\subsection{Representations and exterior square local factors} \label{ssec2.1}

Let $\bk$ be a local field of characteristic zero, with normalized absolute value $\abs{\cdot}_\bk$. 
For a connected reductive group $G$ over $\bk$, denote by $\Irr(G)$ the set of isomorphism classes of irreducible admissible smooth representations of $G$ (in the archimedean case, ``admissible smooth representations'' means Casselman--Wallach representations). 
Here and henceforth, when no confusion is possible, we do not distinguish a reductive group over $\bk$ from the group of its $\bk$-points. 
We also do not distinguish an element of $\Irr(G)$ from an irreducible representation representing it. 
Let $\Irr_{\mathrm{temp}}(G)\subset \Irr(G)$ denote the subset of the tempered unitarizable representations.

For an  integer $k\geq 0$, write $G_k:=\GL_k(\bk)$. Let $G :=G_{2n}$ ($n\geq 1$) and let $N$ be the  upper triangular maximal unipotent subgroup of $G$. Fix a nontrivial unitary character $\psi $ of $\bk$, and define a character 
\[
\psi_N:  N\to \BC^\times,\quad [x_{i,j}]_{1\leq i,j\leq  2n}\mapsto \psi\left(\sum^{2n-1}_{i=1}x_{i,i+1}\right).
\] 
To shorten the notation,  in this paper we  write $\omega(g) := \omega(\det g)$ and $|g|_\bk := \abs{\det g}_\bk$ for a character $\omega$ of $\bk^\times$ and $g\in G_m$. 

We consider a representation of $G$ given by the  normalized smooth parabolic induction 
\be \label{nt}
\pi_\lambda :=  \Ind^{G}_{\overline P}(\tau_\lambda) = \Ind^{G}_{\overline P} ( \tau_1 \abs{\cdot}^{\lambda_1}_\bk \,\widehat\otimes \, \tau_2\abs{\cdot}_\bk^{\lambda_2}\,\widehat\otimes \cdots \widehat\otimes\, \tau_r\abs{\cdot}^{\lambda_r}_\bk),
\ee
where 
\begin{itemize}
\item $\overline P$ is the block lower triangular parabolic subgroup of $G$ with Levi subgroup 
\[
M :=   G_{n_1} \times G_{n_2}\times \cdots\times G_{n_r},
\]
where $r\geq 1, n_1,n_2,\ldots, n_r\geq 1$, and $n_1+n_2+\cdots+n_r=2n$;
\item
 $\tau := \tau_1\,\widehat \otimes \, \tau_2\, \widehat\otimes \cdots\widehat\otimes \, \tau_r \in \Irr_{\temp}(M)$,  where  $\tau_i\in\Irr_{\temp}(G_{n_i})$ ($i=1,2,\dots, r$);
 \item  $\lambda := (\lambda_1, \lambda_2, \dots, \lambda_r)  \in \BC^r$; and 
 \item  $\tau_{\lambda}:=\tau_1 \abs{\cdot}^{\lambda_1}_ \bk \,\widehat\otimes \, \tau_2\abs{\cdot}_\bk^{\lambda_2}\,\widehat\otimes \cdots \widehat\otimes\, \tau_r\abs{\cdot}^{\lambda_r}_\bk\in\Irr(M)$.
 \end{itemize}
Here ``$\,\widehat\otimes\,$"stands for the completed inductive  tensor product. In the non-archimedean case, the smooth representations are equipped with the finest locally convex topologies and the completed inductive  tensor product agrees with the algebraic tensor product. 
The following facts are well-known.
 \begin{itemize}
  \item  The representation $\pi_\lambda$ is of Whittaker type in the sense that 
  \[
  \dim \Hom_{N}(\pi_\lambda, \psi_N)=1.
  \] 
  \item For fixed $\tau \in \Irr_{\temp}(M)$, $\pi_\lambda$ is irreducible for $\lambda$ outside a measure zero subset of  $\BC^r$.
  \item Each $\pi \in \Irr_{\rm gen}(G)$  (the subset of generic representations in $\Irr(G)$)  is isomorphic to an induced representation  $\pi_\lambda$ of the form \eqref{nt}.
 \end{itemize}
Denote by $\Re(z)$ the real part of a complex number $z$. We will use the following notation: for $\lambda :=(\lambda_1,\lambda_2, \dots \lambda_r) \in \BC^r$, write
\[  
\min\Re(\lambda):=\min_{i=1,2,\dots, r} \Re(\lambda_i),\quad \max\Re(\lambda) :=\max_{i=1,2,\dots, r} \Re(\lambda_i).
\]
 Following \cite{BP21}, an admissible smooth representation of $G$ is called  {\it nearly tempered} if it is isomorphic to some $\pi_\lambda$ as in \eqref{nt} with  $|\Re(\lambda_i)|<1/4$ for all $i=1, 2, \ldots, r$.
It is known that all nearly tempered representations are irreducible.

For $\pi\in \Irr(G)$, denote by $\phi_\pi$ the local $L$-parameter of $\pi$ under the local Langlands correspondence, which is a $2n$-dimensional  admissible representation of the 
Weil-Deligne group $\CW_\bk'$  of $\bk$.  
Fix a character $\eta$ of $\bk^\times$.  We have the twisted exterior square local factors (see \cite{CST17, Sh24})
\be \label{exL}
\begin{aligned} & \oL(s, \pi, \wedge^2\otimes \eta^{-1}) := \oL(s, \wedge^2\phi_\pi\otimes \eta^{-1}),\\
& \varepsilon(s, \pi, \wedge^2\otimes\eta^{-1}, \psi) := \varepsilon(s, \wedge^2\phi_\pi\otimes\eta^{-1}, \psi),\\
& \gamma(s, \pi, \wedge^2\otimes \eta^{-1},\psi) :=  \varepsilon(s, \pi, \wedge^2\otimes\eta^{-1}, \psi)\cdot  \frac{\oL(1-s, \pi^\vee, \wedge^2\otimes \eta)}{\oL(s, \pi, \wedge^2\otimes \eta^{-1})},
\end{aligned}
\ee
where the right hand sides are as in \cite{T79}.
For the parabolic induction $\pi_\lambda$ in \eqref{nt}, we define 
\be \label{exL2nt}
\begin{aligned}
\oL(s, \pi_\lambda, \wedge^2\otimes\eta^{-1})  := & \prod^r_{i=1} \oL(s+2\lambda_i, \wedge^2\phi_{\tau_i}\otimes\eta^{-1})   \\
&\qquad \cdot \prod_{1\leq j< k \leq r} \oL(s+\lambda_j+\lambda_k, \phi_{\tau_j}\otimes\phi_{\tau_k}\otimes\eta^{-1}),
\end{aligned}
\ee
and $\varepsilon(s, \pi_\lambda, \wedge^2\otimes\eta^{-1},\psi)$ and $\gamma(s, \pi_\lambda, \wedge^2\otimes \eta^{-1})$ are similarly defined. 

By the compatibility of local Langlands correspondence with parabolic induction and unramified twists, if $\pi_\lambda^0$ denotes the unique Langlands subquotient of $\pi_\lambda$, then
\[
 \oL(s, \pi_\lambda,  \wedge^2\otimes \eta^{-1}) = \oL(s, \pi_\lambda^0,  \wedge^2\otimes \eta^{-1}),\quad \varepsilon(s, \pi_\lambda, \wedge^2\otimes\eta^{-1},\psi) = 
\varepsilon(s, \pi_\lambda^0, \wedge^2\otimes\eta^{-1},\psi)
\] 
where the right hand sides are given by \eqref{exL}. In particular, $\eqref{exL}$ and $\eqref{exL2nt}$ coincide when $\pi_\lambda$ is irreducible. 

\subsection{Jacquet-Shalika integrals} \label{sec1.2.2}
Fix the self-dual Haar measure on $\bk$ with respect to $\psi$.  For integers $k, k'\geq 0$, denote by $\bk^{k \times k'}$ the space of 
$k\times  k'$ matrices over $\bk$, and write $M_k :=\bk^{k\times k}$.  We endow $\bk^{k\times k'}$ with the product measure, and fix the Haar measure on $G_k$ to be 
$\od\! g = | g|_\bk^{-k}\cdot \prod_{i, j =1,2,\dots, k} \od\! g_{i,j}$ for $g=[g_{i, j}]_{1\leq i,j\leq k}\in G_k$.
For $\phi \in \CS(\bk^k)$, the space of Schwartz functions on  $\bk^k :=\bk^{1\times k}$, define its Fourier transform  with respect to $\psi$ by
\[
\widehat\phi (x) := \int_{\bk^k} \phi(y) \psi(y\, {}^t\!x)\od\! y, \quad x\in \bk^k.
\]
Here and thereafter, ${}^t\!(\cdot)$ indicates the transpose of a matrix.

The Shalika subgroup $S$ of $G$ is defined by
\[
S:=  \Set{\begin{bmatrix} g & Xg \\ 0 & g \end{bmatrix}  | g\in G_n, X\in M_n},
\]
which is a unimodular group. The Haar measures on $S$ and $N$  etc. are induced from the fixed Haar measures on $G_n$ and $\bk$. We always take right invariant quotient measures (when such measures exist) on  homogeneous spaces under  right actions of locally compact groups. 

Define a character 
\be \label{varphi2n}
 \varphi:=\varphi_\eta: S\to \BC^\times, \quad \begin{bmatrix} g& Xg  \\ 0& g\end{bmatrix} \mapsto \eta(g) \psi({\rm tr}\, X).
\ee
Let $S$ act on $\bk^n$ from the right by
\be \label{Sact}
(v,h)\mapsto v\cdot h:=vg, \quad \textrm{where } v\in \bk^n,\, h = \begin{bmatrix} g & Xg   \\ 0& g \end{bmatrix}\in S.
\ee
Then we define a representation $R$ of $S$ on $\CS(\bk^n)$ by 
\be  \label{Seven}
\left(R(h)\phi\right)(v) :=   \phi(v\cdot h), 
\quad h\in S,\, \phi \in \CS(\bk^n), \, v\in \bk^n.
\ee

We identify a permutation $\sigma\in \frak{S}_{2n}$ with the permutation matrix in $G$ whose $(\sigma(i),i)$-th entry is $1$ for all $i=1,2,\dots,2n$. Then $\frak{S}_{2n}$ is identified with a subgroup of $G$.
Introduce as in \cite{JS90} the following element
\[ 
  \sigma_{2n}:= \begin{pmatrix} 1 & 2 & \cdots & n & n+1 & n+2 & \cdots & 2n  \\
1 & 3 & \cdots & 2n-1 & 2 & 4 & \cdots & 2n \end{pmatrix}\in \frak S_{2n}.
\]

Let $\pi_\lambda = \Ind^{G}_{\overline P}(\tau_\lambda)$ ($\lambda\in\BC^r$) be an induced representation of $G$ as in \eqref{nt}. 
Fix a nonzero Whittaker functional  
    	\[
        \nu_{\tau,\psi}\in \Hom_{N\cap M}(\tau, \psi_N|_{N\cap M}),
        \] 
        which is unique up to scalar multiplications. 
For each $\lambda\in \BC^r$, it determines  a non-zero Whittaker  functional 
	\[
	J_{ \lambda}\in \Hom_{N}(\pi_\lambda, \psi_N) 
	\]
    such that 
    \[ 
	J_{ \lambda}(f)=\int_{(N\cap M)\backslash N} \la \nu_{\tau,\psi}, f(x)\ra\cdot \overline\psi_N(x) \od\! x
    \]
    for  all $f\in \pi_\lambda$ whose support is contained in $\overline PN$.   
Denote by 
		\begin{equation}\label{whittaker2}
	\CW(\pi_\lambda,\psi):=\set{W_f | f\in \pi_\lambda}
	\subset \Ind_{N}^{G} \psi_N \qquad(\text{smooth induction})
	\end{equation}
    the  Whittaker model of $\pi_\lambda$ with respect to $(N, \psi_N)$, where
    \[
    W_{f}(g)=\la J_{\lambda}, g\cdot f\ra, \quad \textrm{ for all $g\in G$}.
    \]
     We equip $\CW(\pi_{\lambda},\psi)$ with the quotient topology of $\pi_{\lambda}$. In general, for every admissible smooth representation 
     $\pi$ of $G$ such that $\dim\Hom_{N}(\pi, \psi_N)=1$, we can similarly define its Whittaker model $\CW(\pi, \psi)\subset \Ind_{N}^{G} \psi_N$.

For $W \in \CW(\pi_\lambda, \psi)$, $\phi \in \CS(\bk^n)$,  and $s\in\BC$, the Jacquet-Shalika integral introduced in 
\cite{JS90} can be reformulated as 
\be \label{JSint}
 \oZ_{\rm JS}(s, W, \phi, \varphi^{-1}):  = 
\int_{\overline{S}} W(\sigma_{2n} h) \cdot \left(R(h)\phi\right)(e_n) \cdot\varphi^{-1}(h)\cdot | h|_\bk^{\frac{s}{2}}\od\! h,
\ee
where  $e_n :=(0,0,\ldots, 0,1)\in \bk^n$  and 
\[ 
\overline{S}:= \left(\sigma_{2n}^{-1}N \sigma_{2n} \cap S \right)\bsl S. 
\]
Write
 \be\label{zcirc}
            \oZ_{\rm JS}^\circ(s, W, \phi, \varphi^{-1}): =  \frac{\oZ_{\rm JS}(s, W, \phi, \varphi^{-1})}{\oL(s, \pi_\lambda, \wedge^2\otimes\eta^{-1})}.
            \ee

Let $K$ denote the standard maximal compact subgroup of $G$, namely 
\[
K: =
\begin{cases}
\RO(2n), & \text{if } \bk= \R,\\[2pt]
\RU(2n), & \text{if } \bk \cong \BC,\\[2pt]
\GL_{2n}(\CO_\bk), & \text{if } \bk \text{ is non-archimedean with ring of integers } \CO_\bk.
\end{cases}
\]
Write
\[
  \pi_\tau:=\mathrm{Ind}^{K}_{K\cap \overline P}(\tau),
\]
which is a smooth representation of $K$.  For every $\lambda\in \BC^r$, the restriction of functions induces an identification 
\[
  (\pi_\lambda)|_{K}=\pi_\tau. 
\]
For each $f\in \pi_\tau$, write $f_\lambda $ for the corresponding element in $\pi_\lambda$.

For every integer $m\geq 1$, denote by $w_m$ the $m\times m$ anti-diagonal permutation matrix. For every $W\in \CW(\pi_\lambda, \psi)$, write 
\[
\widetilde{W}(h) := W(w_{2n} {}^t h^{-1}),\quad h\in G.
\] 
Then $\widetilde W\in \CW(\breve\pi_\lambda, \overline\psi)$ (similarly defined as in \eqref{whittaker2}), where $\breve\pi_\lambda$ denotes the MVW involution 
of $\pi_\lambda$ (see \cite{MVW87}) so that
\[
\breve\pi_\lambda(h):=\pi_\lambda(w_{2n}{}^t h^{-1} w_{2n}) \quad \textrm{for all }h\in G.
\]
In particular $\breve\pi_\lambda \cong \pi_\lambda^\vee$ when $\pi_\lambda$ is irreducible. 
Introduce the following element
\[ 
 \sigma'_{2n}: = \begin{bmatrix}  0 & 1_n \\ 1_n & 0 \end{bmatrix} \in G.
\]
Here and thereafter, $1_n$ denotes the $n\times n$ identity matrix. 
Denote by  $\widehat{\bk^\times}$ the set of characters of $\bk^\times$, and for any $\omega\in \widehat{\bk^\times}$ let $\Re(\omega)$ be the real number 
such that $|\omega(a)| =|a|_\bk^{\Re(\omega)}$ for $a\in \bk^\times$. 

Our first main result on the local theory of Jacquet-Shalika integrals for $\GL_{2n}$ is as follows (\cite[Theorem 1.1]{JLST}). 

	\begin{thml} 
    \label{thm:FE_m}
		For each $\lambda\in\BC^r$, consider the representation $\pi_\lambda = \Ind^{G}_{\overline{P}}(\tau_\lambda)$ as in \eqref{nt}. When $\bk$ is non-archimedean, assume that the cuspidal support of $\tau$ is represented by a character of  $(\bk^\times)^{2n}$. 
		\begin{enumerate}
			\item Suppose that $(s,\lambda)\in \BC\times \BC^r$ satisfies that 
            \be \label{ss'}
            \Re(s)>\Re(\eta)-2\min\Re(\lambda).
            \ee
            Then for every $W\in \CW(\pi_\lambda, \psi)$ and $\phi \in \CS( \bk^n)$, the integral $\oZ_{\rm JS}(s, W, \phi, \varphi^{-1})$ converges absolutely.

			\item  There exists a (unique) continuous function
            \[
         \oZ^\circ(\,\cdot\, ,\,\cdot, \,\cdot\, ,\, \cdot\, ):   \BC\times  \BC^r\times \pi_\tau\times  \CS( \bk^n)\rightarrow \BC
            \]
            that is holomorphic in the first two variables and linear in the last two variables such that 
         \[
				\oZ^\circ(s,\lambda, f, \phi) =  \oZ^\circ_{\rm JS}(s, W_{f_\lambda}, \phi, \varphi^{-1})
			\]
            for all $f\in \pi_\tau$, $\phi \in \CS( \bk^n)$, and $(s,\lambda)\in \BC\times \BC^r$ satisfying \eqref{ss'}.

\item Let $(s,\lambda)\in \BC\times \BC^r$. For every $W\in \CW(\pi_\lambda, \psi)$ and $\phi \in \CS( \bk^n)$, it holds the functional equation 
        \be \label{eq:FE_m}
				\oZ_{\rm JS}^\circ(1-s,  \sigma'_{2n}\cdot\widetilde W, \widehat\phi, \varphi) = \varepsilon(s, \pi_\lambda, \wedge^2\otimes \eta^{-1},\psi) \oZ_{\rm JS}^\circ(s, W, \phi, \varphi^{-1}).
\ee

			\item If $\lambda\in \BC^r$ satisfies that $\max\Re(\lambda)< \min\Re(\lambda)+1/2$, 
			then for every $s_0\in\BC$ there exist $W\in \CW(\pi_\lambda,\psi)$ and $\phi\in \CS(\bk^n)$ such that 
$
\oZ_{\rm JS}^\circ(s_0, W, \phi, \varphi^{-1}) \neq 0.
$
		\end{enumerate}
	\end{thml}

Here and as in \eqref{zcirc}, the left hand side of \eqref{eq:FE_m} is defined by 
\[
\oZ_{\rm JS}^\circ(1-s,  \sigma'_{2n}\cdot\widetilde W, \widehat\phi, \varphi) := \frac{\oZ_{\rm JS}(1-s,  \sigma'_{2n}\cdot\widetilde W, \widehat\phi, \varphi)}{\oL(1-s, \pi_\lambda^\vee,\wedge^2\otimes\eta)}\in \BC,
\]
where $\oZ_{\rm JS}(1- s,  \sigma'_{2n}\cdot\widetilde W, \widehat\phi, \varphi)$ is similarly defined as in \eqref{JSint} and converges absolutely when $1- \Re(s) >2\max\Re(\lambda)-\Re(\eta)$. Note  that $\breve\pi_\lambda$ and $\pi_\lambda^\vee$ are  also  induced representations as in \eqref{nt}, and that $\oL(s, \breve\pi_\lambda, \wedge^2\otimes\eta) = \oL(s, \pi_\lambda^\vee, \wedge^2\otimes\eta)$ (defined as in \eqref{exL2nt}).

Recall that every $\pi \in \Irr_{\rm gen}(G)$ is isomorphic to an induced representation $\pi_\lambda$, for which we can similarly define 
\[
\oZ_{\rm JS}^\circ(s, W, \phi, \varphi^{-1}) = \frac{\oZ_{\rm JS}(s, W, \phi, \varphi^{-1}) }{\oL(s, \pi, \wedge^2\otimes\eta^{-1})},\quad W\in \CW(\pi,\psi), \ \phi\in \CS(\bk^n).
\]

In particular, we have the following:
\begin{itemize}
\item  Theorem  \ref{thm:FE_m} applies to all $\pi \in \Irr_{\rm gen}(G)$ when $\bk$ is archimedean.  
\item If $\pi_\lambda\otimes |\eta|^{-\frac{1}{2}}$ is nearly tempered, then the condition in Theorem \ref{thm:FE_m} (4) clearly holds. Here $|\eta|^{-\frac{1}{2}}$ indicates the character $|\eta(\det(\cdot))|^{-\frac{1}{2}}$ of $G$. 
\end{itemize}

\subsection{Open orbit integrals and modifying factors}
The proof of Theorem~\ref{thm:FE_m} in \cite{JLST} is purely local and uses the idea from \cite{LLSS23}, which studies modifying factors for the Rankin--Selberg convolution case. The strategy is to compare the Jacquet--Shalika integrals for principal series representations with the integrals over the open orbit of the action of the Shalika subgroup $S$ on a certain variety, as given in Theorem~\ref{thm:MF_m}. Note that $S$ is a spherical subgroup of $G$.

Such a comparison in turn produces certain modifying factors, which are compatible in the non-archimedean case with the conjecture for $p$-adic $L$-functions given by Coates and Perrin-Riou in \cite{CPR89, C89}. This kind of phenomenon has been observed for several families of period integrals (see \cite{LSS21, LLSS23, LS25}).
The comparison for the Friedberg--Jacquet integrals in the archimedean case (Theorem~\ref{MFSha}) will be established in this paper, which is crucial for our proof of the archimedean period relations (Theorem~\ref{APR}) and of the Blasius--Deligne conjecture for standard $L$-functions of symplectic type (Theorem~\ref{BDconj}). Its $p$-adic analog has been established in \cite{LS25}, which leads to the construction of nearly ordinary standard $p$-adic $L$-functions of symplectic type.

We introduce an $S$-variety $\CX$ as follows. Let $\overline{B}$ be the lower triangular Borel subgroup of $G$, and let $\CB:= \overline{B} \backslash G$ be the flag variety on which $G$ acts from the right. Define 
\[
\CX:= \CB \times \bk^n.
\]
We have specified a right action of $S$ on $\bk^n$ in \eqref{Sact}. 
The diagonal action of $S$ on $\CX$ has a unique Zariski-open orbit, with a base point
\[ 
x_0 :=
(\overline{B} z_0, v_n),
\]
where
\[
v_n := (1,1,\dots,1) \in \bk^n,
\]
and
\begin{equation} \label{zm}
z_0:=
\begin{bmatrix} 1_n & 0 \\ 0 & w_n \end{bmatrix}.
\end{equation}
Moreover, the stabilizer of $x_0$ in $S$ is trivial.

View an element $\xi=(\xi_1,\xi_2,\dots,\xi_{2n})\in (\widehat{\bk^\times})^{2n}$ as a character of $\overline{B}$ in the obvious way and put 
\be \label{PS}
I(\xi) : = {\rm Ind}^{G}_{\overline{B}}(\xi)\qquad \ (\textrm{normalized smooth induction}).
\ee
For $f\in I(\xi)$, $\phi \in \CS(\bk^n)$,  and $s\in\BC$, formally define the  {\it open orbit integral} 
\be \label{openint}
\Lambda_{\rm JS}(s, f, \phi, \varphi^{-1}) :=
 \int_{S} f(z_0 h) \cdot \left(R(h) \phi\right)(v_n) \cdot \varphi^{-1}(h)\cdot | h|_\bk^{\frac{s}{2}} \od\!h.
\ee
Denote by $W_f \in \CW(I(\xi), \psi)$  the Whittaker function associated to $f$ and $\psi$ via  the Jacquet integral 
\[
W_f(g) =\int_{N} f(ug) \overline\psi_N(u)\od\!u
\]
in the sense of holomorphic continuation (see \cite{J67} and \cite[Theorem 15.4.1]{W92} for a detailed explanation). 
Put  $\tilde\xi:=(\xi_{2n}^{-1}, \dots, \xi_1^{-1})$. For every  $f\in I(\xi)$, write 
\[
\tilde f(h): = f(w_{2n} {}^t h^{-1}), \quad h\in G.
\]
Then $\tilde f \in I(\tilde\xi)$ and  
\[
\widetilde{W_f} = W_{\tilde f} \in \CW(I(\tilde\xi), \overline\psi).
\]
Here and thereafter, by abuse of notation we write $W_{\tilde f}$ for the Whittaker function associated to $\tilde f$ and $\overline\psi_N$, which should not cause any confusion.

The connected component $\CM$ of $(\widehat{\bk^\times})^{2n}$ containing $\xi$ is the set of all the unramified twists of $\xi$, which  is a complex manifold
of dimension $2n$.  A {\it standard section} on $\CM$ is a family $\{f_{\xi'}\in I(\xi')\}_{\xi'\in \CM}$ such that $f_{\xi'} |_{K}$ does not depend on $\xi'\in \CM$. For every $f\in I(\xi)$, there is a unique standard section $\{f_{\xi'}\in I(\xi')\}_{\xi'\in \CM}$  such that $f_\xi = f$. 
Define  
\be \label{minmax'}
\begin{aligned}
& \Omega_\eta:= \Set{ (s, \xi ) \in \BC \times (\widehat{\bk^\times})^{2n}  |  \begin{array}{l} \Re(\xi_1)<\Re(\xi_2)<\cdots < \Re(\xi_{2n}), \\
 -2  \Re(\xi_1) < \Re(s)   - \Re(\eta) < 1 - 2 \Re(\xi_{2n}) \end{array} } \\
\subset\ & \widehat\Omega_\eta:= \Set{ (s, \xi ) \in \BC \times (\widehat{\bk^\times})^{2n}  |  \begin{array}{l} \Re(\xi_1)<\Re(\xi_2)<\cdots < \Re(\xi_{2n}), \\
 -2  \Re(\xi_1) < \Re(s)   - \Re(\eta) < 1 - 2 \Re(\xi_{2n-1}) \end{array} },
 \end{aligned}
\ee
with $\xi=(\xi_1,\xi_2,\dots,\xi_{2n})$.  

Following Tate's thesis (\cite{T50, K03}), denote the gamma factor of $\omega \in \widehat{\bk^\times}$ by 
\[
\gamma(s, \omega, \psi) = \varepsilon(s, \omega, \psi) \frac{\oL(1-s,\omega^{-1})}{\oL(s, \omega)}.
\]

The relevant analytic properties of the open orbit integrals $\Lambda_{\rm JS}(s, f, \phi, \varphi^{-1})$ are as follows (\cite[Theorem 1.3]{JLST}). 

\begin{thml} \label{thm:FE'_m} Let $\phi\in \CS(\bk^n)$.  
\begin{enumerate}
\item For $(s, \xi)\in \widehat\Omega_\eta$ given in \eqref{minmax'} and $f\in I(\xi)$,  the integral $\Lambda_{\rm JS}(s,f,\phi,\varphi^{-1})$ in \eqref{openint} converges absolutely. If moreover $(s,\xi)\in \Omega_\eta$, then it holds the functional equation
\be \label{eq:FE'}
\Lambda_{\rm JS}(1-s, \sigma'_{2n}\cdot \tilde f, \widehat\phi, \varphi) = \left(\prod^n_{i=1} \gamma(s, \xi_i \xi_{2n+1-i} \eta^{-1}, \psi)\right)\cdot  \Lambda_{\rm JS}(s,f,\phi,\varphi^{-1}).
\ee

 \item Let $\{f_{\xi}\in I(\xi)\}_{\xi\in \CM}$  be a standard section on a connected component $\CM$ of $(\widehat{\bk^\times})^{2n}$. Then the function 
 \[
 \widehat\Omega_\eta \cap (\BC\times \CM) \to\BC, \quad (s, \xi) \mapsto \Lambda_{\rm JS}(s, f_\xi,\phi,\varphi^{-1})
 \]
 has a meromorphic continuation to $\BC\times \CM$.
 \end{enumerate}
\end{thml}

The Jacquet--Shalika integral $\oZ_{\operatorname{JS}}(s, W_f, \phi, \varphi^{-1})$ and the open orbit integral $\Lambda_{\operatorname{JS}}(s, f, \phi, \varphi^{-1})$ are related as follows (\cite[Theorem 1.5]{JLST}).

\begin{thml} \label{thm:MF_m}
For $(s,\xi)\in \widehat\Omega_\eta$ given in \eqref{minmax'},  $f\in I(\xi)$, and $\phi\in \CS(\bk^n)$, it holds that 
\[
\Lambda_{\rm JS}(s,f,\phi,\varphi^{-1}) = \left(\prod_{1\leq i <j \leq 2n-i} \gamma(s, \xi_i \xi_j \eta^{-1},\psi) \right)\cdot 
\oZ_{\rm JS}(s, W_f, \phi, \varphi^{-1}).
\]
\end{thml}


\section{Friedberg-Jacquet integrals and modifying factors} \label{sec:FJMF}

In this section, we retain the notation in Section \ref{sec:MLR} and apply Theorems \ref{thm:FE_m} and \ref{thm:MF_m} to construct explicit twisted Shalika functionals $\lambda_{I(\xi)}$ as in Proposition \ref{prop:Sha}. These functionals will then be used to define and study the Friedberg--Jacquet local integrals, which form the main ingredients to establish the archimedean period relations for Friedberg-Jacquet integrals in the next section (Theorem \ref{APR}).

\subsection{Twisted Shalika functionals}

Let $\pi$ be an  admissible  smooth representation of $G$. Recall the representation $R$ of $S$ on $\CS(\bk^n)$ as in \eqref{Seven}.
It is an important observation that there is a continuous $S$-equivariant linear map 
\begin{equation}\label{qm}
\pi\,\widehat\otimes\,\CS(\bk^n)\twoheadrightarrow \pi, \qquad v\otimes \phi \mapsto \phi(0)v,   
\end{equation}
which brings the connections among different kinds of integrals discussed in the rest of this section. The  map \eqref{qm} is surjective and open. Via it we view $\pi$ as a quotient space of $\pi\,\widehat\otimes\,\CS(\bk^n)$ and view \eqref{qm} as  a quotient map. This induces an injective linear map 
\be \label{injsha}
\Hom_S(\pi,\varphi) \hookrightarrow \Hom_{S}(\pi\,\widehat\otimes\,\CS(\bk^n),  \varphi).
\ee

We have the following uniqueness of local Jacquet-Shalika periods proved in \cite{CSW}, a fact that will play an essential role in the sequel.

     \begin{leml}\label{uniquesh}
   Let $\pi$ be an irreducible admissible  smooth representation of $G$.  Then 
    \[
    \dim \Hom_{S}(\pi\,\widehat\otimes\,\CS(\bk^n),  \varphi)\leq 1.
    \]
    In particular, we have
    \[
     \dim \Hom_{S}(\pi,  \varphi)\leq 1.
    \]
    \end{leml}

We mention that the uniqueness $\dim \Hom_{S}(\pi,  \varphi)\leq 1$ has been known by the works \cite[Proposition 6.1]{JR},  \cite[Theorem 1.1]{AGJ09} and \cite[Theorem A]{CS20}.

Now we work with principal series representations. 
For $\xi = (\xi_1, \xi_2, \dots, \xi_{m})\in (\widehat{\bk^\times})^{m}$ with $m\geq 1$, 
let $I(\xi)$ be the principal series representation of $G_m$ defined similarly to \eqref{PS}, and denote by $\pi(\xi)$ the unique irreducible generic subquotient of $I(\xi)$.

Our basic result on twisted Shalika models is as follows.

\begin{thml} \label{cor:sha}
Let $\xi\in (\widehat{\bk^\times})^{2n}$ and $\pi =\pi(\xi)$. Assume that $\pi\otimes \abs{\eta}^{-\frac{1}{2}}$ is nearly tempered and $\Hom_S(\pi,\varphi)\neq \{0\}$. Then 
\begin{enumerate}
\item $\oZ^\circ_{\rm JS}(0, W, \phi, \varphi^{-1})=0$ for all $W\in \CW(\pi,\psi)$ and $\phi\in \CS(\bk^n)$ with $\phi(0)=0$;

\item 
The map 
\[
 W\mapsto \oZ^\circ_{\rm JS}(0, W, \phi, \varphi^{-1}),\quad W\in \CW(\pi, \psi)
\]
induces a generator of $\Hom_S(\pi,\varphi)$, 
where $\phi$ is an arbitrary element of $\CS(\bk^n)$ such that $\phi(0)=1$.
\end{enumerate}
\end{thml}

\begin{proof}
By \eqref{injsha} and Lemma \ref{uniquesh}, we have an isomorphism of one-dimensional spaces 
\[
\Hom_S(\pi,\varphi) \cong \Hom_S(\pi\,\widehat\otimes\,\CS(\bk^n), \varphi).
\]
By Theorem \ref{thm:FE_m}, the normalized Jacquet-Shalika integral 
\[
\oZ^\circ_{\rm JS}(0, W, \phi,\varphi^{-1}),\quad W\in \CW(\pi,\psi), \ \phi\in\CS(\bk^n)
\]
induces a generator of $\Hom_S(\pi \,\widehat\otimes\,\CS(\bk^n), \varphi)$. The theorem follows easily. 
\end{proof}

Assume that $\xi=(\xi_1,\xi_2, \dots, \xi_{2n})\in (\widehat{\bk^\times})^{2n}$ in the rest of this section.  
In the following, when $\pi(\xi)$ satisfies the conditions of Theorem \ref{cor:sha} and is a quotient of $I(\xi)$,  we reinterpret the generator of $\Hom_{S}(\pi(\xi),\varphi)$, which will be crucial for the study of modifying factors and the proof of archimedean period relations
for standard $L$-functions of symplectic type (Theorem \ref{APR}) via the Friedberg-Jacquet local zeta integrals.  

In view of Theorem \ref{thm:MF_m},  introduce the modified exterior square $L$-function 
\[
\begin{aligned}
& \ \CL(s,  I(\xi), \wedge^2\otimes \eta^{-1})  \\
: = &  \left(\prod_{1\leq i<j \leq 2n-i}  \gamma(s, \xi_i \xi_j \eta^{-1}, \psi)\right) \cdot \oL(s, I(\xi), \wedge^2\otimes \eta^{-1}) \\
 = &    \prod_{1\leq i<j  \leq 2n-i} \varepsilon(s, \xi_i\xi_j\eta^{-1},\psi)\oL(1-s, \xi_i^{-1} \xi_j^{-1}\eta) 
\cdot \prod_{1\leq i<j\leq 2n, \, i+j >2n }\oL(s, \xi_i \xi_j\eta^{-1}).
\end{aligned}
\]

\begin{remarkl} In the $p$-adic case, under certain slope conditions (nearly ordinary or non-critical slope) $\CL(s,  I(\xi), \wedge^2\otimes \eta^{-1})$ is expected to be the factor at $p$ of certain exterior square $p$-adic $L$-function. 
\end{remarkl}

Let $\CM$ denote the connected component of $(\widehat{\bk^\times})^{2n}$ containing $\xi$.
By Theorem \ref{thm:FE_m} and Theorem \ref{thm:MF_m}, for every standard section $\{f_{\xi'}\in I(\xi')\}_{\xi'\in \CM}$ and every function $\phi \in \CS(\bk^n)$, the function on $\BC\times\CM$ given by 
\[
 (s,\xi')\mapsto  
 \frac{\Lambda_{\rm JS}(s, f_{\xi'}, \phi, \varphi^{-1})}{\CL(s, I(\xi'), \wedge^2\otimes \eta^{-1})}
 \]
is holomorphic and coincides with 
$
 (s,\xi')\mapsto \oZ_{\rm JS}^\circ(s, W_{f_{\xi'}}, \phi, \varphi^{-1}).
$
However, the last function might vanish at $s=0$ and $\xi'=\xi$. To remedy this issue,  we introduce 
\[
\Gamma(s, I(\xi), \wedge^2\otimes \eta^{-1},\psi) : = \prod_{1\leq i<j\leq 2n, \, i+j >2n} \gamma(s, \xi_i \xi_j \eta^{-1}, \psi),
\]
and define
\be \label{dxi}
d_{\xi}:= \ord_{s=0}\Gamma(s, I(\xi), \wedge^2\otimes \eta^{-1},\psi) \in \BZ 
\ee
so that $s^{-d_\xi}\Gamma(s, I(\xi), \wedge^2\otimes \eta^{-1},\psi)$
is holomorphic and nonzero at $s=0$.

\begin{prpl} \label{prop:Sha}
Assume that $\pi(\xi)$ satisfies the conditions of Theorem \ref{cor:sha} and is a quotient of $I(\xi)$.
\begin{enumerate}
\item 
The functional  
\be \label{eq:prop:Sha}
I(\xi)\,\widehat\otimes\,\CS(\bk^n)\rightarrow \BC,\quad f\otimes \phi\mapsto \left(s^{d_\xi} \, \Lambda_{\rm JS}(s, f, \phi, \varphi^{-1})\right)|_{s=0}
\ee
 is well defined, nonzero and factors to a functional  $\lambda_{I(\xi)}\in \Hom_{S}(I(\xi),\varphi)$ through the quotient map $I(\xi)\,\widehat\otimes\,\CS(\bk^n)\twoheadrightarrow I(\xi)$ .
 
 \item For every generator  $\lambda_{\pi(\xi)}$ of $ \Hom_{S}(\pi(\xi), \varphi)$, there exists a unique element  $p_\xi\in \Hom_{G}(I(\xi), \pi(\xi))$ such that $\lambda_{\pi(\xi)}\circ p_\xi = \lambda_{I(\xi)}$. 
\end{enumerate}
\end{prpl}

\begin{proof}
By Theorems \ref{thm:FE_m} and \ref{thm:MF_m}, for $f\in I(\xi)$ and $\phi\in \CS(\bk^n)$ we have the following functional equation:
\[
\begin{aligned}
\Gamma(s, I(\xi),\wedge^2\otimes\eta^{-1},\psi)\,\Lambda_{\rm JS}(s, f, \phi, \varphi^{-1}) 
 =  &\,\gamma(s, I(\xi), \wedge^2\otimes \eta^{-1}, \psi) \, \oZ_{\rm JS}(s, W_f, \phi,\varphi^{-1}) \\
=  & \, \oZ_{\rm JS}(1-s, \sigma'_{2n}.W_{\tilde f}, \widehat\phi, \varphi). 
\end{aligned}
\]
Since $\pi(\tilde\xi)\otimes \abs{\eta}^{\frac{1}{2}}$ is nearly tempered, $\oL(1-s, \pi(\tilde\xi), \wedge^2\otimes\eta)$ is holomorphic at $s=0$, hence it follows from Theorem \ref{thm:FE_m} (4) that the functional \eqref{eq:prop:Sha} is well defined and nonzero. By Theorem \ref{thm:FE_m} (2), Theorem \ref{cor:sha} (1) and the above functional equation, 
it factors through the quotient map $I(\xi)\,\widehat{\otimes}\,\CS(\bk^n)\twoheadrightarrow I(\xi)$. This proves Part (1). 

Let $\lambda_{\pi(\xi)}$ be a generator of $\Hom_{S}(\pi(\xi), \varphi)$. The functional $\lambda_{I(\xi)}\in \Hom_{S}(I(\xi), \varphi)$ in Part (1) factors through the Whittaker model of $I(\xi)$. Since $\pi(\xi)$ is realized as the unique irreducible generic quotient of $I(\xi)$, the space $\Hom_{G}(I(\xi), \pi(\xi))$ is one-dimensional. 
Theorems \ref{thm:MF_m}, \ref{cor:sha} and the fact that $I(\xi)$ and $\pi(\xi)$ have the same Whittaker model imply that  
there exists a unique element $p_\xi\in \Hom_{G}(I(\xi), \pi(\xi))$ such that $\lambda_{I(\xi)}=\lambda_{\pi(\xi)}\circ p_\xi$. This proves Part (2).
\end{proof}

\subsection{Friedberg-Jacquet integrals}

We call the triple $(\lambda_{I(\xi)},\lambda_{\pi(\xi)},p_\xi)$ with $\lambda_{\pi(\xi)}\circ p_\xi = \lambda_{I(\xi)}$ in Proposition~\ref{prop:Sha} a \textit{Shalika datum} attached to the representation $\pi(\xi)$.
In the following, for any fixed Shalika datum $(\lambda_{I(\xi)},\lambda_{\pi(\xi)},p_\xi)$, we study the local Friedberg--Jacquet integrals in~\cite{FJ93}.
For $\chi \in \widehat{\bk^\times}$, the local Friedberg--Jacquet integral for $\pi(\xi)$ and $\chi$ is defined by
\begin{equation} \label{FJint}
\oZ_{\mathrm{FJ}}(s, v, \chi):=
\int_{G_n} \left\langle \lambda_{\pi(\xi)}, \begin{bmatrix} g & 0 \\ 0 & 1_n \end{bmatrix} \cdot v \right\rangle \,
\chi(g) |g|_\bk^{s-\frac{1}{2}} \od\!g, \quad \text{for } v\in \pi(\xi).
\end{equation}
It converges absolutely for $\Re(s)$ sufficiently large and extends to a holomorphic multiple of $\oL(s, \pi(\xi)\otimes\chi)$ on the complex plane. By definition,
if $f\in I(\xi)$ has image $v\in \pi(\xi)$ under $p_\xi$, then 
\begin{equation}\label{FJ-I}
\oZ_{\mathrm{FJ}}(s, v, \chi) = \oZ_{\mathrm{FJ}}(s, f, \chi):=
\int_{G_n} \left\langle \lambda_{I(\xi)}, \begin{bmatrix} g & 0 \\ 0 & 1_n \end{bmatrix} \cdot f \right\rangle \,
\chi(g) |g|_\bk^{s-\frac{1}{2}} \od\!g.
\end{equation}
Note that in this expression of the local Friedberg--Jacquet zeta integrals, the local Shalika functional $\lambda_{I(\xi)}$ is defined in Part (1) of Proposition~\ref{prop:Sha}, in terms of the local integral defined by the open-orbit method.

We introduce certain auxiliary integrals. For $f\in I(\xi)$, $\phi\in \CS(\bk^n)$ and $(t, s)\in \BC^2$, formally define the integral 
\be \label{Zts}
\oZ(t, s; f, \phi, \chi) :=\int_{G_n} \Lambda_{\rm JS}\left(t, \begin{bmatrix} g & 0 \\ 0 & 1_n\end{bmatrix}\cdot f, \phi, \varphi^{-1}\right) \chi(g)\abs{g}_{\bk}^{s-\frac{1}{2}}\od\!g.
\ee

\begin{leml}\label{lem:Zts}
    Let $U\subset \BC$ be an open bounded subset. Then there exists a nonzero holomorphic function $a_U$ on $U$ and a constant $C_U\in \BR$ such that $a_U(t) \oZ(t, s; f, \phi,\chi)$ is  holomorphic on 
    \[
    (t, s)\in \Omega_U:=U\times \set{s\in \BC | \Re(s)>C_U}
    \]
     defined by absolutely convergent integrals, for all $f\in I(\xi)$ and $\phi\in \CS(\bk^n)$. Moreover, $a_U(t)\oZ(t,s;f,\phi,\chi)$ is continuous in $f$ and $\phi$.
\end{leml}

\begin{proof}
    This easily follows from the proof of \cite[Lemma 6.1]{JR}, \cite[Lemma 3.4]{AGJ09}, and \cite[Proposition 3.1]{FJ93}.
\end{proof}

We now introduce another type of integral, whose comparison with the Friedberg–Jacquet integral yields the modifying factors for standard $L$-functions of symplectic type. To this end, we first introduce a certain Rankin–Selberg period.
Let
\[ 
H := \Set{ \begin{bmatrix} g_1 \\ & g_2 \end{bmatrix} | g_1, g_2\in G_n}.
\]
For every standard section $\{f_{\xi'}\in I(\xi')\}_{\xi'\in \CM}$ and every  function $\phi \in \CS(\bk^n)$, by considering the restriction of $f_{\xi'}$ to $H$, it follows easily from \cite[Proposition 1.4, Theorem 1.6 (a)]{LLSS23} that
\begin{eqnarray} 
\Omega_\eta\cap (\BC\times\CM)&\rightarrow & \BC,\nonumber \\
(t,\xi')&\mapsto & \Lambda_{\rm RS}(t, f_{\xi'}, \phi,\eta^{-1}):= \int_{G_n} f_{\xi'} \left(z_0\begin{bmatrix} g \\ & g \end{bmatrix}\right)  \phi(v_ng) \eta^{-1}(g) |g|_\bk^t \od\!g  \label{RSint}
\end{eqnarray}
is a holomorphic function defined by absolutely convergent integrals and has a meromorphic continuation to $\BC\times\CM$. Here $z_0$ is defined in \eqref{zm} and $\Omega_\eta$ is defined in \eqref{minmax'}.

Let $\overline{Q}$ be the block lower triangular maximal parabolic subgroup of $G$ with Levi subgroup $H$.  Then the right action of $H$ on the Grassmannian 
$ \overline{Q} \bsl G $
has a unique open orbit with a base point $\overline{Q} \gamma$, where 
\be \label{gamman}
\gamma:=  \begin{bmatrix} 1_n & 1_n \\ 0 & 1_n \end{bmatrix},
\ee
and the stabilizer of $\overline{Q} \gamma$ in $H$ is  
$
\Delta G_n: = S\cap H.
$
 
Consider the following space
\be\label{sharp}
I(\xi)^\sharp:=\set{f\in I(\xi) | {\rm supp}(f)\subset \overline{Q} \gamma H}.
\ee
For $f\in I(\xi)^\sharp$, $\phi\in \CS(\bk^n)$ and $(t, s)\in \BC^2$, formally define the integral 
\[
\Lambda(t, s; f,\phi, \chi):= \int_{G_n} \Lambda_{\rm RS} \left(t, \gamma \begin{bmatrix} g  \\ & 1_n\end{bmatrix}\cdot f, \phi,\eta^{-1}\right) \chi(g)|g|_\bk^{s-\frac{1}{2}}\od\! g. 
\]
By the support condition, it is easy to see that for an open bounded subset $U\subset \BC$, there is a nonzero holomorphic function $b_U$
on $U$ such that $b_U(t)\Lambda(t, s; f,\phi, \chi)$ is holomorphic on $(t, s)\in U\times \BC$ defined by absolutely convergent integrals,
for all $f\in I(\xi)^\sharp$ and $\phi\in\CS(\bk^n)$. Moreover, it is continuous in $f$ and $\phi$.
The following is our main result on the local Friedberg-Jacquet integrals and the corresponding modifying factors. 

\begin{thml}\label{MFSha}
Assume that $\pi(\xi)$ satisfies the conditions of Theorem \ref{cor:sha} and is a quotient of $I(\xi)$. Let $f\in I(\xi)^\sharp$ and $\phi \in \CS(\bk^n)$.
\begin{enumerate} 

\item 
We have
\be\label{modfj}
\Lambda(t, s; f, \phi, \chi) =\left( \prod^n_{i=1} \gamma(s, \xi_i\chi, \psi)\right) \cdot \oZ(t, s; f, \phi, \chi),
\ee
where both sides have meromorphic continuations to $(t,s)\in\BC^2$.

\item As meromorphic functions of $s$, we have 
\[
\left(t^{d_\xi} \Lambda(t, s; f, \phi, \chi)\right)_{t=0} = \left( \prod^n_{i=1} \gamma(s, \xi_i\chi, \psi)\right) \cdot \oZ_{\rm FJ}(s, f,\chi)\cdot \phi(0).
\]
\end{enumerate}
\end{thml}

\begin{proof} Let
\[
\xi^{1}:=(\xi_1,\xi_2,\ldots, \xi_n) \quad \textrm{and} \quad \xi^{2}:=(\xi_{n+1}, \xi_{n+2}, \ldots, \xi_{2n}).
\]
 By induction in stages, we have an isomorphism
 \[
 \begin{aligned}
I(\xi) & \xrightarrow{\sim} \Ind^{G}_{\overline{Q}}(I(\xi^1)\,\widehat\otimes\, I(\xi^2)),\\
 f & \mapsto \left( f': G \to I(\xi^1)\,\widehat\otimes\, I(\xi^2),\ g\mapsto ( h\mapsto \delta_{\overline{Q}}^{-1/2}(h)f(hg), \ h\in H)\right),
 \end{aligned}
 \]
where $\delta_{\overline{Q}}$ is the modular character of $\overline{Q}$.  
Similar to \eqref{RSint}, for $f_i\in I(\xi^i)$, $i=1,2$ and $t\in \BC$, define the integral
\[
\Lambda'_{\rm RS}(t, f_1\otimes f_2, \phi, \eta^{-1}) = \int_{G_n} f_1(g) f_2(w_n g) \phi(v_n g)\eta^{-1}(g)\abs{g}_\bk^t\od\!g
\]
in the sense of meromorphic continuation. This extends to a meromorphic family of $\Delta G_n$-equivariant linear functionals 
on $I(\xi^1)\,\widehat\otimes\,I(\xi^2)\,\widehat\otimes\,\CS(\bk^n)$ such that 
\[
\Lambda_{\rm RS}(t, f, \phi, \eta^{-1}) = \Lambda_{\rm RS}'(t, f'(1_{2n}), \phi, \eta^{-1}).
\]

Assume that $f\in I(\xi)^\sharp$. Let $U\subset \BC$ be an open bounded subset, and  
let $\Omega_U$ be as in Lemma \ref{lem:Zts}. By definition, 
\[
\Lambda_{\rm JS}(t, f, \phi, \varphi^{-1}) = \int_{M_n} \Lambda'_{\rm RS} \left(t, f'\left(\begin{bmatrix} 1_n & X \\ & 1_n \end{bmatrix}\right), \phi,\eta^{-1}\right) \overline\psi({\rm tr}\, X)\od\! X
\]
as meromorphic functions of $t\in U$. Then for $(t,s)\in \Omega_U$ we have
\[
\begin{aligned}
 &\oZ(t, s; f, \phi,\chi) \\
= \, &   \int_{G_n} \int_{M_n} \Lambda'_{\rm RS}\left(t, f'\left(\begin{bmatrix} 1_n & X\\ & 1_n \end{bmatrix} \begin{bmatrix} g & \\ & 1_n\end{bmatrix}\right),\phi,\eta^{-1}\right)
 \overline\psi({\rm tr}\, X)\od\! X \, \chi(g)|g|_\bk^{s-\frac{1}{2}}\od\! g \\
= \, &  \int_{G_n} \int_{M_n} \Lambda'_{\rm RS}\left(t, I(\xi^1)(g)\cdot f'\left(\begin{bmatrix} 1_n & X\\ & 1_n \end{bmatrix}\right),\phi,\eta^{-1}\right)
 \overline\psi({\rm tr}\, gX)\od\! X \, \chi(g)|g|_\bk^{s+\frac{n-1}{2}}\od\! g,
\end{aligned}
\]
where we change the variable $X\mapsto gX$ in the last step, and $I(\xi^1)(g)$ acts on the first factor of $I(\xi^1)\,\widehat\otimes\, I(\xi^2)$.

Define the smooth compactly supported function
\[
F_f: M_n \to I(\xi^1)\,\widehat\otimes\, I(\xi^2),\qquad X\mapsto f'\left(\begin{bmatrix}1_n & X \\ & 1_n\end{bmatrix}\right),
\]
and the function 
\[
\Phi_t(g, X) :=\Lambda'_{\rm RS}\left(t, I(\xi^1)(g)\cdot F_f(X),\phi,\eta^{-1}\right) \chi(g),\quad (g, X)\in G_n\times M_n.
\]
When $\bk$ is archimedean, for a regular value of $t\in U$ there exist an integer $M\geq 0$ and a continuous seminorm $p$ on $I(\xi^1)\,\widehat\otimes\, I(\xi^2)$ only depending on $t$ and $\phi$ such that
\be \label{est}
|D_X\Phi_t(g,X)| \leq (1+\|g\|+\|g^{-1}\|)^M p(D_XF_f(X)),\quad (g,X)\in G_n\times M_n
\ee
for every constant coefficient differential operator $D_X$ in $X$. Here $\|\cdot\|$ is a fixed Euclidean norm on $M_n$.  
Using \eqref{est} and integration by parts,  for every integer $N\geq 0$ and every compact subset $K\subset M_n$ contained in $G_n$, there exists a continuous seminorm $p_{K,N}$ on $C^\infty_K(M_n, I(\xi^1)\,\widehat\otimes\, I(\xi^2))$ (smooth functions supported in $K$) such that 
\[
\int_{M_n} | \Phi_t(g,X)|\od\!X \leq p_{K,N}(F_f)(1+\|g\|)^{-N}(1+\|g\|+\|g^{-1})^M,\quad g\in G_n
\]
when ${\rm supp}(F_f)\subset K$. It follows easily that
\[
\oZ(t, s; f, \phi,\chi) = \int_{G_n} \int_{M_n} \Lambda'_{\rm RS}\left(t, I(\xi^1)(g)\cdot F_f(X),\phi,\eta^{-1}\right)
 \overline\psi({\rm tr}\, gX)\od\! X \, \chi(g)|g|_\bk^{s+\frac{n-1}{2}}\od\! g
\]
is continuous in $F_f\in C^\infty_K(M_n, I(\xi^1)\,\widehat\otimes\, I(\xi^2))$. The nonarchimedean case is similar.

Since $\Lambda(t,s; f,\phi,\chi)$ is holomorphic in $s$, by the above continuity and meromorphic continuation, we may assume that 
$\Phi_t\in {\rm MC}(I(\xi^1)\otimes\chi)\otimes C^\infty_c(M_n)$, where ${\rm MC}(I(\xi^1)\otimes\chi)$ denotes the space spanned by the generalized matrix coefficients 
\[
\langle \ell, g\cdot v\rangle, \quad v\in I(\xi^1)\otimes\chi, \quad \ell\in \left(I(\xi^1)\otimes\chi\right)' \ (\text{continuous linear dual}).
\]
Then the above inner integral over $M_n$ equals $\CF_{\overline\psi}(\Phi_t)(g, g)$, where $\CF_{\overline\psi}$ indicates the Fourier transform in the variable $X$ with respect to $\overline\psi$. 
Thus $\oZ(t, s; f, \phi,\chi)$ can be viewed as a Godement-Jacquet integral (\cite{GJ72, AGJ09}) for the representation $I(\xi^1)\otimes\chi$ of $G_n$. By the meromorhic continuation and functional equation of Godement-Jacquet integrals (\cite[Theorem 4.3]{AGJ09}), for $-\Re(s)$ sufficiently large we have
\[
\begin{aligned}
 & \gamma(s, I(\xi^1)\otimes\chi, \psi) \oZ(t,s; f, \phi, \chi)  \\
=\, &\int_{G_n} \Phi_t(g^{-1}, g) |g|_\bk^{\frac{n+1}{2}-s}\od\! g \\
= \, & \int_{G_n} \Lambda_{\rm RS}'\left(t, I(\xi^1)(g^{-1})\cdot f'\left(\begin{bmatrix} 1_n & g \\ & 1_n \end{bmatrix}\right), \phi,\eta^{-1}\right)  \chi(g^{-1}) |g|_\bk^{\frac{n+1}{2}-s}\od\! g\\
=\, &\int_{G_n}  \Lambda_{\rm RS}'\left(t, f'\left(\begin{bmatrix} g^{-1} & 1_n \\ & 1_n\end{bmatrix}\right),\phi, \eta^{-1}\right) \chi(g^{-1}) |g|_\bk^{\frac{1}{2}-s}\od\! g. 
\end{aligned}
\]
By the change of variable $g^{-1}\mapsto g$, the last integral is equal to 
\[
 \int_{G_n} \Lambda_{\rm RS}'\left(t, f'\left(\gamma \begin{bmatrix} g\\ & 1_n\end{bmatrix}\right), \phi, \eta^{-1} \right) \chi(g) |g|_\bk^{s-\frac{1}{2}}\od\! g = \Lambda(t, s; f, \phi, \chi).
\]
It follows that 
$\gamma(s, I(\xi^1)\otimes\chi, \psi) \oZ(t, s; f, \phi, \chi)   = \Lambda(t, s; f, \phi, \chi) $ 
for all $s\in\BC$ by the uniqueness of meromorphic continuation. By varying $U$, this proves Part (1). 

Taking $U$ to be an open bounded neighborhood of $0$, Part (2) easily follows from Part (1), Proposition \ref{prop:Sha} and the dominated convergence theorem. 
\end{proof}

\subsection{Restriction to the open orbit}
For later applications, we also need the following result. The right action 
\be \label{Hact}
\CX= \CB\times \bk^n \curvearrowleft H, \quad (\overline{B}g, v)\cdot \begin{bmatrix} g_1 & 0 \\ 0 & g_2 \end{bmatrix} = 
\left(\overline{B} g \begin{bmatrix} g_1 & 0 \\ 0 & g_2 \end{bmatrix}, vg_2\right)
\ee
has a unique open orbit 
\[
\CO:= (\overline{B} \gamma_0, v_n)\cdot H
\]
with trivial stabilizer, where 
\be \label{gamma0}
\gamma_0:=\begin{bmatrix} 1_n & 1_n \\ 0 & w_n \end{bmatrix}.
\ee
We realize $I(\xi)$ as the space of Schwartz sections of the corresponding 
line bundle $\CL_\xi$ on $\CB$. Let $\CL'_\xi$ be the pullback of $\CL_\xi$ to $\CX$, so that 
$I(\xi)\,\widehat\otimes\,\CS(\bk^n)=\Gamma^\varsigma(\CX,\CL'_\xi)$ is the space of Schwartz sections of $\CL'_\xi$, on which \eqref{Hact} induces an action of $H$. Denote by 
\[
\Gamma_c^\infty(\CO, \CL'_\xi|_{\CO})
\]
the space of smooth sections of $\CL'_\xi|_\CO$ with compact support. Note that 
\be \label{incl}
\Gamma_c^\infty(\CO, \CL'_\xi|_{\CO}) \subset I(\xi)^\sharp\,\widehat\otimes\,\CS(\bk^n).
\ee

Fix characters $\eta$ and $\chi$ of $\bk^\times$. For $(t,s)\in\BC^2$, define the following character 
\[
\theta_{t,s}:= \chi^{-1}\abs{\cdot}_\bk^{\frac{1}{2}-s} \otimes \eta\chi\abs{\cdot}_\bk^{s-\frac{1}{2}-t}: H\to \BC^\times.
\]

\begin{leml} \label{lem:res}
    Let the notation be as above. Then there exist open dense subsets $\Omega_1, \Omega_2\subset\BC$ such that the restriction map 
    \[
    {\rm res}_{t,s}: \Hom_H(\Gamma^\varsigma(\CX, \CL'_\xi),\theta_{t,s}) \to \Hom_H( \Gamma_c^\infty(\CO, \CL'_\xi|_\CO), \theta_{t,s})
    \]
    is injective for every $(t,s)\in \Omega_1\times \Omega_2$. In particular, 
    \[
    \dim \Hom_H(I(\xi)\,\widehat\otimes\,\CS(\bk^n), \theta_{t,s})\leq 1,\qquad (t,s)\in \Omega_1\times\Omega_2.
    \]
\end{leml}

\begin{proof}
    The method of the proof is standard, and we give a sketch as follows. First of all, there are only finitely many $H$-orbits in $\CX$. 
    As is well-known, there are only finitely many $H$-orbits in $\CB$. Realize $\CB$ as the set of complete flags $\CF_\bullet:= (\{0\}\subset \CF_1\subset\cdots\subset \CF_{2n} = \bk^{2n})$. 
    Write $\bk^{2n} = X \oplus Y$, where $X$ (resp. $Y$) is spanned by the first (resp. last) standard basis vectors, so that 
    $\CX=\CB\times Y$ and $H=\GL(X)\times \GL(Y)$.
    Following \cite[Section 2]{CPSU}, there exist bases 
    $x_1,x_2,\ldots, x_n$ and $y_1, y_2,\ldots, y_n$ of $X$ and $Y$ respectively, an integer $m\in \{0,1,\ldots, n\}$ and an ordering 
    $z_1,z_2,\ldots,z_{2n}$ of the basis 
    \be \label{basis}
    x_1+y_1, x_1-y_1, \ldots, x_m+y_m, x_m - y_m, x_{m+1}, x_{m+2}, \ldots, x_n, y_{m+1}, y_{m+2}, \ldots, y_n
    \ee
    such that $x_i+y_i$ occurs before $x_i-y_i$, $i=1,2,\ldots,n$ and $\CF_j = {\rm Span}\{z_1, z_2, \ldots, z_i\}$, $j=1,2,\ldots, 2n$. It follows that the projection of $H_{\CF_\bullet}$ (the stabilizer of $\CF_{\bullet}$ in $H$)
    to $\GL(Y)$ contains a maximal split torus. This torus acts on $Y$ with $2^n$ orbits. This proves the finiteness of $H$-orbits in $\CX$.

    Moreover, $\CO$ is the orbit of $(\CF_\bullet, y)$, where $\CF_\bullet$ is associated with the ordered basis 
    $x_1+y_1, x_2+y_2, \ldots, x_n+y_n, x_n-y_n,\ldots, x_2-y_2, x_1-y_1$ and $y= y_1+y_2+\cdots+y_n$. 
    Assume that $\CO' \neq \CO$ is the orbit of a point $x= (\CF_\bullet', y')\in \CX$, where $\CF_\bullet'$ is associated to an ordering of a basis 
    of the form \eqref{basis} as above. Next we show case-by-case that there 
    exists an algebraic homomorphism
    \[
     \gamma_x: \bk^\times \to H_x \quad (\text{the stabilizer of $x$ in $H$})
    \]
    such that 
    \be \label{theta}
    \theta_{t,s}\circ \gamma_x = \chi^{-1}\abs{\cdot}_\bk^{\frac{1}{2}-s} \quad \text{or} \quad \eta\abs{\cdot}_\bk^{-t}.
    \ee

    {\bf Case 1.} $m<n$. Define $\gamma_x(a)\in H_x$, $a\in \bk^\times$ such that 
    $x_n \gamma_x(a)= ax_n$ and $\gamma_x(a)$ fixes every other $x_j$ and every $y_j$. Then $\theta_{t,s}\circ \gamma_x = \chi^{-1}\abs{\cdot}_\bk^{\frac{1}{2}-s}$.

    {\bf Case 2.} $m=n$, and $y'=\sum^n_{i=1}c_i y_i$ with some $c_i=0$. Define 
    $\gamma_x(a)\in H_x$, $a\in \bk^\times$ such that $x_i \gamma_x(a)= ax_i$, $y_i \gamma_x(a)= ay_i$, and $\gamma_x(a)$ fixes all other $x_j$ and $y_j$.
    Then $\theta_{t,s}\circ \gamma_x = \eta\abs{\cdot}_\bk^{-t}$.

    {\bf Case 3.} $m=n$, and $y'=\sum^n_{i=1}c_i y_i$ where all $c_i\neq 0$. We may assume that exist $1\leq i<j\leq n$ such that $x_i+y_i$ occurs before $x_j+y_j$
    and $x_i-y_i$ occurs before $x_j-y_j$ in the ordering $z_1,z_2,\ldots, z_{2n}$. Define $\gamma_x(a)\in H_x$, $a\in \bk^\times$ such that
    \[
    \begin{aligned}
    & x_i\gamma_x(a) = a x_i, \quad x_j \gamma_x(a) = x_j+ \frac{c_i}{c_j}(1-a) x_i, \\
    & y_i\gamma_x(a) = a y_i, \quad y_j \gamma_x(a) = y_j+ \frac{c_i}{c_j}(1-a) y_i, 
    \end{aligned}
    \]
    and $\gamma_x(a)$ fixes all other $x_k$ and $y_k$. Then $\theta_{t,s}\circ \gamma_x = \eta\abs{\cdot}_\bk^{-t}$.
    
    This gives an algebraic homomorphism $\gamma_x: \bk^\times \to H_x$ satisfying \eqref{theta}. Now we finish the proof of the lemma. First assume that $\bk$ is archimedean. Let 
    \[
    N_x^* = \left((T_x\CX/T_x\CO')\otimes_\BR\BC\right)^*\quad (\text{$T_x$ denotes the tangent space})
    \]
    be the complexified conormal space. By \cite[Theorems 2.5.6, 2.5.7]{AG09}, if $\ker({\rm res}_{t,s})\neq \{0\}$, then there exist an orbit $\CO'\neq\CO$ as above and an integer $r\geq 0$ such that 
    \be \label{stabilizer}
    \Hom_{H_x}(\CL'_{\xi, x}\otimes {\rm Sym}^r(N_x^*) \otimes \delta_{H_x}^{-1}, \theta_{t,s})\neq \{0\},
    \ee
    where ${\rm Sym}^r$ denotes the $r$-th symmetric power and $\delta_{H_x}$ is the modular character of $H_x$. 
    In the nonarchimedean case, if $\ker({\rm res}_{t,s})\neq \{0\}$, then there exists an orbit $\CO'\neq\CO$ as above such that
    \be \label{stabilizer2}
    \Hom_{H_x}(\CL_{\xi, x}\otimes \delta_{H_x}^{-1}, \theta_{t,s})\neq \{0\}.
    \ee
    The lemma easily follows from \eqref{theta}, \eqref{stabilizer} and \eqref{stabilizer2}.
\end{proof}

\section{Archimedean Period Relations} \label{sec:PAPR}

In this section, we assume that $\bk$ is  archimedean and apply Theorem \ref{MFSha} and Lemma \ref{lem:res} to establish the archimedean period relations for standard $L$-functions of symplectic type for $G$. 

We set up some notations concerning cohomological representations and refer to \cite{Cl90} and \cite{LLS24} for more details. 
Denote by $\CE_\bk$ the set of continuous field embeddings 
$\iota: \bk\hookrightarrow\BC$.  
For every real Lie subgroup $L$ of $G$, denote by $L_\BC$ the Zariski closure of $L$ in  $G_{\BC}= \GL_{2n}(\bk\otimes_\BR\BC)$.  
Let $\mu = (\mu^\iota)_{\iota\in\CE_\bk} \in (\BZ^{2n})^{\CE_\bk}$ be a dominant weight for $G_\BC$, where 
\[
\mu^\iota=(\mu^\iota_1\geq  \mu^\iota_2\geq \ldots\geq \mu^\iota_{2n})\in \BZ^{2n}.
\]
Let $F_\mu$ denote an irreducible algebraic representation 
of $G_{\BC}$ with highest weight $\mu$. Then there exists a unique (up to isomorphism) generic irreducible Casselman-Wallach representation 
$\pi_\mu$ of $G$ such that the total continuous cohomology 
\[
\oH^*_{\rm ct}(\BR^\times_+\bsl G^0; \pi_\mu\otimes F_\mu^\vee)\neq \{0\},
\]
where $\BR^\times_+$ is the split component of the center of $G$. 

Now we assume that $\mu$ is 
pure  in the sense of \cite[Lemma 4.9]{Cl90}, namely there exists $w_0\in \BZ$ such that 
\be\label{pure1}
\mu_1^\iota+\mu_{2n}^{\bar \iota}=\mu_2^\iota+\mu_{2n-1}^{\bar \iota}=\ldots=\mu_{2n}^\iota+\mu_{1}^{\bar \iota}=w_0\qquad \textrm{for all }\iota\in \CE_\bk,
\ee
where $\bar \iota\in \CE_\bk$ denotes the complex conjugation of $\iota$. Then $\pi_\mu$ is essentially tempered. Further assume that $\mu$ is of symplectic type, namely  there 
exists a family $\{w_\iota\in\BZ\}_{\iota\in \CE_\bk}$ such that 
\be\label{pure2}
\mu^\iota_1+\mu^\iota_{2n} = \mu^\iota_2 + \mu^\iota_{2n-1} = \cdots = \mu^\iota_n + \mu^\iota_{n+1} = w_\iota \qquad \textrm{for all }\iota\in \CE_\bk.
\ee
Put 
\[
\eta_\mu:=\prod_{\iota\in\CE_\bk} (\iota|_{\bk^\times})^{w_\iota}\in \widehat{\bk^\times}.
\]
As is well-known, $\pi_\mu\otimes |\eta_\mu|^{-\frac{1}{2}}$ is tempered. Let $\psi$  be the nontrivial unitary character of $\bk$ given by 
\[
\psi(x) := \exp\left(2\pi {\rm i} \sum_{\iota\in\CE_\bk}\iota(x)\right),\quad x\in \bk. 
\]

\begin{leml}  \label{etamu}
    Let the assumptions be as above. Then $\eta= \eta_\mu$ is the unique character of $\bk^\times$ such that 
$\Hom_{S}(\pi_\mu,\varphi_{\eta})\neq\{0\}$, where  $\varphi_{\eta}$ is the character of $S$ associated to $\eta$ and $\psi$ as in \eqref{varphi2n}.
\end{leml}

\begin{proof}
    As in Section \ref{ssec2.1}, denote by $\phi_{\pi_\mu}: \CW_{\bk}\to \GL_{2n}(\BC)$ the $L$-parameter of $\pi_\mu$, where $\CW_\bk$ is the Weil group of $\bk$.
    By \cite[Theorem 1.3]{G25}, for $\eta\in \widehat{\bk^\times}$ we have $\Hom_S(\pi_\mu, \varphi_\eta)\neq \{0\}$ if and only if 
    there is a symplectic form $\langle\cdot,\cdot\rangle$ on $\BC^{2n}$ such that 
    \[
    \langle \phi_{\pi_\mu}(w)(x), \phi_{\pi_\mu}(w)(y)\rangle = \eta(w)\cdot\langle x, y\rangle
    \]
    for all $w\in \CW_\bk$ and $x,y\in \BC^{2n}$, where $\eta$ is viewed as a character of $\CW_{\bk}$ via local class field theory. Then the lemma follows easily from
    the explicit $L$-parameter $\phi_{\pi_\mu}$ given by the realization of $\pi_\mu$ as a standard module in \cite[Section 3.5]{Cl90}.
\end{proof}

Write $\varphi_\mu:= \varphi_{\eta_\mu}$, and fix a generator  $\lambda_{\pi_\mu}$ of the one-dimensional space $\Hom_S(\pi_\mu, \varphi_\mu)$, which is called a non-zero twisted Shalika functional of $\pi_\mu$.  Similar to \eqref{inftype}, assume that $\chi$ is a character of $\bk^\times$ of the form
$\chi = \chi_\natural |_{\bk^\times} \cdot \chi^\natural$, 
where $\chi_\natural = \bigotimes_{\iota\in \CE_{\bk}}\iota^{\od\!\chi_{\iota}}$, $\od\!\chi_{\iota}\in \BZ$, and $\chi^\natural$ is quadratic. 
For any given and fixed non-zero Shalika functional $\lambda_{\pi_\mu}$, as in \eqref{FJint},  we have the 
normalized Friedberg-Jacquet integral
\be \label{NFJ}
\oZ_{\rm FJ}^\circ(s, v, \chi) := \frac{ \oZ_{\rm FJ}(s, v,\chi)}{\oL(s, \pi_\mu\otimes\chi)},\quad v\in \pi_\mu,
\ee
which is holomorphic and non-zero for all $s\in\BC$.

Consider the principal series representation 
$
I_\mu :=\Ind^{G}_{\overline{B}}(\chi_\mu \rho),
$
 where 
 \[
 \chi_\mu := \left(\prod_{\iota\in\CE_\bk}(\iota|_{\bk^\times})^{\mu^\iota_1}, \dots, \prod_{\iota\in\CE_\bk} (\iota|_{\bk^\times})^{\mu^\iota_{2n}}\right)\in (\widehat{\bk^\times})^{2n},
 \]
and $\rho$
is the square root of the modular character of the upper triangular Borel subgroup of $G$ (viewing as a character of the diagonal torus by restriction).  
By \cite[Lemma 2.2]{LLS24}, $I_\mu$ has a unique irreducible quotient  which is isomorphic to $\pi_\mu$. 
Let $\lambda_{I_\mu}\in\Hom_S(I_\mu, \varphi_{ \mu})$ be the functional supplied by Proposition \ref{prop:Sha}, so that there is a unique map  
$p_\mu \in \Hom_{G}(I_\mu, \pi_\mu)$ such that $\lambda_{\pi_\mu}\circ p_\mu = \lambda_{I_\mu}$, i.e. the triple $(\lambda_{I_\mu},\lambda_{\pi_\mu},p_\mu)$ is a Shalika datum. 

All the above discussions apply to the zero weight $\mu =0$ case. In such a case the coefficient system $F_0:=\BC$ is the trivial representation.  Realize $F_\mu$ as the algebraic induction ${}^{\rm alg}\Ind^{G_{\BC}}_{\overline{B}_{\BC}}\chi_\mu$. This yields an explicit translation map 
$
\imath_\mu \in \Hom_G(I_0, I_\mu\otimes F_\mu^\vee)
$
as in \cite[Section 2.2, Formula (19)]{LLS24}. 
By \cite[Lemma 2.4 and Formula (22)]{LLS24}, there is a unique morphism 
$
\jmath_\mu \in \Hom_G(\pi_0, \pi_\mu\otimes F_\mu^\vee)
$
that makes the following diagram commutative:
\be\label{CDtrans}
\xymatrix{
I_0 \ar@{^(->}[r]^(0.4){\imath_\mu} \ar@{>>}[d]_{p_0} & I_\mu\otimes F_\mu^\vee \ar@{>>}[d]^{p_\mu\otimes {\rm id}} \\
\pi_0 \ar@{^(->}[r]^(0.4){\jmath_\mu} & \pi_\mu \otimes F_\mu^\vee 
}
\ee
Define the character 
$
\xi_{\mu, \chi}:=\chi  \boxtimes  (\chi^{-1}\eta_\mu^{-1})
$
of $H\cong G_n\times G_n$. Similar to \eqref{char:xi},  define the character 
$\xi_{\mu,\chi_\natural} := \otimes_{\iota\in\CE_\bk} ( {\det}^{\od\!\chi_{\iota}}\boxtimes  {\det}^{-\od\!\chi_{\iota}-w_\iota})$ 
of $H_{\BC}\cong G_{n,\BC}\times G_{n,\BC}$. Note that 
$
\xi_{\mu, \chi} \otimes \xi_{\mu, \chi_\natural}^\vee|_H = \chi^\natural \boxtimes \chi^\natural
$
only depends on $\chi^\natural$.

Assume that  $\chi_\natural$ is  $F_\mu$-balanced in the sense of Definition \ref{def:bal}. It is known that (see \cite[Theorem 7.2]{Sc84})
$\dim \Hom_{H_{\BC}}(F_\mu^\vee, \xi_{\mu, \chi_\natural})=1$, and 
we specify a generator of $\Hom_{H_{\BC}}(F_\mu^\vee, \xi_{\mu, \chi_\natural})$ as follows. 
Let $v_\mu^\vee\in (F_\mu^\vee)^{\overline{N}_{\BC}}$ be the lowest weight vector specified as in \cite[Section 2.1]{LLS24}, where 
$\overline{N}$ is the unipotent radical of $\overline{B}$. Recall the element $\gamma_0$ in \eqref{gamma0}.

\begin{lem} \label{balanced-vec}
There exists a unique generator
$\lambda_{F_\mu, \chi_\natural} \in \Hom_{H_{\BC}}(F_\mu^\vee, \xi_{ \mu,\chi_\natural})
$
with the property that $\lambda_{F_\mu,  \chi_\natural}(\gamma_0^{-1}. v_\mu^\vee) =1$.
\end{lem}

\begin{proof}
Recall the  realization  $F_\mu={}^{\rm alg}\Ind^{G_{\BC}}_{\overline{B}_{\BC}}\chi_\mu$. 
From the fact that 
$\overline{B}_{\BC}\, \gamma_0 \, H_{\BC}\subset G_{\BC}$ 
is Zariski  open dense, by restriction we have an embedding $F_\mu \hookrightarrow \BC[\overline{B}_{\BC}\, \gamma_0 \, H_{\BC}]$ (the space of algebraic functions).
Then the lemma follows easily. 
\end{proof}

Let
$
\lambda_{F_\mu, \chi_\natural} \in \Hom_{H_{\BC}}(F_\mu^\vee, \xi_{\mu, \chi_\natural})
$
be as in Lemma \ref{balanced-vec}. 
The functional $\oZ^\circ_{\rm FJ}(\frac{1}{2}, \cdot, \chi)\otimes \lambda_{F_\mu, \chi_\natural}$  induces the archimedean modular symbol 
\be \label{archMS}
\wp_{\mu, \chi}: \oH^{d_\bk}_{\rm ct}(\BR^\times_+\bsl G^0; \pi_\mu\otimes F_\mu^\vee)\otimes \oH^0_{\rm ct}(\BR^\times_+\bsl H^0; \xi_{\mu,\chi}\otimes \xi_{\mu, \chi_\natural}^\vee)
\to \oH^{d_\bk}_{\rm ct}(\BR^\times_+\bsl  H^0; \BC),
\ee
where
\be \label{dk}
d_\bk:=\begin{cases}n^2+n-1, & \textrm{if }\bk\cong\BR,\\
2n^2-1, & \textrm{if }\bk\cong\BC.
\end{cases}
\ee

In the next section, we will prove that the archimedean modular symbol $\wp_{\mu,\chi}$ in \eqref{archMS} is non-vanishing (Theorem \ref{thm:NV}).  
The main result of this section is the following archimedean period relation, which is an application of Theorem \ref{MFSha}.

\begin{thml}[Archimedean Period Relation] \label{APR}
Let the notation and assumption be as above. Then the diagram
\[
\begin{CD}
 \oH^{d_\bk}_{\rm ct}(\BR^\times_+\bsl G^0; \pi_\mu\otimes F_\mu^\vee)\otimes \oH^0_{\rm ct}(\BR^\times_+\bsl H^0; \xi_{\mu,\chi}\otimes \xi_{\mu, \chi_\natural}^\vee) @>\Omega_{\mu,\chi_\natural} \cdot\wp_{\mu,\chi}>>  \oH^{d_\bk}_{\rm ct}(\BR^\times_+\bsl  H^0; \BC) \\
 @A  \jmath_\mu\otimes {\rm id} AA  @| \\
  \oH^{d_\bk}_{\rm ct}(\BR^\times_+\bsl G^0; \pi_0)\otimes \oH^0_{\rm ct}(\BR^\times_+\bsl H^0; \xi_{0,\chi^\natural}) @> \wp_{0,\chi^\natural}>>  \oH^{d_\bk}_{\rm ct}(\BR^\times_+\bsl  H^0; \BC)
\end{CD}
\]
commutes, where
$\Omega_{\mu,\chi_\natural}: ={\rm i}^{ \sum_{\iota\in\CE_\bk}\sum^n_{i=1} (\mu^\iota_i+\od\!\chi_{\iota})}.
$
\end{thml}

The rest of this section is devoted to the proof of Theorem \ref{APR}. Write 
\[
\zeta_\mu:= \chi_\mu\rho = (\zeta_{\mu,1}, \zeta_{\mu,2},\dots, \zeta_{\mu,  2n}) \in (\widehat{\bk^\times})^{2n},
\]
 so that $I_\mu = I(\zeta_\mu)$ as in \eqref{PS}. 
Define 
\[
\oZ_{\rm FJ}^\diamond (s, f, \chi):=  \frac{\oZ_{\rm FJ}(s, f, \chi)}{\oL(s, \pi_\mu\otimes\chi)},\quad f\in I_\mu,
\]
which is holomorphic and non-vanishing on $I_\mu$ for each $s\in\BC$. Put
\[
\Xi_{\mu, \chi_\natural}(s):=  \left(\prod^n_{i=1} \frac{\gamma(s,  \zeta_{0, i} \cdot \chi^\natural, \psi)}{\gamma(s, \zeta_{\mu, i}\cdot \chi, \psi)}\right) \cdot \frac{\oL(s, \pi_0)}{\oL(s, \pi_\mu\otimes\chi)},
\]
which a priori depends on $\chi^\natural$ (in the real case) and is meromorphic.  Similar to the proof of \cite[Proposition 4.7]{LLS24}, 
we can prove the following lemma. 

\begin{leml}
As a meromorphic function of  $s\in \BC$, 
$\Xi_{\mu,\chi_\natural}(s) \equiv \Omega_{\mu, \chi_\natural}^{-1},
$
where $\Omega_{\mu,\chi_\natural}$ is the constant in Theorem \ref{APR}.
\end{leml}

\begin{proof}
    We give the proof in the complex case, and omit the real case which is similar by using the standard results for the archimedean local factors.
    Write $\CE_\bk=\{\iota,\bar\iota\}$, and put
\[
a_i=\mu_i^\iota+\od\!\chi_\iota,\qquad
b_i=\mu_i^{\bar\iota}+\od\!\chi_{\bar\iota},\qquad
r_i=\frac{2n+1-2i}{2}
\quad(1\le i\le 2n).
\]
The balanced condition \eqref{bal} gives 
\[
-\mu_n^\iota\le \od\!\chi_\iota\le-\mu_{n+1}^\iota,\qquad
-\mu_n^{\bar\iota}\le \od\!\chi_{\bar\iota}\le-\mu_{n+1}^{\bar\iota}.
\]
It follows that
\begin{equation}\label{eq:signs}
a_i,b_i\ge0\quad(i\le n),\qquad
a_i,b_i\le0\quad(i>n).
\end{equation}

For $a, b\in \BC$ with $a-b\in \BZ$, recall the local factors
\begin{equation}\label{eq:local-factors}
\begin{aligned}
\oL(s,\iota^a\bar \iota^b)
&=\Gamma_\BC(s+\max\{a,b\}),\\
\gamma(s,\iota^a\bar \iota^b,\psi)
&={\rm i}^{|a-b|}
\frac{\Gamma_\BC(1-s-\min\{a,b\})}
     {\Gamma_\BC(s+\max\{a,b\})},
\end{aligned}
\end{equation}
where  $\Gamma_\BC(s)=2(2\pi)^{-s}\Gamma(s)$, $s\in\BC$.

The L-parameter of $\pi_\mu\otimes\chi$ is the direct sum of the
characters
\begin{equation}\label{eq:parameter}
\iota^{a_i+r_i}\bar \iota^{\,b_{2n+1-i}-r_i},
\qquad 1\le i\le 2n.
\end{equation}
For $i\le n$, we have $r_i>0$, so \eqref{eq:signs} implies
\[
a_i+r_i>0>b_{2n+1-i}-r_i.
\]
Then \eqref{eq:local-factors} gives
\begin{equation}\label{eq:L-mu}
\oL(s,\pi_\mu\otimes\chi)
=\prod_{i=1}^{n}
\Gamma_\BC(s+r_i+a_i)\Gamma_\BC(s+r_i+b_i).
\end{equation}
In particular,
\begin{equation}\label{eq:L-zero}
\oL(s,\pi_0)=\prod_{i=1}^{n}\Gamma_\BC(s+r_i)^2.
\end{equation}

On the other hand, we have the  characters
\[
\zeta_{\mu,i}\chi =\iota^{a_i+r_i}\bar \iota^{b_i+r_i},
\qquad
\zeta_{0,i}=\iota^{r_i}\bar \iota^{r_i}.
\]
For $1\le i\le n$, put
$c_i=\min\{a_i,b_i\}\in\BZ$, $d_i=\max\{a_i,b_i\}\in\BZ$.
Applying \eqref{eq:local-factors}, we obtain
\[
\frac{\gamma(s,\zeta_{0,i},\psi)}
     {\gamma(s,\zeta_{\mu,i}\chi,\psi)}
=
{\rm i}^{-|a_i-b_i|}
\frac{
\Gamma_\BC(1-s-r_i)\Gamma_\BC(s+r_i+d_i)}
{
\Gamma_\BC(s+r_i)\Gamma_\BC(1-s-r_i-c_i)}.
\]
Substituting this formula and \eqref{eq:L-mu}--\eqref{eq:L-zero}
into the definition of $\Xi_{\mu,\chi_\natural}(s)$, and using
$\{a_i,b_i\}=\{c_i,d_i\}$, yields
\begin{equation}\label{eq:Xi}
\begin{aligned}
\Xi_{\mu,\chi_\natural}(s)
& =
\prod_{i=1}^{n}{\rm i}^{-|a_i-b_i|}
\frac{
\Gamma_\BC(s+r_i)\Gamma_\BC(1-s-r_i)}
{\Gamma_\BC(s+r_i+c_i)\Gamma_\BC(1-s-r_i-c_i)} \\
& = \prod_{i=1}^{n}{\rm i}^{-|a_i-b_i|} (-1)^{c_i}.
\end{aligned}
\end{equation}

Finally, $|a_i-b_i|+2c_i=a_i+b_i$, and therefore
\[
\begin{aligned}
\Xi_{\mu,\chi_\natural}(s)
&=\prod_{i=1}^{n}{\rm i}^{-|a_i-b_i|}(-1)^{c_i}\\
&={\rm i}^{-\sum_{i=1}^{n}(a_i+b_i)}\\
&={\rm i}^{-\sum_{i=1}^{n}
(\mu_i^\iota+\od\!\chi_\iota+
 \mu_i^{\bar\iota}+\od\!\chi_{\bar\iota})}\\
&=\Omega_{\mu,\chi_\natural}^{-1}.
\end{aligned}
\]
\end{proof}

Therefore in view of \eqref{CDtrans}, Theorem \ref{APR} is  reduced to the following result.

\begin{prpl}
The following diagram is commutative:
\[
\begin{CD}
I_\mu \otimes F_\mu^\vee @>\oZ_{\rm FJ}^\diamond(s, \cdot, \chi) \otimes \lambda_{F_\mu, \chi_\natural}>> \BC  \\
@A \imath_\mu AA  @AA \Xi_{\mu,\chi_\natural}(s)A \\
I_0  @>\oZ_{\rm FJ}^\diamond(s, \cdot, \chi^\natural)>>  \BC
\end{CD}
\]
\end{prpl}

\begin{proof}
Following \cite[Section 2.2]{LLS24}, we realize  $I_\mu \otimes F_\mu^\vee$ as a space of $F_\mu^\vee$-valued functions $f$ on $G$, such that $g \in G$ acts by
$(g\cdot f)(x) := g\cdot (f(xg))$ for $g, x\in G$. 
Then the translation $\imath_\mu: I_0 \to I_\mu \otimes F_\mu^\vee$ is given by 
\be \label{trans}
\imath_\mu(f)(x) := f(x)\cdot (x^{-1}\cdot v_\mu^\vee),\quad f\in I_0, \ x\in G.
\ee
Clearly $\imath_\mu$ maps $I_0^\sharp$ into $I_\mu^\sharp\otimes F_\mu^\vee$, where $I_\mu^\sharp = I(\zeta_\mu)^\sharp$ is defined by \eqref{sharp}.

Recall that $d_\xi$ is defined by \eqref{dxi}. By direct calculations, it is easy to verify that 
\be \label{deq}
d_{\zeta_\mu}= d_{\zeta_0} = \begin{cases} n(n+1)/2, &  \text{if $\bk\cong\BR$}, \\
n^2, & \text{if $\bk\cong\BC$}.
\end{cases}
\ee
In view of Theorem \ref{MFSha},
Lemma \ref{lem:res}, \eqref{incl} and \eqref{deq}, it suffices to prove the commutativity of following diagram where $\phi\in \CS(\bk^n)$:
\be\label{CDLambda}
\begin{CD}
I_\mu^\sharp \otimes F_\mu^\vee @>\Lambda(t, s; \cdot, \phi, \chi) \otimes \lambda_{F_\mu, \chi_\natural}>> \BC\cup\{\infty\}  \\
@A \imath_\mu AA  @| \\
I_0^\sharp  @>\Lambda(t, s; \cdot, \phi, \chi^\natural)>>  \BC\cup\{\infty\}.
\end{CD}
\ee

By definition, for $f\in I_0^\sharp$ we have
\be \label{Lambdasharp}
\begin{aligned}
  & \langle \Lambda(t, s; \cdot, \phi, \chi) \otimes \lambda_{F_\mu, \chi_\natural}, \imath_\mu(f)\rangle   \\
= & \int_{G_n} \left\langle \Lambda_{\rm RS}(t, \cdot, \phi, \eta_\mu^{-1} ) \otimes \lambda_{F_\mu,\chi_\natural},   I_\mu\left(\gamma\begin{bmatrix} g & \\ & 1\end{bmatrix}\right)\cdot \imath_\mu(f)\right\rangle  \chi(g) |g|_\bk^{s-\frac{1}{2}}\od\! g,
\end{aligned}
\ee
where $\gamma$ is given by \eqref{gamman}, and the outer integrals here and below are interpreted in the sense of meromorphic continuation from a domain of absolute convergence. We find that 
\[
\begin{aligned}
& \left\langle \Lambda_{\rm RS}(t, \cdot, \phi, \eta_\mu^{-1} ) \otimes \lambda_{F_\mu,\chi_\natural},   I_\mu\left(\gamma\begin{bmatrix} g & \\ & 1\end{bmatrix}\right)\cdot \imath_\mu(f)\right\rangle  \\
= &  \int_{G_n} \left\langle \lambda_{F_\mu,\chi_\natural},  \imath_\mu(f)\left(z_0\begin{bmatrix} g' \\ & g' \end{bmatrix} \gamma \begin{bmatrix} g \\ & 1\end{bmatrix} \right) \right\rangle \phi(v_n g') \eta_\mu^{-1}(g') |g'|_\bk^t\od\! g'.
\end{aligned}
\]
Noting that $z_0\gamma = \gamma_0$ and 
\[
z_0\begin{bmatrix} g' \\ & g' \end{bmatrix} \gamma \begin{bmatrix} g \\ & 1\end{bmatrix} = \gamma_0 \begin{bmatrix}  g'g\\ & g'\end{bmatrix},
\]
from Lemma \ref{balanced-vec} and \eqref{trans} it is easy to check that 
\[ 
\left\langle \lambda_{F_\mu,\chi_\natural}, \imath_\mu(f)\left(\gamma_0 \begin{bmatrix} g'g\\ & g'\end{bmatrix} \right)\right\rangle = f\left(\gamma_0 \begin{bmatrix} g'g\\ & g'\end{bmatrix} \right) 
\eta_\mu(g') \prod_{\iota\in \CE_\bk} \iota(\det g)^{-\od\!\chi_{\iota}}.
\]

It follows that
\[
\begin{aligned}
 & \left\langle \Lambda_{\rm RS}(t, \cdot, \phi, \eta_\mu^{-1} ) \otimes \lambda_{F_\mu,\chi_\natural},   I_\mu\left(\gamma\begin{bmatrix} g & \\ & 1\end{bmatrix}\right)\cdot \imath_\mu(f)\right\rangle  \\
 = \, &  \Lambda_{\rm RS}\left(t,  \gamma\begin{bmatrix} g & \\ & 1\end{bmatrix}\cdot f, \phi, \eta_0^{-1}\right) 
 \prod_{\iota\in\CE_\bk}\iota(\det g)^{-\od\!\chi_{\iota}},
 \end{aligned}
\]
noting that $\eta_0$ is the trivial character. Plugging the last equation into \eqref{Lambdasharp} shows that 
\[
\langle \Lambda(t, s; \cdot, \phi, \chi) \otimes \lambda_{F_\mu, \chi_\natural}, \imath_\mu(f)\rangle = \Lambda(t, s; f, \phi, \chi^\natural),
\]
which verifies the commutativity of \eqref{CDLambda}.
\end{proof}

\section{Non-vanishing Hypothesis at Infinity} \label{sec:NV}

In this section, we retain the notation from the last section and prove the following non-vanishing hypothesis at infinity, which is another key ingredient in the proof of Theorem \ref{BDconj}. 

\begin{thml}[Non-vanishing Hypothesis at Infinity] \label{thm:NV}
The archimedean modular symbol 
\[
\wp_{\mu,\chi}:  \oH^{d_\bk}_{\rm ct}(\BR^\times_+\bsl G^0; \pi_\mu\otimes F_\mu^\vee)\otimes \oH^0_{\rm ct}(\BR^\times_+\bsl H^0; \xi_{\mu,\chi}\otimes \xi_{\mu, \chi_\natural}^\vee)
\to \oH^{d_\bk}_{\rm ct}(\BR^\times_+\bsl  H^0; \BC)
\]
in \eqref{archMS} is nonzero. 
\end{thml}

We first make two reductions:
\begin{itemize}
\item The case when $\bk = \BR$ has been proved in \cite[Theorem 5.1]{S19}. 
\item By Theorem \ref{APR}, it suffices to prove Theorem \ref{thm:NV} when $\mu =0$, in which case $\chi$ must be trivial. 
\end{itemize}
In the rest of this section, we assume that $\bk = \BC$ and $\mu =0$, and prove the following result which implies Theorem \ref{thm:NV} in view of the above reductions. 

\begin{prpl} \label{prop:NV} For $\bk = \BC$, the archimedean modular symbol
\[
\wp_0: \RH^{d_\bk}_{\rm ct}(\BR^\times_+\bs G; \pi_0) \to \RH^{d_\bk}_{\rm ct}(\BR^\times_+\bs H; \BC) 
\]
induced by the functional $\oZ^\circ_{\rm FJ}(\frac{1}{2}, \cdot, {\bf 1}): \pi_0\to\BC$ is nonzero, where ${\bf 1}$ denotes the trivial character of $\bk^\times$. 
\end{prpl}

We view $G$ and $H$ as real Lie groups, and denote their Lie algebras by $\frak g$ and $\frak h$, respectively. Let $K = \RU(2n)$ and $C\cong \RU(n)\times \RU(n)$ be the standard maximal compact subgroups of $G$ and $H$, with Lie algebras $\frak k$ and $\frak c$, respectively. Put
\[
\frak s: = \frak g/(\BR\oplus \frak k),\qquad \frak q: = \frak h / (\BR\oplus \frak c). 
\]
For every real Lie algebra $\frak l$, write $\frak l_\BC$ for its complexification. 

 By \cite[Theorem 6.1]{HM} and \cite[Proposition 9.4.3]{W88}, we have identifications
\[
\begin{aligned}
    & \RH^{d_\bk}_{\rm ct}(\BR^\times_+\backslash G; \pi_0)  = \RH^{d_\bk}(\frak g_\BC/\BC, K; \pi_0)  = \Hom_{K}(\wedge^{d_\bk} \frak s_\BC, \pi_0), \\
    &
         \RH^{d_\bk}_{\rm ct}(\BR^\times_+\backslash H; \BC) =  \RH^{d_\bk}(\frak h_\BC/\BC, C; \BC)  = \Hom_{C}(\wedge^{d_\bk}\frak q_\BC,  \BC).
       \end{aligned}
       \]
     Moreover, the archimedean modular symbol  $\wp_0$ in Proposition \ref{prop:NV} is identical to the map
      \begin{equation}\label{cpbkj2}
          \wp_0: \Hom_{K}(\wedge^{d_\bk} \frak s_\BC, \pi_0) \to  \Hom_{C}(\wedge^{d_\bk}\frak q_\BC,  \BC)
      \end{equation}
     which sends an element $\phi $ to the composition of
     \[
     \wedge^{d_\bk} \frak q_\BC \hookrightarrow \wedge^{d_\bk} \frak s_\BC \xrightarrow{\phi} \pi_0 \xrightarrow{\oZ^\circ_{\rm FJ} (\frac{1}{2},\cdot, {\bf 1} )} \BC.
\]

Following \cite[Section 5.2]{S19}, we define an embedding
          \[
                \BC^{2n}\hookrightarrow \frak g
\]
     which sends $(a_1,a_2,\ldots,a_{2n})$ to the matrix
    \[
    \left[
  \begin{array}{cccccccc}
    \frac{a_1+a_{2n}}{2} & 0 & \cdots & 0 & 0 &\cdots & 0 & \frac{a_1-a_{2n}}{-2\mathrm i}  \\
     0&\frac{a_2+a_{2n-1}}{2}  & \cdots & 0 & 0 &\cdots & \frac{a_2-a_{2n-1}}{-2\mathrm i}  & 0 \\
     \cdots & \cdots & \cdots & \cdots & \cdots &\cdots & \cdots &\cdots  \\
   0 & 0 & \cdots & \frac{a_n+a_{n+1}}{2} & \frac{a_n-a_{n+1}}{-2\mathrm i} &\cdots & 0 & 0  \\
   0 & 0 & \cdots &\frac{a_n-a_{n+1}}{2\mathrm i}   & \frac{a_{n+1}+a_n}{2} &\cdots & 0 & 0  \\
  \cdots & \cdots & \cdots & \cdots & \cdots &\cdots & \cdots &\cdots  \\
   0& \frac{a_2-a_{2n-1}}{2\mathrm i} & \cdots & 0 & 0 &\cdots & \frac{a_{2n-1}+a_2}{2}  & 0 \\
   \frac{a_{1}-a_{2n}}{2\mathrm i} & 0 & \cdots & 0 & 0 &\cdots &  & \frac{a_{2n}+a_1}{2}  \\
  \end{array}
   \right].
  \]
   This embedding identifies $\BC^{2n}$ as a Cartan subalgebra $\frak t$ of $\frak g$. With the so obtained Cartan subalgebra $\frak t$, the root system of $\frak g$ is
$$\Phi(\frak g, \frak t) = \set{\pm(\eps_i-\eps_j) \mid 1\leq i< j \leq 2n}.$$
   Here $\eps_1, \eps_2, \ldots, \eps_{2n}$ are the coordinate functions on $\BC^{2n}$.
The positive root system
$$\Phi^+(\frak g, \frak t) = \set{\eps_i-\eps_j \mid 1\leq i< j \leq 2n}$$
defines a Borel subalgebra $\frak b$ of $\frak g$. Write $\frak u$ for its nilpotent radical so that $\frak b=\frak t\ltimes \frak u$. Put
\[
\b':=\b \times \b, \quad \t' := \t \times \t \quad \text{and}\quad \u' := \u \times \u,
\]
viewed as Lie subalgebras of $\g_\BC= \g\times \g$. Here the identification $\g_\BC= \g\times \g$ is induced by the complexification map 
\[
\g\rightarrow \g\times \g, \quad x\mapsto (x, \bar x)\quad (\textrm{$\bar x$ denotes the entry-wise complex conjugation of $x$}). 
\]

Put 
\[
J:=\begin{bmatrix} - 1_n & 0 \\ 0 & 1_n\end{bmatrix} \in H \subset H_{\BC}. 
\]
Then ${\rm Ad}_J(\t') = \t'$, and the Borel subalgebra ${\rm Ad}_J(\b')$ of $\g_\BC$ is opposite to $\b'$, where ${\rm Ad}$ stands for the adjoint action. 

Let $\tau_\circ$ be the minimal $K$-type of $\pi_0$, which has extremal weight 
\[
(2n-1, 2n-3, \ldots, 3-2n, 1-2n). 
\]
The Cartan-Helgason theorem (see \cite[Chapter V, Theorem 4.1]{He00})  implies that
\begin{equation} \label{d1}
\dim \tau^C_\circ =1.
\end{equation}

\begin{leml}\label{cprv00}
There exists $\phi_0\in \Hom_{K}(\wedge^{d_\bk} \s_\BC , \tau_\circ)$ such that the composition of
\[
 \wedge^{d_\bk}\q_\BC\hookrightarrow \wedge^{d_\bk} \s_\BC \xrightarrow{\phi_0} \tau_\circ
\]
is nonzero. Moreover, the image of this composition map equals $\tau^C_\circ$.
\end{leml}

\begin{proof}
The proof is a modification of  that of \cite[Lemma 2.10]{S17}. Fix a $K$-invariant positive definite Hermitian form on $\s_\BC$. This induces a $K$-invariant positive definite Hermitian form  on $\wedge^{d_\bk} \s_\BC$.

For every subspace $\l$ of $\g_\BC$, we define
\[
  \l_{\mathrm n}:=\l/(\l\cap (\BC\oplus \k_\BC))\subset \s_\BC.
\]
In particular, $(\h_\BC)_{\mathrm n}=\q_\BC$.
Write
\[
  (\t')_{\mathrm n}=\a\oplus \a',
\]
where $\a:=(\t')_{\mathrm n}\cap (\h_\BC)_{\mathrm n}$, and $\a'$ is its orthogonal complement in $(\t')_{\mathrm n}$. Then we have an orthogonal decomposition
\begin{equation}\label{dimeq0}
   \s_\BC=\left(\a\oplus (\u')_{\mathrm n}\right)\oplus \left(\a'\oplus(\Ad_J(\u'))_{\mathrm n}\right).
\end{equation}
It is easy to verify that 
\[
\Ad_J(\b') \cap \h_\BC \subset \t'.
\]
This further implies that 
\begin{equation}\label{dimeq1}
  \q_\BC\cap \left(\a'\oplus(\Ad_J(\u'))_{\mathrm n}\right)=\{0\}.
\end{equation}
Note that
\begin{equation}\label{dimeq2}
  \dim \q_\BC=\dim \left(\a\oplus (\u')_{\mathrm n}\right).
\end{equation}

Using \eqref{dimeq1}, \eqref{dimeq2} and the orthogonal decomposition \eqref{dimeq0}, we conclude that the  pairing
\[
  \langle\,,\,\rangle: \q_\BC\times \left(\a\oplus (\u')_{\mathrm n}\right)\rightarrow \BC
\]
is non-degenerate. Therefore the one-dimensional subspaces $\wedge^{d_\bk} \q_\BC$ and $\wedge^{d_\bk}(\a\oplus (\u')_{\mathrm n})$ of $\wedge^{d_\bk}\s_\BC$ are not orthogonal to each other. The space $\wedge^{d_\bk}(\a\oplus (\u')_{\mathrm n})$ generates an irreducible $K$-subrepresentation of $\wedge^{d_\bk}\s_\BC$, which is isomorphic to $\tau_\circ$.
We write $\tau'_\circ$ for this representation. Let $\phi_0': \wedge^{d_\bk}\s_\BC \rightarrow \tau'_\circ$ denote the orthogonal projection. Then the
 composition of
\[
 \wedge^{d_\bk}\q_\BC\hookrightarrow \wedge^{d_\bk} \s_\BC \xrightarrow{\phi_0'} \tau'_\circ
\]
is nonzero. This proves the first assertion of the lemma. The second assertion then follows by \eqref{d1} and by considering the $C$-actions.
\end{proof}

Recall that $\tau_\circ$ is the minimal $K$-type of $\pi_0$.  By \cite[Theorem 1.2]{LT20}, the functional 
\[
\oZ^\circ_{\rm FJ}(\frac{1}{2},\cdot, {\bf 1}): \tau_\circ \to \BC
\]
is nonzero. By the $C$-invariance of this functional, it is nonzero when restricted to the one-dimensional subspace $\tau^C_\circ$.  Let 
\[
\phi_0\in \Hom_{K}(\wedge^{d_\bk} \s_\BC , \tau_\circ) \subset \Hom_{K}(\wedge^{d_\bk} \s_\BC , \pi_0)
\]
be as in Lemma \ref{cprv00}.  Then $\wp_0$ in \eqref{cpbkj2} maps $\phi_0$ to a nonzero element of $\Hom_{C}(\wedge^{d_\bk}\frak q_\BC,  \BC)$.
This finishes the proof of Proposition \ref{prop:NV}.

\section{Non-archimedean Period Relations} \label{sec:NAPR}

In this section, suppose that $\bk$ is non-archimedean. 
\subsection{$\Aut(\BC)$-actions on Shalika models}

Denote by $\widehat{\bk^\times}$
the group of all characters of $\bk^\times$.
Then 
\[
\bigsqcup_{\eta\in \widehat{\bk^\times}} \Ind_S^{G} \varphi_\eta
\]
is a bundle  over $\widehat{\bk^\times}$ of smooth representations of $G$, where $\varphi_\eta$ is given in \eqref{varphi2n}. The set $\widehat{\bk^\times}$ carries an obvious action of $\Aut(\BC)$:
\[
\Aut(\BC)\curvearrowright \widehat{\bk^\times},\quad (\sigma, \eta)\mapsto {^\sigma\!\eta}:=\sigma\circ \eta.
\] 
In what follows, we will  define a $G$-equivariant action 
   \be \label{sigmaf000}
\Aut(\BC)\curvearrowright\bigsqcup_{\eta\in \widehat{\bk^\times}} \Ind_S^{G} \varphi_\eta,\quad (\sigma, f)\mapsto {^\sigma\! f}. 
\ee
such that the natural map $\bigsqcup_{\eta\in \widehat{\bk^\times}} \Ind_S^{G} \varphi_\eta\rightarrow \widehat{\bk^\times}$ is $\Aut(\BC)$-equivariant.

Let $p$ be the residue characteristic of $\bk$.
Let $\mu_p^\infty\subset \BC^\times$ denote the subgroup consisting of all the roots of unity of $p$-power order, and let  $\BQ(\mu_p^\infty)\subset \BC$ be the subfield generated  by  $\mu_p^\infty$. 

Recall the cyclotomic character
\[
  \mathrm{Aut}(\BQ(\mu_p^\infty)/\BQ) \rightarrow \BZ_p^\times,\quad \sigma\mapsto t_{\sigma,p},
\]
which is a topological isomorphism specified by requiring that
\[
  \sigma(x)=x^{t_{\sigma,p}}, \quad \textrm{for all $x\in \mu_p^\infty$}.
\]
Write $\sigma\mapsto t_{\sigma,\bk}$ for the composition of
 \[
 \Aut(\BC)\xrightarrow{\textrm{restriction} }\mathrm{Aut}(\BQ(\mu_p^\infty)) \xrightarrow{\sigma\mapsto t_{\sigma,p}} \BZ_p^\times \subset \bk^\times.
 \]
 Then for all characters  $\psi_0:\bk\rightarrow  \BC^\times$,
\be\label{tsigma}
\sigma(\psi_0(x))=\psi_0(t_{\sigma,\bk} x)\quad \textrm{for all $\sigma\in \Aut(\BC)$ and $x\in \bk$.} 
\ee
    
Denote
 \[
\mathbf{t}_{\sigma,\bk}:= \begin{bmatrix}
                        t_{\sigma,\bk} \cdot 1_n & 0 \\
                        0 & 1_n \\
                      \end{bmatrix} \in H\subset G.
                    \]
It follows from \eqref{tsigma} that for all $\sigma\in \Aut(\BC)$ and $f\in \Ind_S^G\varphi_\eta$, the function 
\[
 {}^\sigma \! f: G\rightarrow \BC, \quad g\mapsto \sigma\left( f(\mathbf{t}_{\sigma,\bk}^{-1}  g)  \right)
\]
is an element of  $\Ind_S^G\varphi_{^\sigma\!\eta}$. This defines the desired action \eqref{sigmaf000}.

Let $\pi$ be an irreducible admissible smooth representation of $G$. Denote by $\BQ(\pi)$ its rationality field, namely 
\be\label{rf}
 \BQ(\pi):= \{x\in \BC\mid \sigma(x)=x\quad \textrm{for all $\sigma\in \Aut(\BC)$ such that $^\sigma \pi\cong \pi$}\}. 
\ee
Here $^\sigma \pi$ denotes the $\sigma$-twist of $\pi$, which is an irreducible admissible smooth representation of $G$ equipped with a $G$-equivariant $\sigma$-linear isomorphism $\pi\rightarrow {^\sigma\!\pi}$. By \cite[Proposition 3.2]{Cl90},    $\pi$ has a $G$-stable $\BQ(\pi)$-rational structure $\pi(\BQ(\pi))$ so that 
\be\label{rapi}
\pi(\BQ(\pi))\otimes_{\BQ(\pi)} \BC=\pi.  
\ee

Now we further assume that  $\pi$ is an irreducible subrepresentation of $\Ind_S^G\varphi_\eta$. Then by uniqueness of twisted Shalika models (Lemma~\ref{uniquesh}), the action \eqref{sigmaf000} restricts to an action  \be\label{sigmaf0001} 
\Aut(\BC/\BQ(\pi,\eta))\curvearrowright\pi, \quad (\sigma, u)\mapsto {^\sigma \! u}. 
\ee
Here $\BQ(\pi,\eta)\subset \BC$ denotes the compositum of $\BQ(\pi)$ and $\BQ(\eta)$ (similar notation will be used without explanation).

\begin{leml}\label{lemlambda}
   There exists a nonzero element $\lambda\in \Hom_S(\pi,\varphi_\eta)$ such that 
    \[\lambda(\pi(\BQ(\pi)))\subset  \BQ(\pi,\eta, \mu_p^\infty),
    \]
    where $\BQ(\pi,\eta, \mu_p^\infty)$ is the compositum of $\BQ(\pi,\eta)$ and $\BQ(\mu_p^\infty)$. 
    \end{leml}
\begin{proof}
   Write 
\[\pi':=\pi(\BQ(\pi))\otimes_{\BQ(\pi)} \BQ(\pi,\eta, \mu_p^\infty),
\]
which is a $\BQ(\pi, \eta,\mu_p^\infty)$-rational form of $\pi$. 
Since 
\[
\varphi_\eta(S)\subset \textrm{the compositum of $\BQ(\eta)$ and $\BQ(\mu_p^\infty)$},
\]
the character $\varphi_\eta: S\rightarrow\BC^\times $ induces a  character 
\[
\overline \varphi_{\eta}: S\rightarrow \BQ(\pi,\eta, \mu_p^\infty)^\times.
\]
Then we have an identification
\begin{equation}\label{coinvs}
  (\pi'\otimes \overline \varphi_\eta^{-1})_S\otimes_{\BQ(\pi,\eta, \mu_p^\infty)} \BC= (\pi\otimes \varphi_\eta^{-1})_S,
\end{equation}
where the subscript $S$ indicates the coinvariant space. 
The right-hand side of \eqref{coinvs} is nonzero, so is the left-hand side. This implies the lemma. 
\end{proof}

\begin{prpl}\label{rationalsh}
 Suppose that $\BQ(\pi)$ is a number field. Then the invariant space $ \pi^{\Aut(\BC/\BQ(\pi,\eta))}$ of the action \eqref{sigmaf0001}  is a $\BQ(\pi,\eta)$-rational structure of $\pi$, that is,
  \[
    \pi^{\Aut(\BC/\BQ(\pi,\eta))}\otimes_{\BQ(\pi,\eta)} \BC=\pi.
  \]
 \end{prpl}
 
\begin{proof}
Note that $\eta^n$ equals the central character of $\pi$. Hence $\BQ(\eta)$ and $\BQ(\pi,\eta)$ are also number fields. 

The rational form in \eqref{rapi} yields an action
\be\label{sigmaf0002} 
\Aut(\BC/\BQ(\pi))\curvearrowright\pi, \quad (\sigma, u)\mapsto \sigma\cdot  u. 
\ee
It follows from Lemma~\ref{uniquesh} that 
\be\label{dimones3}
\dim \Hom_G(\pi, \Ind_S^G\varphi_\eta)=1. 
\ee
 The actions \eqref{sigmaf000} and \eqref{sigmaf0002} yield an action 
\be\label{actphi}
\Aut(\BC/\BQ(\pi,\eta))\curvearrowright\Hom_G(\pi, \Ind_S^G\varphi_\eta), \quad (\sigma, \phi)\mapsto \sigma\cdot \phi:=(u\mapsto {^\sigma\!(\phi(\sigma^{-1}\cdot u)))}.
\ee
In order to prove the proposition, it suffices to show that the invariant space of the above action is nonzero:
\be\label{nonzeror}
\Hom_G(\pi, \Ind_S^G\varphi_\eta)^{\Aut(\BC/\BQ(\pi,\eta))}\neq\{0\}. 
\ee

Let $\lambda$ be as in Lemma~\ref{lemlambda}.
By Frobenius reciprocity, $\lambda$ corresponds to an element 
    \[
    0\neq \phi_\lambda\in \Hom_{G}(\pi,\Ind_S^G\varphi_\eta).
    \]
We claim that $\phi_\lambda$ is smooth with respect to the  action \eqref{actphi}, namely there exists a  number field $E\subset \BC$ containing $\BQ(\pi, \eta)$ such that 
 \be\label{claim0}
  \sigma\cdot \phi_\lambda=\phi_\lambda \quad \textrm{for all }\sigma\in \Aut(\BC/E). 
 \ee
Fix a vector $u \in \pi(\BQ(\pi))$ such that \[
    0\neq \la \lambda, u \ra \in \overline \BQ\qquad (\textrm{the field of algebraic numbers in $\BC$}). 
\]
Pick a number field $E\subset \BC$ containing $\BQ(\pi,\eta)$ that is  sufficiently large so that 
\[
\mathbf{t}_{\sigma,\bk}^{-1}\cdot (\phi_\lambda(u ))=\phi_\lambda(u )\quad\textrm{and}\quad \sigma(\la \lambda,  u \ra)=\la \lambda,  u \ra, \quad\textrm{for all }\sigma\in \Aut(\BC/E).
\]
Let $\sigma\in \Aut(\BC/E)$. It follows from  \eqref{dimones3} that 
    \[
    \sigma\cdot  \phi_\lambda= a_\sigma\cdot  \phi_\lambda\quad \textrm{for some }a_\sigma\in \BC^\times.  
    \]
    
Write $\lambda_1:\Ind_S^G\varphi_\eta\rightarrow \BC $ for the map of evaluating at the identity $1\in G$. Then  
\begin{eqnarray*}
\la\lambda_1, (\sigma\cdot  \phi_\lambda)(u )\ra  
&=&\la \lambda_1, \,^\sigma\!\left(\phi_\lambda({\sigma^{-1}} \cdot u )\right)\ra \\
&=&\la \lambda_1, \,^\sigma\!\left(\phi_\lambda(u)\right)\ra \\
&=&\sigma((\phi_\lambda(u ))(\mathbf t_{\sigma,\bk}^{-1}))\\
&=&\sigma(\la \lambda_1, \mathbf t_{\sigma,\bk}^{-1}\cdot (\phi_\lambda(u ))\ra) \\
&=&\sigma(\la \lambda_1, \phi_\lambda(u )\ra) \\
&=&\sigma(\la \lambda,  u\ra) \\
&=&\la \lambda,  u \ra.
\end{eqnarray*}
On the other hand, 
\begin{eqnarray*}
 \la\lambda_1, (\sigma\cdot  \phi_\lambda)(u )\ra  
&=&a_\sigma\cdot \la\lambda_1, \phi_\lambda(u )\ra   \\
&=&a_\sigma\cdot \la \lambda,  u \ra.
\end{eqnarray*}
Therefore $a_\sigma=1$, which proves the claim of \eqref{claim0}.  By \cite[Proposition 11.1.6]{TASpringerLinearAlgebraicGroups}, this claim implies \eqref{nonzeror}, which proves the proposition. 
\end{proof}

 \begin{remarkl}\label{rkrs}
 The corresponding result of Proposition \ref{rationalsh} for Whittaker models has been proved by G. Harder for the $\GL_2$-case in \cite[page 80]{Hd83} and by
J. Mahnkopf for the general case in \cite[page 592]{M}. See also \cite[Lemma 3.2]{RS08}. Their proof depends crucially on the fact that the Whittaker functional does not vanish on new vectors.
It is suggested in the published version of  \cite{GR14} that the same idea may be applied to the Shalika functionals and
prove the statement of our Proposition \ref{rationalsh}.
However, to the best of our knowledge, the non-vanishing of the Shalika functional on new
vectors was not known before, and has only been established recently in \cite{GM24} for some unramified characters $\eta$
and in \cite{O26} under an assumption such as $\pi$ being supercuspidal, which is still not enough
for general situations. Our proof given here does not use the theory of new vectors.
\end{remarkl}

We omit the proof of the  following lemma, which is easy to check.  
\begin{leml}\label{rspherical}
   Suppose that   $u^\circ\in \pi$ is a  spherical vector whose value at $1\in G$ equals  $1$. Then $u^\circ$ is fixed under the action \eqref{sigmaf0001}.
  \end{leml}

\subsection{Non-archimedean Period Relations}\label{subsec:NAPR}

For a  character  $\chi: \bk^\times \rightarrow \BC^\times$ and $t\in\BC$, define a character 
\[
\xi_{\eta,\chi,t}:= (\chi \abs{\cdot}_\bk^t) \boxtimes (\chi^{-1}\eta^{-1}\abs{\cdot}_\bk^{-t})
\]
of $H$.  Denote by $\frak m_\bk'$ the one-dimensional space of left invariant Borel measures on $G_n$, which is endowed with a $\BQ$-rational structure such that all  open compact subgroups of $G_n$ have rational volumes with respect to the $\BQ$-rational measures.

As before, suppose $\pi$ is an irreducible subrepresentation of $\Ind_S^G\varphi_\eta$. 
Similar to \eqref{NFJ},  we have  the nonzero functional
 \[
  \oZ^\circ_{\rm FJ}(s, \cdot ,\chi) \in \Hom_{H}(\pi\otimes\xi_{\eta, \chi, s-\frac{1}{2}}\otimes {\frak m}_\bk', \BC)
        \]
        given by 
        \be \label{nonarchiFJ}
         \oZ^\circ_{\rm FJ}(s, u\otimes 1\otimes \od\!g ,\chi) = \frac{1}{\oL(s, \pi\otimes\chi)} \int_{G_n} u\left( \begin{bmatrix} g & 0 \\ 0 & 1_n\end{bmatrix}\right)\cdot \chi(g)|g|_\bk^{s-\frac{1}{2}}\od\! g,
        \ee
        for $u\otimes 1\in  \pi\otimes \xi_{\eta, \chi, s-\frac{1}{2}}$ and $\od\!g\in {\frak m}_\bk'$.
Note that if $s_0\in \frac{1}{2}+\BZ$, then $\xi_{\eta, \chi, s_0-\frac{1}{2}}$ is defined over $\BQ(\eta, \chi)$.

Denote by $\mathfrak{c}(\chi)$ and $\mathfrak{c}(\psi)$ the conductors of $\chi$ and $\psi$ respectively, which are fractional ideals of $\bk$. Fix $y\in\bk^{\times}$ such that $\mathfrak{c}(\psi)=y\cdot\mathfrak{c}(\chi)$. Define the local Gauss sum by
	\begin{equation}\label{localgauss}
		\CG(\chi):=\CG(\chi,\psi):=\int_{\CO_\bk^{\times}}\chi(yx)^{-1}\cdot\psi(y x)\mathrm{d}x,
	\end{equation}
	where $\mathrm{d}x$ is the normalized Haar measure so that $\CO_\bk^{\times}$ has volume $1$. We note that the above definition is independent of the choice of $y$. 

For every $\sigma\in \Aut(\BC)$, 
\[
^\sigma\!\pi:=\{^\sigma\! f\mid f\in \pi\}
\]
is an irreducible subrepresentation of $\Ind_S^G\varphi_{{^\sigma\!\eta}}$.

    
  \begin{thml}[Non-archimedean Period Relations]\label{thm:NAPR}
  For all $s_0\in \frac{1}{2}+\BZ$ and
   all characters  $\chi: \bk^\times \rightarrow \BC^\times$, the  diagram
   \begin{equation}\label{CD: non-archi period relation}
       \CD
           \pi\otimes\xi_{\eta, \chi,s_0-\frac{1}{2}}\otimes{\frak m}'_\bk @>\CG(\chi)^n\cdot \,\oZ^\circ_{\rm FJ}(s_0, \cdot, \chi)>> \BC \\
           @V \sigma VV @V \sigma VV  \\
            {}^\sigma\!\pi \otimes\xi_{{}^\sigma\!\eta, {}^\sigma\!\chi, s_0-\frac{1}{2}}\otimes {\frak m}'_\bk @>\CG({}^\sigma\!\chi)^n\cdot \, \oZ^\circ_{\rm FJ}(s_0, \cdot, {}^\sigma\!\chi)>> \BC
         \endCD
   \end{equation}
   commutes for all $\sigma\in\Aut(\BC)$. Here the left vertical arrow is the obvious $\sigma$-linear isomorphism.
   \end{thml}

\begin{proof}
    
Let $\sigma\in \Aut(\BC)$. 
By \cite[Lemma 4.6]{Cl90}, for  every $s_0\in\frac{1}{2}+\BZ$ we have that
\[
  \sigma\left(\frac{1}{\oL(s_0, \pi\otimes\chi)}\right) = \frac{1}{\oL(s_0,{}^\sigma\!\pi\otimes{}^\sigma\!\chi)}.
\]

Let  $u\otimes 1\in \pi\otimes \xi_{\eta,\chi,s_0-\frac{1}{2}}$ and $\od\!g\in \frak m'_\bk$. 
When $s_0\in\frac{1}{2}+\BZ$ is sufficiently large, we have that
\begin{eqnarray*}
&&\sigma\left(\oZ^\circ_{\rm FJ}(s_0, u\otimes 1\otimes \od\!g, \chi) \right)\\
&=& \sigma\left(\frac{1}{\oL(s_0, \pi\otimes\chi)}\right)\int_{G_n} \sigma\left(  u \left(\begin{bmatrix} g & 0 \\ 0 & 1_n\end{bmatrix} \right)\right) \cdot {}^\sigma\!\chi(g) \cdot \abs{g}_{\bk}^{s_0-\frac{1}{2}} \, \sigma\!\cdot\od\!g\\
&=& \frac{{}^\sigma\!\chi(t_{\sigma,\bk}^{-n})}{\oL(s_0, {}^\sigma\!\pi\otimes{}^\sigma\!\chi)}\int_{G_n} \sigma\left(  u \left(\begin{bmatrix} t_{\sigma,\bk}^{-1}g & 0 \\ 0 & 1_n\end{bmatrix} \right)\right)\cdot {}^\sigma\!\chi(g) \cdot \abs{g}_{\bk}^{s_0-\frac{1}{2}} \, \sigma\!\cdot\od\!g\\
&=&  \frac{ {}^\sigma\!\chi(t_{\sigma,\bk}^{-n})}{\oL(s_0, {}^\sigma\!\pi\otimes{}^\sigma\!\chi)}\int_{G_n}  {^\sigma\!u} \left(\begin{bmatrix} g & 0 \\ 0 & 1_n\end{bmatrix}\right)\cdot {}^\sigma\!\chi(g) \cdot \abs{ g}_{\bk}^{s_0-\frac{1}{2}}\, \sigma\!\cdot\od\!g\\
&=&  {}^\sigma\!\chi(t_{\sigma,\bk}^{-n})\cdot \oZ^\circ_{{\rm FJ}}(s_0,\sigma\cdot (u\otimes 1\otimes \od\!g),{}^\sigma\!\chi).
\end{eqnarray*}
   It is well-known that
  \[\sigma(\CG(\chi)) = {}^\sigma\!\chi(t_{\sigma,\bk})\cdot \CG({}^\sigma\!\chi).
   \]
   Hence when $s_0 \in\frac{1}{2}+\BZ$ is sufficiently large,
      \be\label{eqz}
      \sigma\left( \CG(\chi)^n\cdot \oZ^\circ_{\rm FJ}(s_0, u\otimes 1\otimes \od\!g, \chi) \right)= \CG({}^\sigma\!\chi)^n\cdot \oZ^\circ_{\rm FJ}(s_0, \sigma\cdot (u\otimes1\otimes\od\!g),{}^\sigma\!\chi).
      \ee
      Note that both sides of the equality \eqref{eqz}, viewed  as functions of $s_0\in \frac{1}{2}+\BZ$, are Laurent polynomials of $q^{\frac{1}{2}-s_0}$, where $q$ denotes the cardinality of the residue field of $\bk$. Therefore \eqref{eqz} holds for all $s_0\in \frac{1}{2}+\BZ$. This proves the theorem. 
\end{proof}

\begin{remarkl}\label{rkrs2}
Using Proposition \ref{rationalsh}, a variant of Theorem \ref{thm:NAPR} is claimed in \cite[Section 3.9]{GR14}. In the third line of Section 3.9.2 of that paper, it is said that `` put $\xi^\circ_{{}^\sigma\!\Pi_v} ={}^\sigma\!\xi^\circ_{\Pi_v}$ for $\sigma\in \Aut(\BC)$". But this can only be done when ``$\, \xi^\circ_{\Pi_v}$" is already known to be $\BQ(\Pi_v)$-rational.
In this sense, we are not quite able to follow the argument in their paper. Our proof given above is self-contained.
\end{remarkl}

\section{Cohomology Groups and Modular Symbols} \label{sec:CGMS}

In this section we introduce certain cohomology groups and modular symbols, which are needed for the proof of Theorem \ref{BDconj} in the next section.  We turn to the global setting and retain the notation from the Introduction.

\subsection{Cuspidal representations of symplectic type} \label{sec:CRST}
As usual, write $\mathrm{GSpin}_{2n+1}$ for the split odd spin similitude group of semisimple rank $n$ ($n\geq 1$). Its complex dual group is the symplectic similitude group $\mathrm{GSp}_{2n}(\BC)$.

    \begin{dfnl}\label{def:LT}
       Let $\Sigma = \widehat \otimes_v' \Sigma_v$ be an irreducible generic cuspidal automorphic representation of $\mathrm{GSpin}_{2n+1}(\BA)$  and let $\Pi = \widehat \otimes_v'\Pi_v$ be an irreducible cuspidal automorphic representation of
$\GL_{2n}(\BA)$. We say that $\Pi$ is the functorial transfer of $\Sigma$ if for almost all non-archimedean local places $v$ where $\Sigma_v$ and $\Pi_v$ are unramified,
       the $L$-parameter 
       of $\Pi_v$ is equal to the composition
       $$\CW_{\rk_v}\xrightarrow{\phi_v} \mathrm{GSp}_{2n}(\BC) \hookrightarrow \GL_{2n}(\BC)$$
       up to conjugation by $\GL_{2n}(\BC)$, where $\CW_{\rk_v}$ is the local Weil group of $\rk_v$, and $\phi_v$ is the $L$-parameter of $\Sigma_v$.
    \end{dfnl}
    
   \begin{dfnl}\label{def: global repn of symplectic type}
          An irreducible cuspidal automorphic representation $\Pi$ of $\GL_{2n}(\BA)$ $(n\geq 1$) is said to be of symplectic type if it is the functorial transfer of an irreducible generic cuspidal automorphic representation  of $\mathrm{GSpin}_{2n+1}(\BA)$, in the sense of Definition \ref{def:LT}.
     \end{dfnl}

Write $\RG:=\GL_{2n}$. Similar to the local case, denote by $\RS = \GL_n^\dag \ltimes \RU$ the Shalika subgroup of $\RG$, where $\GL_n^\dag$ is the diagonal image of 
$\GL_n$ in $\RH =\GL_n\times \GL_n$, and $\RU \cong {\rm Mat}_{n\times n}$ is the unipotent radical of $\RS$.  
Define a nontrivial unitary character $\bdl{\psi}$ of $\rk\bsl \BA$ by the composition 
\[
\rk\bsl \BA \xrightarrow{{\rm Tr}_{\rk/\BQ}} \BQ \bsl \BA_\BQ \to \BQ\bsl \BA_\BQ / \widehat{\BZ} = \BR/\BZ \xrightarrow{\psi_\BR} \BC^\times,
\]
where $\BA_\BQ$ is the adele ring of $\BQ$, $\widehat\BZ$ is the profinite completion of $\BZ$ and $\psi_\BR(x) = e^{2 \pi {\rm i} x}$, $x\in \BR$. 
Similar to \eqref{varphi2n} in the local case, for a Hecke character $\bdl{\eta}: \rk^\times\bsl \BA^\times\to\BC^\times$, we have an automorphic character $\bdl{\eta}\otimes \bdl{\psi}$ of $\RS(\BA)$. 

Let $\Pi$ be an irreducible cuspidal automorphic representation of $\RG(\BA)$ with central character $\bdl{\eta}^n$. We may identify $\Pi$ with its realization in the space of smooth cuspidal automorphic forms on $\RG(\BA)$,
because of the Multiplicity One Theorem on the cuspidal spectrum of general linear groups (\cite[Theorem 5.5]{Sl74} and \cite[Page 209]{PS79}).

Unless otherwise specified, every discrete group carries the counting measure.  Fix the Haar measure on $\RU(\BA)$ such that the quotient  measure on $\RU(\rk) \bsl \RU(\BA)$ has total volume $1$. Fix once for all a Haar measure 
on $\BR^\times_+ \bsl \GL_n^\dag(\BA)$, where $\BR^\times_+$ is viewed as a subgroup of $\RG(\rk_\infty)\subset \RG(\BA)$ via the diagonal embedding. This yields a Haar measure on the unimodular group 
$\BR^\times_+\bsl \RS(\BA)$, and 
thereby fixes a Shalika functional 
\be \label{globalSha}
\lambda_\BA: \Pi\otimes (\bdl{\eta}\otimes \bdl{\psi})^{-1} \to \BC, \quad \phi \mapsto \int_{(\RS(\rk)\BR^\times_+)\bsl \RS(\BA)} \phi(g) (\bdl{\eta}\otimes \bdl{\psi})^{-1}(g)\od\! g. 
\ee

   \begin{dfnl}\label{def:Shaperiod}
           We say that an irreducible cuspidal automorphic representation $\Pi$ of $\RG(\BA)$ has a nonzero $(\bdl{\eta}, \bdl{\psi})$-Shalika period if its central character equals $\bdl{\eta}^n$, and 
           the Shalika functional $\lambda_\BA$ in \eqref{globalSha} is not identically zero. 
   \end{dfnl}

   The following theorem gives some characterizations for irreducible cuspidal automorphic representations of $\RG(\BA)$ to be of symplectic type, which is due to a collection of works of many people
(\cite[Section 8, Theorem 1]{JS90}, \cite[Theorem 1.1]{AsgariShahidiGenericTransferforGSpin} and \cite[Theorem A]{HundleySayagDescentConstructionforGSpin}, for instance). We refer to \cite[Section 5]{GR13} for
more detailed discussion.

     \begin{thml}\label{prop: strong transfer}
          Let $\Pi$ be an irreducible cuspidal automorphic representation of $\RG(\BA)$ with central character $\bdl{\eta}^n$, where $\bdl{\eta}: \rk^\times \bs \BA^\times\to \BC^\times$ is a Hecke character. Then the following statements are equivalent.
          \begin{enumerate}
            \item The representation  $\Pi$ is the functorial transfer of an irreducible generic cuspidal automorphic representation $\Sigma$ of $\mathrm{GSpin}_{2n+1}(\BA)$ with central character $\omega_\Sigma = \bdl{\eta}$.
            \item The representation $\Pi$ has a nonzero $(\bdl{\eta}, \bdl{\psi})$-Shalika period.
            \item The exterior square $L$-function $\oL(s, \Pi, \wedge^2\otimes\bdl{\eta}^{-1})$ has a pole at $s=1$.
          \end{enumerate}
          Moreover, if any of the above conditions holds, then the transfer from $\mathrm{GSpin}_{2n+1}(\BA)$ is strong at all archimedean local places, in the sense that it respects local $L$-parameters at
each archimedean local place.
     \end{thml}
     
     \begin{remarkl}
     In the literature, the third statement in the above Proposition is often stated as ``\,the partial $L$-function $\oL^S(s, \Pi, \wedge^2\otimes\bdl{\eta}^{-1})$ has a pole at $s=1$".  With the classification of the unitary dual of general linear group over local fields of characteristic zero (\cite{V86} and \cite{T86}, also see \cite[Theorem 7.2]{BadulescuRenardUnitaryDualJLCorrespondence}), it is not hard to show that the partial exterior square $\oL^S(s, \Pi, \wedge^2\otimes\bdl{\eta}^{-1})$ has a pole at $s=1$ if and only if so does the complete $L$-function $\oL(s, \Pi, \wedge^2\otimes\bdl{\eta}^{-1})$.
  \end{remarkl}

  Note that if any of the three  statements in Theorem \ref{prop: strong transfer} holds, one must have that $\Pi^\vee \cong\Pi\otimes\bdl{\eta}^{-1}.$   It follows that if $\Pi = \widehat \otimes_v'\Pi_v$ is of symplectic type, then for every $v\mid\infty$, $\Pi_v$ is a Casselman-Wallach representation of $\RG(\rk_v)$ of symplectic type in the following sense.
  
   \begin{dfnl}\label{def: local repn of symplectic type}
     Let $\bk$ be an archimedean local field. An irreducible Casselman-Wallach representation $\pi$ of $\RG(\bk)$ is said to be of symplectic type if the image of  its $L$-parameter $\CW_\bk\rightarrow \GL_{2n}(\BC)$ is contained in a conjugate of $\mathrm{GSp}_{2n}(\BC)$.
    \end{dfnl}

For the  archimedean local Langlands correspondence, see \cite{L89} (for all real reductive groups) or \cite[Appendix]{J09}, \cite[Sections 3 and 4]{Kn94} (for the general linear groups).

\subsection{Preliminaries on cohomology groups} \label{sec:PCG}

As in the Introduction, let $\Pi = \Pi_f\otimes \Pi_\infty$ be a  regular algebraic irreducible cuspidal automorphic representation 
of $\RG(\BA)$ that has a nonzero $(\bdl{\eta},\bdl{\psi})$-Shalika period and has a coefficient system $F_\mu$. Since the archimedean local components of $\Pi$ are essentially tempered and of symplectic type, the highest weight $\mu\in(\BZ^{2n})^{\CE_\rk}$ is pure and of symplectic type (see \eqref{pure1} and \eqref{pure2}).

We introduce
\[
\CX_\RG := (\RG(\rk) \BR^\times_+)\bsl \RG(\BA) / K_\infty^0\quad \textrm{and}\quad \CX_\RH:= (\RH(\rk)\BR^\times_+)\bsl \RH(\BA) / C_\infty^0,
\]
where $K_\infty$ and $C_\infty$ are the standard maximal compact 
subgroups of $\RG_\infty:=\RG(\rk_\infty)$ and $\RH_\infty:=\RH(\rk_\infty)$ respectively. Then the natural inclusion 
$\imath: \CX_\RH \hookrightarrow \CX_\RG
$
is a proper map. 
Define a real vector space $\frak q_\infty: = (\frak c_\infty\oplus\BR)\bsl \frak h_\infty$, where 
as usual gothic letters denote the Lie algebras of the corresponding real Lie groups, and $\BR$ indicates the Lie algebra of $\BR^\times_+$. Put 
$
d_\infty: = \dim \frak q_\infty = r-1+\sum_{v\mid \infty} d_{\rk_v},
$
where $d_{\rk_v}$ is as in \eqref{dk} and $r$ is the number of archimedean places of $\rk$.  As in \cite{Cl90}, we have the canonical isomorphism 
\be \label{betti}
\iota_{\rm can}: \oH^{d_\infty}_{\rm ct}(\BR^\times_+\bsl \RG_\infty^0; \Pi_\infty \otimes F_\mu^\vee) \otimes \Pi_f   \cong  \oH^{d_\infty}_{\rm ct}(\BR^\times_+\bsl \RG_\infty^0; \Pi \otimes F_\mu^\vee)   \hookrightarrow \oH^{d_\infty}_c(\CX_\RG, F_\mu^\vee),
\ee
where $\oH^*_c$ denotes the Betti cohomology with compact support.  As is known (see e.g. \cite[Section 6.3]{LLS24}), \eqref{betti} is $\RG^\natural$-equivariant, where 
\[
\RG^\natural: = \RG(\BA_f)\times \pi_0(\RG_\infty) =\RG(\BA_f) \times \pi_0(\rk_\infty^\times).
\]

Set $\GL_n':=\GL_n\times \{1\} \subset \RH$. Write
\[
\frak m = \frak m_f \otimes \frak m_\infty\quad \textrm{and}\quad \frak m'=\frak m_f' \otimes \frak m_\infty',
\]
where $\frak m$, $\frak m_f$ and $\frak m_\infty$ are respectively the one-dimensional spaces of invariant measures on $\BR^\times_+\bs\RH(\BA)$, $\RH(\BA_f)$ and $\BR^\times_+\bs\RH(\rk_\infty)$,  and  $\frak m'$, $\frak m_f'$ and $\frak m_\infty'$ are respectively 
the one-dimensional spaces of invariant measures on $\GL_n'(\BA)$, $\GL_n'(\BA_f)$ and $\GL_n'(\rk_\infty)$.  We fix Haar measures on $\GL_n^\dag(\BA_f)$ and $\BR^\times_+\bs\GL_n^\dag(\rk_\infty)$ such that their product 
is equal to the fixed Haar measure on $\BR^\times_+\bsl \GL_n^\dag(\BA)$
in Section \ref{sec:CRST}. This enables us to identify 
$\frak m, \frak m_f$ and  $\frak m_\infty$ with $\frak m', \frak m'_f$ and $\frak m'_\infty$, respectively. 

 Let $\omega_\infty:= (\wedge^{d_\infty}\frak q_\infty)\otimes_\BR \BC$, and  let $\frak O_\infty$ be the complex orientation space of $\wedge^{d_\infty}\frak q_\infty$ as in \cite[(25)]{LLS24}, which has a natural $\BQ$-structure.
Similar to \cite[Section 3.1]{LLS24}, we have an identification: 
$
\frak m_\infty = \omega_\infty^* \otimes \frak O_\infty,
$
where a superscript $*$ indicates the linear dual.  Then we have that
\be \label{cohm}
\oH^{d_\infty}_{\rm ct}(\BR^\times_+\bsl \RH_\infty^0; \frak m_\infty^*) \otimes \frak O_\infty = \oH^{d_\infty}_{\rm ct}(\BR^\times_+\bsl \RH_\infty^0; \omega_\infty) = \BC,
\ee
where we use $(\frak h_\infty, \BR^\times_+ C_\infty^0)$-cohomology in the last equality.

Recall that we have an algebraic Hecke character $\chi$ of $\rk^\times \bsl \BA^\times$, with coefficient system $\chi_\natural$. Define the character
$
\xi_{\bdl{\eta}, \chi} : = \chi  \boxtimes (\chi^{-1} \bdl{\eta}^{-1})
$
of $\RH(\BA)$, which has the factorization 
$
\xi_{\bdl{\eta}, \chi} = \xi_{\bdl{\eta}_f, \chi_f} \otimes \xi_{\bdl{\eta}_\infty, \chi_\infty}.
$
Recall the character $\xi_{\mu, \chi_\natural}$ of $\RH(\rk\otimes_\BQ\BC)$ given by \eqref{char:xi}, which is the coefficient system of $\xi_{\bdl{\eta}, \chi}$. 
To ease the notation, write 
\be \label{shortH}
\oH(\Pi):= \oH^{d_\infty}_{\rm ct}(\BR^\times_+\bsl \RG_\infty^0;  \Pi \otimes F_\mu^\vee)\quad \textrm{and}\quad 
\oH(\Pi_\infty):= \oH^{d_\infty}_{\rm ct}(\BR^\times_+\bsl \RG_\infty^0;  \Pi_\infty \otimes F_\mu^\vee).
\ee
Likewise, write 
\[
\begin{aligned}
\oH(\xi_{\bdl{\eta}, \chi}) := \oH^0_{\rm ct}(\BR^\times_+\bsl \RH_\infty^0;\xi_{\bdl{\eta}, \chi} \otimes \xi_{\mu, \chi_\natural}^\vee)\ {\rm and}\  
\oH(\xi_{\bdl{\eta}_\infty, \chi_\infty}) := \oH^0_{\rm ct}(\BR^\times_+\bsl \RH_\infty^0;\xi_{\bdl{\eta}_\infty, \chi_\infty} \otimes \xi_{\mu, \chi_\natural}^\vee).
\end{aligned}
\]
Without further explanation, similar notation applies to the $\sigma$-twist with $\sigma\in\Aut(\BC)$. 

\subsection{Modular symbols and a commutative diagram}\label{sec:MSCD} 
According to \eqref{globalSha}, we have fixed a global Shalika functional $\lambda_\BA$ of $\Pi$, which is nonzero by assumption. By the uniqueness of local Shalika functionals (Lemma \ref{uniquesh}), we have factorizations \[
\Pi=\Pi_f\otimes \Pi_\infty=\widehat \otimes'_v \Pi_v\quad\textrm{and}\quad  \lambda_\BA =\lambda_f \otimes \lambda_\infty=\otimes_v \lambda_v
\]
such that for all but finitely many nonarchimedean places $v$ of $\rk$, $\Pi_v$ is spherical and $\lambda_v$ takes value $1$ at the fixed spherical vector for the definition of restricted tensor product. Here $\lambda_f$, $\lambda_\infty$, and $\lambda_v$ are nonzero Shalika functionals on $\Pi_f$, $\Pi_\infty$, and $\Pi_v$, respectively. 

\subsubsection{Global modular symbol} 
When $\chi_\natural$ is $F_\mu$-balanced, fix a generator 
\[
\lambda_{F_\mu, \chi_\natural} \in \Hom_{\RH(\rk\otimes_\BQ\BC)}( F_\mu^\vee \otimes \xi_{\mu,\chi_\natural}^\vee,\BC)
\]
that is similar to Lemma \ref{balanced-vec}.
Recall the space of measures $\frak m_f$ on $\RH(\BA_f)$ and the orientation space $\frak O_\infty$. Put
$\frak m^\natural : = \frak m_f \otimes \frak O_\infty.
$
In the notation of \eqref{shortH}, we have  the global modular symbol 
\begin{eqnarray}
\nonumber \wp \colon  \oH(\Pi) \otimes \oH(\xi_{\bdl{\eta},\chi})\otimes \frak m^\natural  
& \hookrightarrow &  \oH^{d_\infty}_c(\CX_\RG, F_\mu^\vee) \otimes \oH^0(\CX_\RH, \xi_{\mu, \chi_\natural}^\vee)\otimes \frak m^\natural \\
& \xrightarrow{\imath^*} &  \oH^{d_\infty}_c(\CX_\RH, F_\mu^\vee) \otimes \oH^0(\CX_\RH, \xi_{\mu, \chi_\natural}^\vee) \otimes \frak m^\natural  \\
\nonumber  & \xrightarrow{\lambda_{F_\mu, \chi_\natural}} & \oH^{d_\infty}_c(\CX_\RH, \BC) \otimes \frak m^\natural \\
\nonumber  & \xrightarrow{\int_{\CX_\RH}}&\BC,
\end{eqnarray}
where $\int_{\CX_\RH}$ is the pairing with the fundamental class (see e.g. \cite[Section 5.1]{M} for details).

\subsubsection{Archimedean modular symbol} Similar to the local case (see \eqref{nonarchiFJ}), using the chosen $\lambda_\infty$ at the beginning of Section \ref{sec:MSCD}, we have the normalized 
Friedberg-Jacquet periods 
\[
\oZ^\circ_{\rm FJ}(\frac{1}{2}, \cdot, \chi_\infty) = \frac{\oZ_{\rm FJ}(\frac{1}{2}, \cdot, \chi_\infty)}{\oL(\frac{1}{2}, \Pi_\infty\otimes\chi_\infty)}: \Pi_\infty\otimes \xi_{\bdl{\eta}_\infty, \chi_\infty} \to \frak m_\infty'^* = \frak m_\infty^*,
\]
where we have identified $\frak m_\infty$ with $\frak m'_\infty$ as in Section \ref{sec:PCG}.  As above assume that
$\chi_\natural$ is $F_\mu$-balanced. Introduce the normalized archimedean modular symbol
\begin{eqnarray} \label{AMS}
\wp_{\infty}^\circ \colon \oH(\Pi_\infty) \otimes \oH(\xi_{\bdl{\eta}_\infty,\chi_\infty})\otimes \frak O_\infty  \to   \oH^{d_\infty}_{\rm ct}(\BR^\times_+ \bsl \RH_\infty^0; \frak{m}_\infty^*)\otimes \frak O_\infty =  \BC,
\end{eqnarray}
where the first arrow is induced by restriction and the functional 
\[
\Omega_{\mu, \chi_\natural} \cdot \oZ^\circ_{\rm FJ}(\frac{1}{2}, \cdot, \chi_\infty )\otimes \lambda_{F_\mu, \chi_\natural},
\]
and the last equality is \eqref{cohm}.

We  mention that the above formulation is more canonical, while in the archimedean modular symbol  given by \eqref{archMS} we have fixed the measure on $\GL_n(\bk)$ for simplicity. 

\subsubsection{Non-archimedean Friedberg-Jacquet period} 
Recall the factorization $\lambda_f = \otimes_{v\nmid\infty}\lambda_v$ as given in the beginning of Section \ref{sec:MSCD} and write $\frak m_f= \frak m_f' = \otimes_{v\nmid \infty}\frak m_v'$. We introduce the normalized non-archimedean Friedberg-Jacquet period
\be \label{NAMS}
\wp^\circ_{f}:= \otimes_{v\nmid \infty} \wp^\circ_{v}: \Pi_f \otimes \xi_{\bdl{\eta}_f, \chi_f} \otimes \frak m_f\to \BC, 
\ee
where  $\wp^\circ_{v}: \Pi_v\otimes \xi_{\bdl{\eta}_v, \chi_v}\otimes \frak m_v'\to\BC$ is given by
\[
\wp^\circ_{v}:= \CG(\chi_v)^n \cdot \oZ^\circ_{\rm FJ}(\frac{1}{2}, \cdot, \chi_v). 
\]
In the above,  $\CG(\chi_v)$ is the local Gauss sum defined using $\bdl{\psi}_v$ as in \eqref{localgauss}, and $\oZ^\circ_{\rm FJ}(\frac{1}{2},\cdot,\chi_v)$ is given by \eqref{nonarchiFJ}. Then we define the global Gauss sum
\be \label{globalgauss}
\CG(\chi) =\CG(\chi_f) := \prod_{v\nmid\infty}\CG(\chi_v).
\ee

By the unramified calculation (\cite[Proposition 3.2]{FJ93}), $\wp_v^\circ(e_v\otimes 1\otimes \od\!g_v)=1$ for almost all places $v\nmid\infty$ such that 
$\Pi_v$ and $\chi_v$ are unramified, $\bdl{\psi}_v$ has conductor $\CO_{\rk_v}$, $e_v\in \Pi_v$ is the spherical vector used in the definition of the restricted tensor product  $\Pi_f$,
and $\od\!g_v\in \frak m_v'$ is the normalized Haar measure such that $\GL_n(\CO_{\rk_v})$ has volume 1. It follows that 
\eqref{NAMS} is well-defined. 

\subsubsection{A commutative diagram}  
The following is a consequence of \cite[Proposition 2.3]{FJ93}, which relates the local Friedberg-Jacquet periods and  the global period
\[
\oZ_{\rm FJ}(\frac{1}{2},\cdot, \chi):\Pi \otimes \xi_{\bdl{\eta}, \chi}\to \BC, \quad  \phi\otimes 1 \mapsto  \int_{( \RH(\rk)\BR^\times_+)\bsl \RH(\BA)} \phi(h)  \xi_{\bdl{\eta}, \chi}(h)\od\! h.
\]

\begin{prpl} The following diagram is commutative:
\be \label{FJCD}
\begin{CD}
\oH(\Pi_\infty) \otimes \oH(\xi_{\bdl{\eta}_\infty, \chi_\infty}) \otimes \frak O_\infty \otimes  \Pi_f \otimes \xi_{\bdl{\eta}_f, \chi_f}\otimes \frak m_f  @> \wp^\circ_{\infty}\otimes \wp^\circ_{f}>>  \BC  \\
 @V \iota_{\rm can} VV  @VV \frac{\oL(\frac{1}{2}, \Pi\otimes \chi)}{\Omega_{\mu, \chi_\natural} \cdot \CG(\chi)^n} V \\
\oH(\Pi)\otimes \oH(\xi_{\bdl{\eta},\chi})\otimes \frak m^\natural @> \wp >>  \BC,
\end{CD}
\ee
where the left vertical arrow is induced by \eqref{betti}.
\end{prpl}

\section{Shalika Periods and the Blasius-Deligne Conjecture} \label{sec:SPBD}

In this section we will  define the family of Shalika periods in the cohomological setting and prove Theorem \ref{BDconj}.

\subsection{The kernels of modular symbols}
Recall that  $\pi_0(\RG_\infty) = \pi_0(\rk_\infty^\times)$ naturally acts on $\oH(\Pi)$ and $\oH(\Pi_\infty)$, and it is clear that the action of 
$\pi_0(\RH_\infty)$ on $\frak O_\infty$ factors through the natural surjective map 
\be \label{pi0map}
\pi_0(\RH_\infty)\to \pi_0(\RG_\infty) = \pi_0(\rk_\infty^\times).
\ee
Write 
\[
\CH(\Pi):=\RH(\Pi)\otimes \frak O_\infty,\qquad \CH(\Pi_\infty) := \RH(\Pi_\infty)\otimes \frak O_\infty,
\]
viewed as representations of $\pi_0(\rk_\infty^\times)$ via \eqref{pi0map}.
For every  $\varepsilon' \in \widehat{\pi_0(\rk_\infty^\times)}$, we write their $\varepsilon'$-isotypic components as  $\CH(\Pi)[\varepsilon']$ and $\CH(\Pi_\infty)[\varepsilon']$, respectively.  

By Lemma \ref{etamu}, it is easy to see that $\oH(\xi_{\bdl{\eta}_\infty, \chi_\infty})$ descends to the character $\chi^\natural$ of $\pi_0(\rk_\infty^\times)$ via \eqref{pi0map}, and 
from now on we make the identification
\be \label{iden}
\oH(\xi_{\bdl{\eta}_\infty, \chi_\infty}) = \varepsilon:=  \chi^\natural.
\ee
Recall the modular symbol 
\[
\wp_{\infty}^\circ: \CH(\Pi_\infty) \otimes \varepsilon = \oH(\Pi_\infty) \otimes \oH(\xi_{\bdl{\eta}_\infty, \chi_\infty}) \otimes \frak O_\infty \xrightarrow{\eqref{AMS}}\BC.
\]
It follows from  Theorem \ref{thm:NV} that for every $\varepsilon'\in \widehat{\pi_0(\rk_\infty^\times)}$, the map 
\[
\CH(\Pi_\infty)[\varepsilon']\to \BC, \quad \kappa\mapsto \wp^\circ_{\infty}(\kappa\otimes 1)
\]
is nonzero if and only if 
$\varepsilon' = \varepsilon$. Define a map 
\be \label{wpeps}
\wp^\circ_{\varepsilon}: \CH(\Pi_\infty)[\varepsilon]\to \BC, \qquad \kappa \mapsto \wp^\circ_{\infty}(\kappa\otimes 1),
\ee  
whose kernel $\Ker\,  \wp^\circ_{\varepsilon}$ is a subspace of $\CH(\Pi_\infty)[\varepsilon] $ of codimension one.

Recall that 
$\Pi_\infty \cong  \pi_\mu := \widehat\otimes_{v\mid \infty} \pi_{\mu_v},
$
where $\mu_v:= \{\mu^{\iota}\}_{\iota\in \CE_{\rk_v}}$, and we have a Shalika functional $\lambda_\infty$ on $\Pi_\infty$ from the factorization 
$\lambda_\BA =\lambda_f\otimes \lambda_\infty$ as in Section \ref{sec:MSCD}.  
Let $\pi_0 \in \Irr(\RG_\infty)$ be the specialization of $\pi_\mu$ at $\mu=0$, and we 
fix a nonzero Shalika functional $\lambda_{0,\infty}$ on $\pi_0$.  Put $I_\mu := \widehat\otimes_{v\mid\infty}I_{\mu_v}$ and 
$\lambda_\mu:=\otimes_{v\mid\infty} \lambda_{\mu_v}$, where $\lambda_{\mu_v}$ is the Shalika functional on $I_{\mu_v}$ as in Proposition \ref{prop:Sha}.
There is a unique map $p_\mu \in \Hom_{\RG_\infty}(I_\mu, \Pi_\infty)$ such that $\lambda_\infty \circ p_\mu = \lambda_\mu$, and we define $p_0\in \Hom_{\RG_\infty} (I_0, \pi_0)$ similarly. Then similar 
to \eqref{CDtrans}, there is a unique map $\jmath_\mu\in \Hom_{\RG_\infty}( \pi_0, \Pi_\infty \otimes F_\mu^\vee)$  such that the diagram 
\[
\xymatrix{
I_0 \ar@{^(->}[r]^(0.4){\imath_\mu} \ar@{>>}[d]_{p_0} & I_\mu\otimes F_\mu^\vee \ar@{>>}[d]^{p_\mu\otimes {\rm id}} \\
\pi_0 \ar@{^(->}[r]^(0.4){\jmath_\mu} & \Pi_\infty \otimes F_\mu^\vee
}
\]
commutes, where $\imath_\mu:=\otimes_{v\mid\infty}\imath_{\mu_v}$. The map $\jmath_\mu$ induces an isomorphism 
\be \label{cohtrans}
\jmath_\mu: \CH(\pi_0) :=\oH^{d_\infty}_{\rm ct}(\BR^\times_+\bsl \RG_\infty^0; \pi_0)\otimes \frak O_\infty \xrightarrow{\sim} \CH(\Pi_\infty). 
\ee 
Specializing at $\mu =0$ and $\chi_\infty = \varepsilon$  in \eqref{wpeps}, we obtain a map 
\[
\wp^\circ_{0, \varepsilon}: \CH(\pi_0)[\varepsilon] \to \BC. 
\]

\begin{leml} \label{lem:ker}
The map $\wp^\circ_\varepsilon$ in \eqref{wpeps} 
depends only on $\varepsilon$, but not on the character $\chi_\infty$ with $ \chi^\natural = \varepsilon$. 
\end{leml}

\begin{proof}
To prove the lemma, it suffices to show that the diagram 
\be \label{MSCD}
\begin{CD}
\CH(\Pi_\infty)[\varepsilon] @> \wp^\circ_{\varepsilon} >> \BC \\
@A \jmath_\mu AA  @| \\
\CH(\pi_0)[\varepsilon] @> \wp^\circ_{0,\varepsilon} >> \BC
\end{CD}
\ee
commutes. We have Lie group decompositions
\[
       \BR_+^\times\bs \RG_\infty = \left(\prod_{v\mid\infty} \BR_+^\times\bs \RG(\rk_v)\right) \times A_\RG' \quad\textrm{and}\quad \BR_+^\times\bs \RH_\infty =\left(\prod_{v\mid\infty}  \BR_+^\times\bs \RH(\rk_v) \right)\times A_\RG',
\]
   where $A_\RG':=\BR^\times_+\backslash (\BR^\times_+)^r$ and $r$ is the number of archimedean places of $\rk$.
 By the K\"{u}nneth formula, we have
$
          \CH(\Pi_\infty) =\CH'\oplus \CH'',
$
                         where
                         \[
                         \begin{aligned}
                          & \CH':= \left(\otimes_{v\mid \infty} \RH_\mathrm{ct}^{d_{\rk_v}}(\BR_+^\times\bs \RG(\rk_v)^0; \Pi_v\otimes F_{\mu_v}^\vee)\right)\otimes \RH_\mathrm{ct}^{r-1}(A_\RG'; \BC)\otimes\frak O_\infty, \\
                    & \CH'':= \bigoplus_{ \textrm{$a_v\neq d_{\rk_v}$ for some $v\mid \infty$}}\left(\otimes_{v\mid \infty} \RH_\mathrm{ct}^{a_v}(\BR_+^\times\bs \RG(\rk_v)^0; \Pi_v\otimes F_{\mu_v}^\vee)\right)\otimes\RH_\mathrm{ct}^{d_\infty-\sum_{v\mid \infty } a_v}(A_\RG'; \BC)\otimes \frak O_\infty.
               \end{aligned}
               \]
           It is easy to see that $\wp_\infty^\circ$ vanishes on $\CH''\otimes\varepsilon$, and $\wp_\infty^\circ|_{\CH'\otimes\varepsilon}$ equals the composition of
               \begin{eqnarray*}
               \CH'\otimes\varepsilon &=&\left(\otimes_{v\mid \infty}  \RH_\mathrm{ct}^{d_{\rk_v}}(\BR_+^\times\bs \RG(\rk_v)^0; \Pi_v\otimes F_{\mu_v}^\vee)\right)\otimes \varepsilon\otimes\frak O_\infty\otimes \RH_\mathrm{ct}^{r-1}(A_\RG'; \BC)\\
               &\xrightarrow{\left(\otimes_{v\mid \infty} \Omega_{\mu_v,\chi_{v\natural}}\cdot\wp_{\mu_v,\chi_v}\right)\otimes 1} &\left(\otimes_{v\mid \infty} \RH_\mathrm{ct}^{d_{\rk_v}}(\BR_+^\times\bs \RH(\rk_v)^0; \BC)\right)\otimes \frak O_\infty\otimes \RH_\mathrm{ct}^{r-1}(A_\RG'; \BC)\\
               &= & \otimes_{v\mid\infty} \frak m_{\rk_v}' \xrightarrow{\sim}\BC,
                              \end{eqnarray*}
where for the last isomorphism we identify the fixed Haar measure in $\frak m_{\rk_v}'$ as in Section \ref{sec1.2.2} with $1\in\BC$. 
Thus the diagram \eqref{MSCD} commutes by Theorem \ref{APR}.
\end{proof}

Using the Shalika functional $\lambda_f$, we embed $\Pi_f$ into
$\Ind^{\RG(\BA_f)}_{\RS(\BA_f)} (\bdl{\eta}_f\otimes \bdl{\psi}_f)$.  Let $\sigma\in\Aut(\BC)$.
Similar to \eqref{sigmaf000},  we have a $\sigma$-linear isomorphism 
\[
\sigma:\Ind^{\RG(\BA_f)}_{\RS(\BA_f)}  (\bdl{\eta}_f\otimes \bdl{\psi}_f)
\to \Ind^{\RG(\BA_f)}_{\RS(\BA_f)} ({}^\sigma\!\bdl{\eta}_f\otimes \bdl{\psi}_f),
\]
which by Lemma \ref{rspherical} restricts to a $\sigma$-linear isomorphism 
$\sigma: \Pi_f \to {}^\sigma\!\Pi_f$. Applying Theorem \ref{thm:NAPR} for all $v\nmid\infty$, we have the non-archimedean period relation 
\be\label{NAPR}
\begin{CD}
\Pi_f \otimes \xi_{\bdl{\eta}_f,\chi_f}  \otimes \frak m_f @> \wp^\circ_{f} >> \BC \\
@V\sigma VV @VV \sigma V \\
{}^\sigma\!\Pi_f \otimes {}^\sigma\!\xi_{\bdl{\eta}_f,\chi_f}  \otimes \frak m_f  @> {}^\sigma\!\wp^\circ_{f} >> \BC,
\end{CD}
\ee
where $\wp^\circ_{f}$ and ${}^\sigma\!\wp^\circ_{f}$ are given by \eqref{NAMS}.

We also have a $\sigma$-linear isomorphism on the Betti cohomology 
\be \label{twist-betti}
\sigma: \oH^{d_\infty}_c(\CX_\RG, F_\mu^\vee) \to  \oH^{d_\infty}_c(\CX_\RG, {}^\sigma\!F_\mu^\vee),
\ee
which via \eqref{betti} induces a $\sigma$-linear isomorphism 
$
\sigma: \CH(\Pi) \to \CH({}^\sigma\!\Pi). 
$
Since  \eqref{twist-betti} intertwines the actions of $\pi_0(\rk_\infty^\times)$, we have a further restriction (\cf \cite[Proposition 6.2]{LLS24}) $\sigma: \CH(\Pi)[\varepsilon] \to \CH({}^\sigma\!\Pi)[\varepsilon]$. 
This induces a $\sigma$-linear isomorphism 
\be \label{archi-act}
\sigma: \CH(\Pi_\infty)[\varepsilon] \to \CH({}^\sigma\!\Pi_\infty)[\varepsilon]
\ee
making the following diagram commutative:
\be \label{loc-globCD}
\begin{CD}
\CH(\Pi_\infty)[\varepsilon] \otimes \Pi_f  @> \sigma >>  \CH({}^\sigma\!\Pi_\infty)[\varepsilon] \otimes {}^\sigma\!\Pi_f  \\
@V \iota_{\rm can} VV @VV \iota_{\rm can} V \\
\CH(\Pi)[\varepsilon] @> \sigma >> \CH({}^\sigma\!\Pi)[\varepsilon].
\end{CD}
\ee

\begin{leml}\label{ri}
    The fixed point set $\left(\CH(\Pi_\infty)[\varepsilon]\right)^{\Aut(\BC/\BQ(\Pi,\bdl{\eta}))}$ is a $\BQ(\Pi,\bdl{\eta})$-rational form of $\CH(\Pi_\infty)[\varepsilon]$. 
\end{leml}
\begin{proof}
 We have the obvious identification
 \[
 \CH(\Pi_\infty)[\varepsilon]=\Hom_{\RG(\BA_f)}(\Pi_f,  \CH(\Pi)[\varepsilon]),
 \]
 which respects the actions of $\Aut(\BC/\BQ(\Pi,\bdl{\eta}))$. Proposition \ref{rationalsh} and Lemma \ref{rspherical} imply that $\Pi_f^{\Aut(\BC/\BQ(\Pi,\bdl{\eta}))}$ is a $\BQ(\Pi,\bdl{\eta})$-rational structure of $\Pi_f$. The universal coefficient theorem for Betti cohomology implies that $\oH^{d_\infty}_c(\CX_\RG, F_\mu^\vee)^{\Aut(\BC/\BQ(\Pi,\bdl{\eta}))}$ is a $\BQ(\Pi,\bdl{\eta})$-rational structure of $\oH^{d_\infty}_c(\CX_\RG, F_\mu^\vee)$. By \cite[Lemma 3.2.1]{Cl90}, this further implies that  $\left(\CH(\Pi)[\varepsilon]\right)^{\Aut(\BC/\BQ(\Pi,\bdl{\eta}))}$ is a $\BQ(\Pi,\bdl{\eta})$-rational structure of $\CH(\Pi)[\varepsilon]$. Pick an open compact subgroup $K_f$ of $\RG(\BA_f)$ such that $\Pi_f^{K_f}\neq \{0\}$. Then the restriction yields  
 a linear embedding 
 \[
 \CH(\Pi_\infty)[\varepsilon]\hookrightarrow\Hom_{K_f}(\Pi_f^{K_f},  \left(\CH(\Pi)[\varepsilon]\right)^{K_f})\subset \Hom_{\BC}(\Pi_f^{K_f},  \left(\CH(\Pi)[\varepsilon]\right)^{K_f}) ,
 \]
 which is obviously $\Aut(\BC/\BQ(\Pi,\bdl{\eta}))$-equivariant. Since $\Pi_f^{K_f}$ is finite-dimensional, the lemma follows by 
 \cite[Lemma 3.2.1]{Cl90}.
\end{proof}

Introduce a family of  representations ${}^\sigma\!\Pi^\natural: = {}^\sigma\!\Pi_f\otimes \varepsilon$ of $\RG^\natural = \RG(\BA_f)\times \pi_0(\rk_\infty^\times)$, 
where $\varepsilon$ is realized as the $\sigma$-twist of  \eqref{iden}, noting that $({}^\sigma\!\chi)^\natural = \chi^\natural$ (\cf \cite[Remark 6.3]{LLS24}). We equip $\frak m_f$ with  a natural $\BQ$-rational structure as in \cite[Section 5.2]{LLS24}.  

For all the modular symbols on the cohomologies of  $\sigma$-twists, we will also put a left superscript $\sigma$ for clarity. By \eqref{FJCD}, \eqref{loc-globCD} and the well-known $\Aut(\BC)$-equivariance of global 
modular symbols, we have a commutative diagram 
\be \label{mainCD}
\begin{tikzcd}
\arrow[ddd, bend right =80, "\sigma" '] \CH(\Pi_\infty)[\varepsilon] \otimes \Pi^\natural \otimes \xi_{\bdl{\eta}_f,\chi_f} \otimes \frak m_f  \arrow[rr,  "\wp^\circ_{\infty} \otimes \wp^\circ_{f}"]  \arrow[d, "\iota_{\rm can}" ']
& & \BC \arrow[d, "\frac{\oL(\frac{1}{2}, \Pi\otimes\chi)}{\Omega_{\mu,\chi_\natural}\cdot\CG(\chi)^n}"]\\
\CH(\Pi)[\varepsilon]\otimes  \oH(\xi_{\bdl{\eta},\chi}) \otimes \frak m_f  \arrow[rr, "\wp"]  \arrow[d, "\sigma" '] & & \BC  \arrow[d, "\sigma" ] \\
\CH({}^\sigma\!\Pi)[\varepsilon] \otimes \oH({}^\sigma\!\xi_{\bdl{\eta},\chi}) \otimes \frak m_f \arrow[rr, " {}^\sigma\!\wp"] && \BC \\
\CH({}^\sigma\!\Pi_\infty)[\varepsilon] \otimes {}^\sigma\!\Pi^\natural \otimes {}^\sigma\!\xi_{\bdl{\eta}_f,\chi_f} \otimes \frak m_f  \arrow[u, "\iota_{\rm can}"]  \arrow[rr, "{}^\sigma\!\wp^\circ_{\infty} \otimes\, {}^\sigma\!\wp^\circ_{f}"]   &&   \BC. \arrow[u, " \frac{\oL(\frac{1}{2}, {}^\sigma\!\Pi\otimes\,{}^\sigma\!\chi)}{\Omega_{\mu,\chi_\natural}\cdot \CG({}^\sigma\!\chi)^n} " ']
\end{tikzcd}
\ee
 Here we have used the facts that ${}^\sigma\!F_\mu = F_{{}^\sigma\!\mu}$ with ${}^\sigma\!\mu:= \{\mu^{\sigma^{-1}\circ \iota}\}_{\iota\in \CE_\rk}$,
and that 
\[
\Omega_{{}^\sigma\!\mu, {}^\sigma\!\chi_\natural} = \Omega_{\mu, \chi_\natural}.
\]

\subsection{Shalika periods and the end of proof}

Now we define the notion of {\it Shalika periods} in the cohomological setting with a fixed factorization $\lambda_\BA =\lambda_f \otimes \lambda_\infty$ as in Section \ref{sec:MSCD}. 


We have a $\Upsilon$-equivariant complex vector bundle 
\be \label{VB}
\bigsqcup_{({}^\sigma\Pi, {}^\sigma\bdl{\eta}),\, \sigma\in\Upsilon}\CH({}^\sigma\Pi_\infty)[\varepsilon] \longrightarrow \set{({}^\sigma\Pi, {}^\sigma\bdl{\eta}) | \sigma\in\Upsilon},
\ee
where the action of $\Upsilon$ on the total space is given by \eqref{archi-act}. 

\begin{leml}\label{lem:equisec}
    There exists a $\Upsilon$-equivariant section 
\be \label{equivsection}
({}^\sigma\Pi, {}^\sigma\bdl{\eta}) \longmapsto {}^\sigma\kappa_\varepsilon \in \left(\CH({}^\sigma\Pi_\infty)[\varepsilon]\right)^{\rm Aut(\BC/\BQ({}^\sigma\Pi,{}^\sigma\bdl{\eta}))}
\ee
of \eqref{VB} such that 
${}^\sigma\kappa_\varepsilon\not\in \Ker \, {}^\sigma\!\wp_\varepsilon^\circ$
for every $\sigma\in\Upsilon$.
\end{leml}

\begin{proof}
Write 
\[
\Upsilon_{\Pi, \bdl{\eta}}:=\{\sigma\in \Upsilon\,\mid\, (^\sigma \Pi,{^\sigma \bdl{\eta}})=(\Pi,\bdl{ \eta})\}\subset \Aut(\BC/\BQ(\Pi,\bdl{\eta}))
\]
so that 
\[
\Upsilon/\Upsilon_{\Pi,\bdl{\eta}}=\set{({}^\sigma\Pi, {}^\sigma\bdl{\eta}) | \sigma\in\Upsilon}
\]
 is a finite set. Lemma \ref{ri} implies that there exists a vector 
 \[
    \kappa_\varepsilon \in  \left(\CH(\Pi_\infty)[\varepsilon]\right)^{\Aut(\BC/\BQ(\Pi,\bdl{\eta}))} \setminus \bigcup_{\sigma \in \Upsilon/\Upsilon_{\Pi, \bdl{\eta}}} \sigma^{-1} \left( \Ker \, {}^{\sigma}\!\wp_\varepsilon^\circ\right).
    \]
  Now we define ${}^\sigma\kappa_\varepsilon:=\sigma(\kappa_\varepsilon)$ for every $\sigma\in\Upsilon/\Upsilon_{\Pi, \bdl{\eta}}$, which gives the required section. 
\end{proof}

Fix a section $({}^\sigma\Pi, {}^\sigma\bdl{\eta}) \to {}^\sigma\kappa_\varepsilon \in \left(\CH({}^\sigma\Pi_\infty)[\varepsilon]\right)^{\rm Aut(\BC/\BQ({}^\sigma\Pi,{}^\sigma\bdl{\eta}))}$ as in Lemma \ref{lem:equisec}. Then we make the following definition. 

\begin{dfnl}[Cohomological Shalika Periods]\label{def:period}
For every $\sigma\in \Upsilon$, define the Shalika period
\[
\Omega_\varepsilon({}^\sigma\!\Pi, {}^\sigma\!\bdl{\eta}):= 
\frac{1}{{}^\sigma\!\wp^\circ_{\varepsilon}({}^\sigma\!\kappa_\varepsilon)}\in \BC^\times.
\]
\end{dfnl}

By the definition and Lemma \ref{lem:ker}, $\Omega_\varepsilon({}^\sigma\!\Pi, {}^\sigma\!\bdl{\eta})$ depends only on ${}^\sigma\!\Pi$, ${}^\sigma\!\bdl{\eta}$, $\varepsilon$, and the choice of the section \eqref{equivsection}.

We finish the proof of the Blasius-Deligne conjecture as follows. 

\begin{proof}(of Theorem \ref{BDconj})
We have a $\Upsilon$-equivariant family of elements
\[
\omega_{{}^\sigma\!\Pi^\natural}\in \Hom_{G(\BA_f)}({}^\sigma\!\Pi_f, \CH({}^\sigma\!\Pi)[\varepsilon]),\quad 
\sigma\in\Upsilon
\]
given by 
\be \label{twist-omega}
\omega_{{}^\sigma\!\Pi^\natural}: {}^\sigma\!\Pi_f \to \CH({}^\sigma\!\Pi)[\varepsilon],\quad {}^\sigma\!\phi_f \mapsto \iota_{\rm can}({}^\sigma\!\kappa_\varepsilon\otimes {}^\sigma\!\phi_f).
\ee
In view of \eqref{mainCD} and \eqref{twist-omega}, we have a commutative diagram 
\[
\begin{CD}
\Pi^\natural \otimes \xi_{\bdl{\eta}_f,\chi_f} \otimes \frak m_f  @> \kappa_\varepsilon \otimes \cdot >> \CH(\Pi_\infty)[\varepsilon] \otimes \Pi^\natural \otimes \xi_{\bdl{\eta}_f,\chi_f} \otimes \frak m_f @> \wp^\circ_{\infty} \otimes \wp^\circ_{f} >> \BC \\
@| @V \iota_{\rm can} VV @VV \frac{\oL(\frac{1}{2}, \Pi\otimes\chi)}{\Omega_{\mu,\chi_\natural}\cdot\CG(\chi)^n} V \\
\Pi^\natural \otimes \xi_{\bdl{\eta}_f,\chi_f} \otimes \frak m_f  @>\omega_{\Pi^\natural}\otimes \iota_{\rm can} >> \CH(\Pi)[\varepsilon]\otimes  \oH(\xi_{\bdl{\eta},\chi}) \otimes \frak m_f  @> \wp >> \BC \\ 
@V\sigma VV @V \sigma VV  @VV\sigma V \\ 
{}^\sigma\!\Pi^\natural \otimes {}^\sigma\!\xi_{\bdl{\eta}_f,\chi_f}  \otimes \frak m_f  @> \omega_{{}^\sigma\!\Pi^\natural}\otimes \iota_{\rm can}   >> \CH({}^\sigma\!\Pi)[\varepsilon] \otimes \oH({}^\sigma\!\xi_{\bdl{\eta},\chi}) \otimes \frak m_f @> {}^\sigma\!\wp>>  \BC  \\
@| @A\iota_{\rm can} AA @AA \frac{\oL(\frac{1}{2}, {}^\sigma\Pi\otimes{}^\sigma\chi)}{\Omega_{\mu,\chi_\natural}\cdot \CG({}^\sigma\!\chi)^n} A \\
{}^\sigma\!\Pi^\natural \otimes {}^\sigma\!\xi_{\bdl{\eta}_f,\chi_f}  \otimes \frak m_f  @> {}^\sigma\!\kappa_\varepsilon\otimes \cdot >>\CH({}^\sigma\!\Pi_\infty)[\varepsilon] \otimes {}^\sigma\!\Pi^\natural \otimes {}^\sigma\!\xi_{\bdl{\eta}_f,\chi_f} \otimes \frak m_f @> {}^\sigma\!\wp^\circ_{\infty} \otimes\,{}^\sigma\!\wp^\circ_{f} >> \BC.
\end{CD}
\]
By chasing the diagram from the top-left corner to the penultimate copy of $\BC$ in the right column, along the boundary  of the diagram in two different directions, we deduce from \eqref{NAPR} and Definition \ref{def:period}  that 
\[
\sigma \left( \frac{\oL(\frac{1}{2}, \Pi\otimes\chi)}{ \Omega_{\mu,\chi_\natural}\cdot\CG(\chi)^n \cdot \Omega_\varepsilon(\Pi, \bdl{\eta} )}\right) =   \frac{\oL(\frac{1}{2}, {}^\sigma\!\Pi\otimes {}^\sigma\!\chi)}{ \Omega_{\mu,\chi_\natural}\cdot\CG({}^\sigma\!\chi)^n \cdot \Omega_\varepsilon({}^\sigma\!\Pi, {}^\sigma\!\bdl{\eta})}.
\]
This proves \eqref{reci}. It is clear that \eqref{alg} follows from applying \eqref{reci} to every $\sigma\in \Aut(\BC/\BQ(\Pi,\bdl{\eta},\chi))$.
\end{proof}

We end this paper with a remark on Definition \ref{def:period}. Consider the following condition.

    \begin{cond}\label{cond}
    There exist $\sigma'\in \Upsilon$ and an $F_\mu$-balanced algebraic Hecke character
$\chi'$ of $\rk^\times\bsl\BA^\times$ such that $\chi'^\natural=\varepsilon$ and $\oL(\frac{1}{2}, {}^{\sigma'}\Pi\otimes {}^{\sigma'}\chi')\neq 0$.
    \end{cond}

\begin{remarkl}
   We expect that 
    for every $\sigma\in\Upsilon$,
    the $\sigma$-linear isomorphism $\sigma: \CH(\Pi_\infty)[\varepsilon] \to \CH({}^\sigma\!\Pi_\infty)[\varepsilon]$ restricts to a $\sigma$-linear isomorphism 
\be \label{stable}
\sigma: \Ker\,\wp^\circ_\varepsilon \xrightarrow{\sim} \Ker\, {}^\sigma\!\wp^\circ_\varepsilon. 
\ee
When this is the case, the class of $\Omega_\varepsilon({}^\sigma\Pi, {}^\sigma\bdl{\eta})$ in $\BC^\times/ \BQ({}^\sigma\Pi, {}^\sigma\bdl{\eta})^\times$ does not depend on the choice 
of the section \eqref{equivsection}.  
   If $\rk$ is totally real, or Condition \ref{cond} holds, then a diagram chasing in \eqref{mainCD} shows that the expectation holds. 

\end{remarkl}

\section*{Acknowledgements}

The research work of the paper was partly carried out during the participation of the authors in one month Research Program: {\it On the Langlands Program: Endoscopy and Beyond} (17 Dec 2018 - 18 Jan 2019), at Institute for Mathematical Sciences, National University of Singapore. We would like to thank the IMS, NUS for invitation and for hospitality. Over the years, we have made a few improvements on the original version based on comments from D. Blasius and L. Clozel, and we would like to thank them for helpful comments on the preliminary version of this paper. We also thank S.-Y. Chen for sending us the preprint \cite{CCDR26}.

D. Jiang is supported in part by 
the Simons Grants: SFI-MPS-SFM-00005659 and 
SFI-MPS-TSM-00013449. 
D. Liu is supported in part by National Key R \& D Program of China No. 2022YFA1005300 and National Natural Science Foundation of China No. 12526208.  B. Sun is supported in part by National Key R \& D Program of China No. 2022YFA1005300  and New Cornerstone Science Foundation. F. Tian
is supported in part by National Key R \& D Program of China No. 2022YFA1005304.

\end{document}